\documentclass[11pt,leqno]{amsart}
\usepackage{amsthm,amsfonts,amssymb,amsmath,oldgerm}
\numberwithin{equation}{section}
\usepackage{fullpage}
\usepackage{color}
\usepackage{calrsfs}
\DeclareMathAlphabet{\pazocal}{OMS}{zplm}{m}{n}

\def\eps{\varepsilon }

\newcommand\R{\mathbb R}

\def\eps{\varepsilon}

\newcommand\br{\begin{remark}}
\newcommand\er{\end{remark}}
\newcommand\bp{\begin{pmatrix}}
\newcommand\ep{\end{pmatrix}}
\newcommand{\be}{\begin{equation}}
\newcommand{\ee}{\end{equation}}
\newcommand\ba{\begin{equation}\begin{aligned}}
\newcommand\ea{\end{aligned}\end{equation}}

\newcommand{\bap}{\begin{app}}
\newcommand{\eap}{\end{app}}
\newcommand{\begs}{\begin{exams}}
\newcommand{\eegs}{\end{exams}}
\newcommand{\beg}{\begin{example}}
\newcommand{\eeg}{\end{exaplem}}
\newcommand{\bpr}{\begin{proposition}}
\newcommand{\epr}{\end{proposition}}
\newcommand{\bt}{\begin{theorem}}
\newcommand{\et}{\end{theorem}}
\newcommand{\bc}{\begin{corollary}}
\newcommand{\ec}{\end{corollary}}
\newcommand{\bl}{\begin{lemma}}
\newcommand{\el}{\end{lemma}}
\newcommand{\bd}{\begin{definition}}
\newcommand{\ed}{\end{definition}}
\newcommand{\brs}{\begin{remarks}}
\newcommand{\ers}{\end{remarks}}

\newcommand{\RR}{{\mathbb R}}

\newcommand{\NN}{{\mathbb N}}

\newcommand{\CC}{{\mathbb C}}

\newcommand{\pa}{{\partial}}
\newcommand{\rmi}{{\mathrm{i}}}
\newcommand{\rmd}{{\mathrm{d}}}
\newcommand{\rms}{{\mathrm{s}}}
\newcommand{\rmu}{{\mathrm{u}}}
\newcommand{\rmc}{{\mathrm{c}}}
\newcommand{\rmr}{{\mathrm{r}}}

\newcommand{\im}{{\rm im }}

\newcommand{\sgn}{\text{\rm sgn}}
\newtheorem{theorem}{Theorem}[section]
\newtheorem{proposition}[theorem]{Proposition}
\newtheorem{corollary}[theorem]{Corollary}
\newtheorem{lemma}[theorem]{Lemma}

\theoremstyle{remark}
\newtheorem{remark}[theorem]{Remark}
\theoremstyle{definition}
\newtheorem{definition}[theorem]{Definition}

\newtheorem{example}[theorem]{Example}

\newcommand\cB{{\mathcal B}}
\newcommand\cD{{\mathcal D}}

\newcommand\cH{{\mathcal H}}
\newcommand\cJ{{\mathcal J}}

\newcommand\cW{{\mathcal W}}

\newcommand\cG{{\mathcal G}}
\newcommand\cK{{\mathcal K}}
\newcommand\cKm{{\mathcal K^{\mathrm{mod}}_{\Gamma,E}}}
\newcommand\cL{{\mathcal L}}
\newcommand\tcL{{\tilde{\mathcal L}}}

\newcommand\cF{{\mathcal F}}

\newcommand\cY{{\mathcal Y}}
\newcommand\cM{{\mathcal M}}

\newcommand\cS{{\mathcal S}}

\newcommand\bH{{\mathbb H}}

\newcommand\bV{{\mathbb V}}

\newcommand\bG{{\mathbb G}}

\newcommand\bX{{\mathbb X}}
\newcommand\bY{{\mathbb Y}}

\newcommand{\tE}{{\widetilde{E}}}
\newcommand{\tR}{{\widetilde{R}}}
\newcommand{\tT}{{\widetilde{T}}}
\newcommand{\tcK}{{\widetilde{\mathcal{K}}}}

\newcommand{\tf}{{\widetilde{f}}}
\newcommand{\tg}{{\widetilde{g}}}
\newcommand{\tth}{{\widetilde{h}}}
\newcommand{\tnu}{{\widetilde{\nu}}}
\newcommand{\ovB}{{\overline{B}}}

\newcommand{\bu}{\mathbf{u}}
\newcommand{\bv}{\mathbf{v}}
\newcommand{\obu}{\overline{\mathbf{u}}}

\newcommand{\bx}{\mathbf{x}}

\newcommand{\dom}{\text{\rm{dom}}}

\newcommand{\beq}{\begin{equation}}
\newcommand{\eeq}{\end{equation}}

\title{Stable manifolds for a class of degenerate evolution equations and exponential decay of kinetic shocks
}

\author{ Alin Pogan}
\address{Miami University, Oxford, OH 45056}
\email{pogana@miamioh.edu}
\thanks{ A. P. research was partially supported under the
Summer Research Grant program, Miami University}
\author{Kevin Zumbrun}
\address{Indiana University, Bloomington, IN 47405}
\email{kzumbrun@indiana.edu}
\thanks{Research of K.Z. was partially supported
under NSF grant no. DMS-0300487}

\begin{document}

\begin{abstract}
We construct stable manifolds for a class of degenerate evolution equations including the
steady Boltzmann equation, establishing in the process exponential decay of associated kinetic shock and boundary layers
to their limiting equilibrium states.
Our analysis is from a classical dynamical systems point of view, but with a number of interesting modifications
to accomodate ill-posedness with respect to the Cauchy problem of the underlying evolution equation.
\end{abstract}

\maketitle

\vspace{0.3cm}
\begin{minipage}[h]{0.48\textwidth}
\begin{center}
Miami University\\
Department of Mathematics\\
301 S. Patterson Ave.\\
Oxford, OH 45056, USA
\end{center}
\end{minipage}
\begin{minipage}[h]{0.48\textwidth}
\begin{center}
Indiana University \\
Department of Mathematics\\
831 E. Third St.\\
Bloomington, IN 47405, USA
\end{center}
\end{minipage}

\vspace{0.3cm}

\tableofcontents

\section{Introduction }\label{s1}

In this paper we study decay rates at infinity of (possibly) large-amplitude relaxation shocks
\be\label{prof}
u(x,t)=u^*(x-st),\quad \lim_{\tau\to\pm\infty}u(\tau)=u^\pm,
\ee
of kinetic-type relaxation systems
\begin{equation}\label{Relax}
A_0 u_t + Au_x = Q(u),
\end{equation}
on a general Hilbert space $\bH$, where $A_0$, $A$ are given (constant) bounded linear operator and $Q$ is a bounded bilinear map.
More generally, we study existence and properties of stable/unstable manifolds for a class of
degenerate evolution equations arising through the study of such profiles.

Making the change of variables $\tau=x-st$ we obtain that the profiles $u^*$ satisfy
the equation
$(A-sA_0)u_\tau = Q(u).$
By frame-indifference, we may without loss of generality take $s=0$, yielding
\begin{equation}\label{nonlinear}
A u_\tau = Q(u).
\end{equation}
We are interested in the singular case, as arises for example for Boltzmann's equation \cite{CN,LiuYu1,LiuYu,MZ2},
that the linear operator $A$ is self-adjoint, one-to-one, but \textit{not invertible}.

This is a crucial point of the current paper, since in the case when the linear operator $A$ has a bounded inverse, one would reduce equation \eqref{nonlinear} to an evolution equation
with bounded linear part, that can be treated similarly as in the case of nonlinear equations on finite dimensional spaces.
The case at hand, when the linear operator $A$ does not have a bounded inverse, requires a different approach, since \eqref{nonlinear} or its linearization along equilibria might not be well-posed; therefore it is not clear if one can use the usual variation of constants formula to look for mild solutions. Rather, we use the frequency domain reformulation of these equations following the approach in \cite{LP2} and \cite{LP3}.

Other cases of non-well-posed equations in the sense that they do not generate an
evolution family either in forward or backward time on the entire space, arise in the study of modulated waves on cylindrical domains (see \cite{PSS,SS1,SS2}), Morse theory (see \cite{AM1,AM2,RobSal}, the theory of PDE Hamiltonian systems
(see \cite{BjornSand}), and the theory of functional-differential equations of mixed type (see \cite{Mallet-Paret}).
The particular form \eqref{nonlinear}, however, in which the singularity arises through the coefficient
of the $\tau$-derivative with other terms bounded, does not seem to have been treated before.
This is the class of  ``degenerate evolution equation'' to which we refer in the title of the paper.

The examples we are interested in arise in certain kinetic and discrete kinetic relaxation approximation models,
{\it in particular, the Boltzman equation}
\be\label{Boltz}
f_t + \xi_1 f_x=Q(f),\quad x\in \R^1, \, \xi\in \R^3,
\ee
$f=f(x,\xi)$ denoting density at spatial point $x$ of particles with velocity $\xi$,
which, after rescaling by $\langle \xi\rangle:=\sqrt{1+|\xi|^2}$, can be put in form \eqref{Relax}, with $A$ is equal to the operator of multiplication by the function $\xi_1/\langle \xi\rangle$ and
$\bH$ an appropriate weighted $L^2$ space in the variable $\xi$.
In the series of papers \cite{MTZ,MZ,MZ2,TZ5},  M\'etivier, Texier and Zumbrun
obtained existence results for a somewhat larger class of models in the case of shocks with small amplitude $\eps:=\|u^+-u^-\|$,
in particular yielding exponential decay rates as $\tau=(x-st) \to \pm \infty$; see also the earlier \cite{CN,LiuYu1}
in the specific case of Boltzmann's equation.
These results were obtained by fixed-point iteration on the whole line, using in an essential way the small-amplitude
assumption to construct initial approximations based on the formal fluid-dynamical approximation.

Here, our interest is in treating {\it large-amplitude} profiles, without a priori information on the shape of the profile,
by {\it dynamical systems techniques} that would apply also in the case of boundary layers where the solution is
not necessarily defined on the whole line.
Our larger goal is to develop dynamical systems tools analogous to those of \cite{GZ,MaZ2,MaZ3,MaZ4,Z2,Z3,Z4,Z5,ZH,ZS},
sufficient to treat 1- and multi-D stability by the techniques of those papers.
See in particular the discussion of \cite[Remark 4.2.1(4), p. 55]{Z4}, proposing a path toward stability of Boltzmann
shock profiles.
For this program, the proof of exponential decay rates and the establishment of a stable manifold theorem
are essential first steps. For a corresponding center manifold theorem, see \cite{PZ}.

\subsection{Assumptions}\label{s:assumptions}
In \cite[Section 4]{MZ2} it is shown that the Boltzman equation with hard-sphere potential can be recast as an equation of form \eqref{Relax} where the linear operator $A$ and the nonlinearity $Q$ satisfy the hypotheses (H1)--(H3) below;
following \cite{MZ2}, we assume these throughout.

\medskip

\noindent{\bf Hypothesis (H1)} The linear operator $A$ is
bounded and self-adjoint on the Hilbert space $\bH$. There exists $\bV$ a \textit{proper}, closed subspace of $\bH$ with $\dim\bV^\perp<\infty$ and $B:\bH\times\bH\to\bV$ is a bilinear, symmetric, continuous map such that $Q(u)=B(u,u)$.

\medskip

To understand the behavior of solutions of \eqref{nonlinear} near the equilibria $u^\pm$ it is important to study the properties of the linearization of \eqref{nonlinear} along $u^\pm$. As pointed out in \cite{MZ2}, in many special cases of \eqref{Relax}, the linear operator $Q'(u^\pm)$ satisfies the Kawashima condition.

\medskip

\noindent{\bf Hypothesis (H2)} We say that a bounded linear operator $T\in\cB(\bH)$ satisfies the Kawashima condition if
\begin{enumerate}
\item[(i)] $T$ is self-adjoint and $\ker T=\bV^\perp$;
\item[(ii)] There exists $\delta>0$ such that $T_{|\bV}\leq -\delta I_{\bV}$;
\item[(iii)] There exists $K\in\cB(\bV)$ skew-symmetric and $\gamma>0$ such that $\mathrm{Re}\,(KA-T)\geq \gamma I$.
\end{enumerate}

In the hypothesis below we summarize the conditions satisfied by the equilibria $u^\pm$.

\medskip

\noindent{\bf Hypothesis (H3)} We assume that the equilibria $u^\pm\in\ker Q$ are such that $Q'(u^\pm)$ satisfies hypothesis (H2) for some $K_\pm\in\cB(\bH)$
skew-symmetric, $\delta_\pm>0$ and $\gamma_\pm>0$.

\medskip

In many interesting examples the linear operator $A$ is an operator of multiplication by a non-zero function on a certain Hilbert space of functions.
One of the goals of this paper is to apply our results to our leitmotif example, the Boltzman equation in the case of nonzero Euler characteristics.
Therefore, to prove the main results in
this paper we
make the additional two hypotheses:

\medskip

\noindent{\bf Hypothesis (H4)} The linear operator $A$ is one-to-one.

\medskip

\noindent{\bf Hypothesis (H5)} The linear operator $P_{\bV^\perp}A_{|\bV^\perp}$ is invertible on the finite dimensional space $\bV^\perp$. Here $P_{\bV^\perp}$ denotes the orthogonal projection onto $\bV^\perp$.

\subsection{Results}\label{s:results}
Our first task is to study the qualitative properties of the linearization of equation \eqref{nonlinear}
about
the equilibria $u^\pm$. Under Hypothesis (H1) and (H3) we prove that the linear operator $\cL^\pm=A\pa_\tau-Q'(u^\pm)$ is not invertible when considered as a linear operator on $L^2(\RR,\bH)$, see Remark~\ref{r2.1} below. In many situations when one is interested in tracing the point spectrum of the linearization along some traveling wave profile, in the case when the linearization is not Fredholm on $L^2(\RR,\bH)$, it is of great interest to study the invertibility properties of the linearization $\cL^\pm$ at the equilibria on weighted spaces $L^2_\eta(\RR,\bH)$ with $\eta\ne0$ small, see e.g. \cite{PS1,PS4}. Making a change of variables, one can readily check that $\cL^\pm$ is invertible on $L^2_\eta(\RR,\bH)$ if and only if the linear operator $\cL^\pm_\eta=A\pa_\tau-Q'(u^\pm)\mp\eta A$ is invertible on $L^2(\RR,\bH)$. Our invertibility result reads as follows:
\begin{theorem}\label{t2.9}
Assume Hypotheses (H1) and (H3). Then, there exists $\eta^*>0$ such that
$\cL^\pm_\eta=A\pa_\tau-Q'(u^\pm)\mp\eta A$ is invertible with bounded inverse on $L^2(\RR,\bH)$ for all $\eta\in(0,\eta^*)$.
\end{theorem}
To prove this theorem we use that the differential operator with constant coefficients $\cL^\pm_\eta$ is similar to an operator of multiplication on $L^2(\RR,\bH)$  by an operator-valued function, $\widehat{\cL^\pm_\eta}(\omega)=2\pi\rmi\omega A-Q'(u^\pm)\mp\eta A$. First, we prove that $\widehat{\cL^\pm_\eta}(\omega)$ is a Fredholm linear operator with index $0$ and empty kernel, which allows us to infer that $\widehat{\cL^\pm_\eta}(\omega)$ is invertible for any $\omega\in\RR$. These steps require two key ingredients: 1) the eigenspaces of the self-adjoint, bounded, linear operator $A$ have a trivial intersection with the orthogonal complement of the image of the bilinear map $Q$; 2) small perturbations of $Q'(u^\pm)$ by terms of the form $sA$ with $s$ small enough, are invertible with bounded inverse on the entire space $\bH$. Finally, we prove that the function $\omega\to\big(\widehat{\cL^\pm_\eta}(\omega)\big)^{-1}$ is bounded on $\RR$, using the Kawashima condition.

Next, we show that equations $Au'=Q'(u^\pm)u$ and $Au'=\big(Q'(u^\pm)\pm\eta A\big)u$ are equivalent to an equation of the form $u'=Su$. In
each
of these cases the linear operator $S$ does not generate a $C^0$-semigroup,
but rather a {\it bi-semigroup} \cite{BGK,LP2};
that is, the linear operator $S$ has the decomposition $S=S_1\oplus(-S_2)$ on a direct sum decomposition of the entire space $\bH=\bH_1\oplus\bH_2$, where $S_j$, $j=1,2$, generates a stable $C^0$-semigroup on $\bH_j$, $j=1,2$. We recall that the first order linear differential operators with constant coefficients $\partial_\tau-S$ is invertible on function spaces such as $L^2(\RR,\bH)$ if and only if the equation $u'=Su$ has an exponential dichotomy on $\RR$. We note that for any $u_0\in\bV^\perp$ the function $u(\tau)=u_0$ is a solution of equation $Au'=Q'(u^\pm)u$. Therefore, equation $Au'=Q'(u^\pm)u$ does not have an exponential dichotomy on the entire space $\bH$;
instead it exhibits an exponential dichotomy on a direct complement of the finite dimensional space $\bV^\perp$. To prove this result, we reduce the equation by using the decomposition
\begin{equation*}
A=\begin{bmatrix}
A_{11}&A_{12}\\
A_{21}&A_{22}\end{bmatrix}:\bV^\perp\oplus\bV\to\bV^\perp\oplus\bV,\quad Q'(u^\pm)=\begin{bmatrix}
0&0\\
0&Q_{22}'(u^\pm)\end{bmatrix}:\bV^\perp\oplus\bV\to\bV^\perp\oplus\bV.
\end{equation*}
Indeed, if $u$ is a solution of equation $Au'=Q'(u^\pm)u$, then the pair $(h,v)$ defined by $v=P_{\bV}u$ and $h=P_{\bV^\perp}u$, satisfies the system
\begin{equation*}
(A_{22}-A_{21}A_{11}^{-1}A_{12})v'=Q'_{22}(u^\pm)v,\quad h=-A_{11}^{-1}A_{12}v.
\end{equation*}
We introduce the linear operators $S^\eta_\pm:=A^{-1}\big(Q'(u^\pm)\pm\eta A\big)$,
$S^{\mathrm{r}}_\pm=\tilde{A}^{-1}Q_{22}'(u^\pm)$, where
$\tilde A:=(A_{22}-A_{21}A_{11}^{-1}A_{12})$.
These are densely defined, as we will show later (see Remark \ref{r3.1}), and indeed generate bi-semigroups.
Our dichotomy results are summarized in the following theorem.
\begin{theorem}\label{c3.8}
Assume Hypotheses (H1), (H3), (H4) and (H5). Then, the following assertions hold true:
\begin{enumerate}
\item[(i)] The bi-semigroup generated by $S^\eta_\pm$ is exponentially stable on $\bH$;
\item[(ii)] The bi-semigroup generated by $S^{\mathrm{r}}_\pm$ is exponentially stable on $\bV$;
\item[(iii)] Equation $Au'=\big(Q'(u^\pm)\pm\eta A\big)u$ has an exponential dichotomy on $\bH$;
\item[(iv)] The linear space $\bH$ can be decomposed into linear stable, center and
unstable subspaces, $\bH=\bV^\perp\oplus \bH^{\mathrm{s}}_\pm\oplus \bH^{\mathrm{u}}_\pm$ such that
\begin{align*}
&\mbox{for any}\;u_0\in\bV^\perp\;\mbox{the function}\;u(\tau)=u_0\;\mbox{is a solution of \eqref{ED-Lpm}}\nonumber\\
&\mbox{for any}\;u_0\in\bH^{\mathrm{s}}_\pm\;\mbox{the solution of \eqref{ED-Lpm} on}\;\RR_+\;\mbox{with}\;u(0)=u_0\;\mbox{decays exponentially at}\;+\infty\nonumber\\
&\mbox{for any}\;u_0\in\bH^{\mathrm{u}}_\pm\;\mbox{the solution of \eqref{ED-Lpm} on}\;\RR_-\;\mbox{with}\;u(0)=u_0\;\mbox{decays exponentially at}\;-\infty.
\end{align*}
\end{enumerate}
\end{theorem}
From this point, we
turn our attention towards our main goal, the existence of stable/unstable manifolds of solutions of equation \eqref{nonlinear} near the equilibria $u^\pm$. The first step is to show that this equation can be reduced to an equation of the form
\begin{equation}\label{general-GammaE}
\Gamma\bu'=E\bu+D(\bu,\bu),
\end{equation}
where $D(\cdot,\cdot)$ is a bounded, bilinear map, $E$ is negative definite and the linear operator $\Gamma^{-1}E$ generates a stable bi-semigroup $\{T^{\Gamma,E}_{\rms/\rmu}(\tau)\}_{\tau\geq0}$. To construct the manifolds, we need to analyze mild solutions of equation \eqref{general-GammaE} on $\RR_\pm$, using the results from Theorem~\ref{t2.9} and Theorem~\ref{c3.8}. Next, we apply, formally, the Fourier transform in \eqref{general-GammaE} and then we solve for $\cF\bu$.
In this way we obtain that mild solutions of equation \eqref{general-GammaE} on say $\RR_+$ satisfy equation
\begin{equation}\label{fix-cK}
\bu(\tau)=T^{\Gamma,E}_\rms(\tau)\bu(0)+\big(\cK_{\Gamma,E} D(\bu,\bu)\big)(\tau),\;\tau\geq 0.
\end{equation}
Here $\cK_{\Gamma,E}$ is the Fourier multiplier defined by the operator-valued function $R_{\Gamma,E}(\omega)=(2\pi\rmi\omega\Gamma-E)^{-1}$. The linear operator $\cK_{\Gamma,E}$ is well-defined and bounded on general Hilbert space valued $L^2(\RR,\bH)$ spaces. To construct center-stable manifolds of evolution equations on finite-dimensional spaces, one uses a fixed point argument to solve equation \eqref{fix-cK} on the space $C_{\mathrm{b}}(\RR,\bH)$, of continuous and bounded functions, or on $L^\infty(\RR,\bH)$. However, in our infinite-dimensional case such an argument does not seem to be possible, since the Fourier multiplier $\cK_{\Gamma,E}$ cannot be extended to a bounded linear operator on $L^\infty(\RR,\bH)$, see Example~\ref{e4.7}. Therefore, a crucial point of our construction is to find a proper subspace of $L^\infty(\RR,\bH)$ that is invariant under $\cK_{\Gamma,E}$. Since the operator-valued function $R_{\Gamma,E}$ is bounded, one can readily check that $H^1(\RR,\bH)$ is invariant under $\cK_{\Gamma,E}$. However, equation \eqref{fix-cK} is a functional equation on the half-line, not the full line. Moreover, since {\em not} every trajectory of the map $R_{\Gamma,E}$ belongs to $L^2$, it turns out that  the space $H^1(\RR_+,\bH)$ is {\em not} invariant under $\cK_{\Gamma,E}$. To deal with this setback,
we parameterize equation \eqref{fix-cK}. In Section~\ref{s4} we find solutions $\bu$ such that $\bu(0)=\bv_0-E^{-1}D(\bu(0),\bu(0))$ where $\bv_0$ is a parameter in a dense subspace. Substituting in  equation \eqref{fix-cK}, we obtain that to construct our center-stable/unstable manifold it is enough to prove existence of solutions of equation
\begin{equation}\label{mod-var-const1}
\bu=T_\rms^{\Gamma,E}(\cdot)P_\rms^{\Gamma,E}\bv_0+(\cK_{\Gamma,E}D(\bu,\bu))_{|\RR_+}-T_\rms^{\Gamma,E}(\cdot)P_\rms^{\Gamma,E}E^{-1}D(\bu(0),\bu(0)).
\end{equation}
An important step of our construction is to find an appropriate subspace of parameters $\bv_0$ such that the trajectory $T_\rms^{\Gamma,E}(\cdot)P_\rms^{\Gamma,E}\bv_0$ belongs to $H^1$.    To achieve this goal, we use that the linear operators $\Gamma$ and $E$ are self-adjoint and bounded, hence the bi-semigroup generator $\Gamma^{-1}E$ is similar to a multiplication operator
by a real valued function bounded from below on some $L^2$ space. See Section~\ref{s4} for the details of this construction.
\begin{theorem}\label{t1.3}
Assume Hypotheses (H1), (H3), (H4) and (H5). Then, for any integer $r\geq 1$ there exists a $C^r$ local stable manifold $\cM^+_\rms$ near $u^+$ and a $C^r$ local unstable manifold $\cM^-_\rmu$ near $u^-$, expressible in $w^\pm=u-u^\pm$ as graphs of $C^r$ functions $\cJ^+_\rms:\bH^+_\rms\cap\dom\big(|\tilde{A}^{-1}Q_{22}'(u^+)|^{\frac{1}{2}}\big)\to\bH^+_\rmu\oplus\bV^\perp$ and $\cJ^-_\rmu:\bH^-_\rmu\cap\dom\big(|\tilde{A}^{-1}Q_{22}'(u^-)|^{\frac{1}{2}}\big)\to\bH^-_\rms\oplus\bV^\perp$ from $\bH_{\rms/\rmu}^\pm\cap\dom\big(|\tilde{A}^{-1}Q_{22}'(u^\pm)|^{\frac{1}{2}}\big)$ with norm
$\|h\|_{\dom\big(|\tilde{A}^{-1}Q_{22}'(u^\pm)|^{\frac{1}{2}}\big)}=\big(\|h\|_{\bH}^2+\big\||\tilde{A}^{-1}Q_{22}'(u^\pm)|^{\frac{1}{2}}h\big\|_{\bH}^2\big)^{\frac{1}{2}}$
to $\bH_{\rmu/\rms}^\mp$ with norm $\|\cdot\|_{\bH}$, respectively, that are locally invariant under the flow of equation $Au'=Q(u)$ and uniquely  determined by the property that $w^\pm\in H^1(\RR_\pm,\bH)$. (Recall that $\tilde{A}=A_{22}-A_{21}A_{11}^{-1}A_{12}$.)
\end{theorem}
Finally, we use this result to prove that a large class of profiles converging to equilibria $u^\pm$ decay exponentially.
\begin{corollary}\label{c5.9}
Assume Hypotheses (H1), (H3), (H4) and (H5). Let $u^*\in H^1(\RR,\bH)$ be a solution 
of equation $Au_\tau=Q(u)$, $H^1$-convergent to $u^\pm$ in the sense that $u^*-u^\pm\in H^1(\RR_\pm,\bH)$, 
and let $-\nu_\pm:=-\nu\big(\tilde{A},Q'_{22}(u^\pm)\big)<0$ be the decay rate of the bi-semigroup generated by the pair $\big(\tilde{A},Q'_{22}(u^\pm)\big)$. 
Then, there exist $\alpha\in (0,\min\{\nu_+,\nu_-\})$ such that $u^*-u^\pm\in H^1_\alpha(\RR_\pm,\bH)$. In particular, we have that there exists $\alpha>0$ such that $\|u^*(\tau)-u^\pm\|\leq c(\alpha)e^{-\alpha|\tau|}$ for any $\tau\in\RR_\pm$.
\end{corollary}

\subsection{Applications to Boltzmann's equation}\label{s:boltzapp}
As mentioned above, the assumptions (H1)-(H5) of Section \ref{s:assumptions} are abstracted from, and satisfied by, the steady Boltzmann
equation with hard sphere collision potential \cite{MZ}, after the change of coordinates
$f\to \langle \xi \rangle^{1/2}f$, $Q \to \langle \xi \rangle^{1/2}Q$, and $\langle \xi\rangle:=\sqrt{1+|\xi|^2}$,
with $A=\xi_1/\langle \xi\rangle$.
Thus, Theorem \ref{t1.3} and Corollary \ref{c5.9} apply in particular to this fundamental case.
More generally, they apply to Boltzmann's equation with any collision potential, or ``cross-section,'' for which (H1)--(H5) are satisfied, with the
crucial aspects being boundedness of the nonlinear collision operator as a bilinear map and spectral gap of the linearized collision operator. 
This includes in addition the hard cutoff potentials of Grad \cite{CN,MZ}.

For the class of admissible cross-sections defined implicitly by (H1)-(H5), 
Corollary \ref{c5.9} implies exponential decay of $H^1_{loc}$ Boltzmann shock or boundary layer profiles {\it of arbitrary amplitude}, 
so long as such profiles (i) exist, and (ii) are uniformly bounded, and (iii) converge to their endstates in the weak sense 
that $u^*-u_\pm$ lies in $H^1(\R_\pm,\bH)$.
This fundamental property, a cornerstone of the dynamical systems approach to stability developed for viscous shock and relaxation waves,
had previously been established for kinetic shocks only in the small-amplitude limit \cite{LiuYu1,MZ}.

However, we do not here establish existence of large-amplitude profiles;
indeed, the ``structure problem,'' as discussed by Truesdell, Ruggeri, Boillat, and others \cite{BR}, 
of existence and structure of large-amplitude Boltzmann shocks, is one of the fundamental open problems in the theory.

\subsection{Discussion and open problems}\label{s:discuss}
In our analysis, the Hilbert structure of $\bH$ and symmetry of $A$ and $Q'(u_\pm)$ play an important role.
This structure is related to existence of a convex entropy for system \eqref{Relax} \cite{CLL}.
In the case of the Boltzmann equation, it is related to increase of thermodynamical entropy
and the Boltzmann $H$-Theorem; see \cite[Notes on the proof of Proposition 3.5, point 2]{MZ2}.

Further insight may be gained using the invertible
coordinate transformation $(-E)^{1/2}$ and spectral decomposition of $(-E)^{-1/2}\Gamma (-E)^{-1/2}$
to write the reduced system $\Gamma\bu'=E\bu+D(\bu,\bu)$ of \eqref{general-GammaE} formally as a
family of scalar equations
\be\label{scal}
(\alpha_\lambda \partial_\tau -1)\bu_\lambda=D_\lambda(\bu,\bu),
\ee
indexed by $\lambda$, where $\bu_\lambda$ is the coordinate of $\bu$ associated with spectrum $ \alpha_\lambda$, real,
in the eigendecomposition of $(-E)^{-1/2}\Gamma (-E)^{-1/2}$, with
$\|\bu\|_{\bH}^2=\int |\bu_\lambda|^2d\mu_\lambda$, where $d\mu$ denotes spectral measure associated with
$\lambda$ and $\alpha_\lambda$ are bounded with an accumulation point at $0$.
In the first place, we see directly that $(\Gamma\partial_\tau -E)$
is boundedly invertible on $L^2(\R,\bH)$, with resolvent kernel given in $\bu_\lambda$ coordinates by
the scalar resolvent
kernel
\be \label{R}
\hbox{\rm $R_\lambda (\tau,\theta)= \alpha_\lambda^{-1}e^{(\tau-\theta)/\alpha_\lambda^{-1}}$ whenever $(\tau-\theta)\alpha_\lambda<0$,}
\ee
which is readily seen to be {\it integrable with respect to $\tau$}, hence bounded coordinate-by-coordinate.

On the other hand, we see at the same time that the operator norm of the full kernel $R$
with respect to $L^2(\mu)$ is
\be\label{opnorm}
\|R(\tau,\theta)\|_{L^2(\mu)\to L^2(\mu)}= \sup_{\alpha_\lambda(\tau-\theta)<0}
|\alpha_\lambda^{-1}e^{(\tau-\theta)/\alpha_\lambda^{-1}}| \leq C/|\tau-\theta|\quad\mbox{for any}\quad \tau\ne\theta,
\ee
with a reverse inequality $\geq  C_1|\tau-\theta|$ holding for all $(\tau -\theta)_\lambda=c\alpha_\lambda$, in particular for a sequence of values approaching zero, hence {\it is unbounded}. If it happens that $\alpha_\lambda$ are continuous and cover all of $\R$, then this holds everywhere.
Likewise, a construction as in Example~\ref{e4.7} shows that
$(A\partial_\tau -Q'(u_\pm))$ is {\it not} boundedly invertible on $L^\infty(\R,\bH)$, motivating
our choice of spaces $H^1(\R,\bH)$, $H^1(\R_+,\bH)$ in the analysis,
rather than the usual $L^\infty(\R,\bH)$. This implies, by contradiction, that the operator norm of the resolvent kernel is 
not only unbounded but non-integrable (cf. \cite{LiuYu}).

Using the finite-dimensional variation of constants formula scalar mode-by-scalar mode,
we may, further, express \eqref{scal} as the fixed point equation
\begin{align}\label{spectral}
\bu_\lambda(\tau)&=
e^{\alpha_\lambda^{-1}\tau} \Pi_S \bu_\lambda(0)
+ \int_0^\tau \Pi_S \alpha_\lambda^{-1} e^{\alpha_\lambda^{-1}(\tau-\theta)}D_\lambda (\bu(\theta) ,\bu(\theta) )\, \rmd\theta\nonumber\\
&\quad-\int_\tau^{+\infty}\Pi_U \alpha_\lambda^{-1}e^{\alpha_\lambda^{-1}(\tau-\theta)}D_\lambda(\bu(\theta),\bu(\theta))\, \rmd\theta,
\end{align}
where $\Pi_U$ and $\Pi_S$ denote projections onto the stable and unstable subspaces determined by
$\sgn \alpha_\lambda$. In \eqref{spectral} and \eqref{deriv-spectral} we denote the spectral components of $\Pi_{S/U}g$ by $\Pi_{S/U}g_\lambda$ for any $g\in L^2(\mu)$, slightly abusing the notation.
From this we find after a brief calculation/integration by parts the derivative formula
\begin{align}\label{deriv-spectral}
	\bu_\lambda'(\tau)&=
	\alpha_\lambda^{-1} e^{\alpha_\lambda^{-1}\tau}\big( \Pi_S \bu_\lambda(0) + D_\lambda(\bu(0),\bu(0)) \big)
+ \int_0^\tau \Pi_S \alpha_\lambda^{-1} e^{\alpha_\lambda^{-1}(\tau-\theta)}D_\lambda' (\bu(\theta) ,\bu(\theta) ) \, \rmd\theta \nonumber\\
& \quad
-\int_\tau^{+\infty}\Pi_U \alpha_\lambda^{-1}e^{\alpha_\lambda^{-1}(\tau-\theta)}D_\lambda'(\bu(\theta) ,\bu(\theta))\, \rmd\theta,
\end{align}
which shows that $\bu\in H^1(\R_+,L^2(\mu))$ only if
$\alpha_\lambda^{-1} e^{\alpha_\lambda^{-1}\tau}\Pi_S \big(\bu_\lambda(0) + D_\lambda(\bu(0),\bu(0)) \big) \in L^2(\R_+,L^2(\mu))$,
or
$$
\Pi_S \big(  \bu_\lambda(0) + D_\lambda(\bu(0),\bu(0)) \big) \in \dom ((-E)^{-1/2} \Gamma (-E)^{-1/2})^{1/2}.
$$
This is quite different from the usual finite-dimensional ODE or dynamical systems scenario,
and explains why we need to take some care in setting up the $H^1(\R_+,\bH)$ contraction formulation.
In particular, we find it necessary to parametrize not by $\Pi_S \bu(0)$ as is customary in the
finite-dimensional ODE case, but rather by $\Pi_S \bv_0:=
\Pi_S \big(  \bu_\lambda(0) + D_\lambda(\bu(0),\bu(0)) \big)$ where $\bu_\lambda'(0)= \alpha_\lambda^{-1}\bv_0$.

\subsubsection{Relation to previous work}\label{s:previous}
The issue of noninvertibility of $A$ for relaxation systems \eqref{Relax} originating from
kinetic models and approximations was pointed out in \cite{MasZ0,MasZ1,Z4}.
This issue has been treated for finite-dimensional systems by Dressler and Yong \cite{DY}
using singular perturbation techniques; see also \cite{H,LMNPZ,NPZ}.
These analyses concern the case that $A$ has an {\it eigenvalue at zero}, and are of completely
different character from the analysis carried out
here of the case that $A$ has {\it essential spectrum at zero}, i.e., an essential singularity;
they are thus complementary to ours.
In the present, semilinear setting, the case that $A$ has a kernel is particularly simple,
giving a {\it constraint} restricting solutions (under suitable nondegeneracy conditions)
to a certain manifold, on which there holds a reduced relaxation system of standard,
nondegenerate, type.
%
For Boltzmann's equation \eqref{Boltz}, Liu and Yu \cite{LiuYu} have investigated
existence of invariant manifolds in a different (weighted $L^\infty(x,\xi)$) {\it Banach space} setting, using 
time-regularization and detailed pointwise bounds.

As noted earlier, the treatment of ill-posed equations $u'-Su=f$, and derivation of resolvent bounds via generalized
exponential dichotomies, has been carried out in a variety of contexts \cite{AM1,AM2,PSS,RobSal,BjornSand,SS1,SS2}.
The essential difference here is that the corresponding resolvent equation
$\Gamma u'-E u=f$ associated with \eqref{general-GammaE}  rewrites formally in the more singular form
$$u'- Su= \Gamma^{-1}f,
\quad S=\Gamma^{-1}E,
$$
for which the singularity $\Gamma^{-1}$ enters not only in the generator
$S$ but also in the source.  Thus, the solution operator is not the one $(\partial_x-S)^{-1}$
deriving from (generalized) exponential dichotomies of the homogeneous flow, but
the more singular $(\Gamma \partial_x-E)^{-1}$
of \eqref{fix-cK}, or, formally, the unbounded multiple $ \Gamma^{-1}(\partial_x -S)^{-1}$.
This explains the new features of unboundedness/nonintegrability of (the operator norm of) the resolvent, alluded to below \eqref{opnorm}.

\subsubsection{Open problems}\label{s:open}
Our $H^1$ analysis suggests a number of interesting open questions.
The first regards smoothing properties of the profile problem.
In the finite-dimensional evolution setting, regularity of solutions is limited only by regularity of
coefficients; here, however, that is not true even at the linear level.
Certainly, for further (e.g., stability) analysis, we require profiles of {\it at least} regularity $H^1$, and likely higher.
Our arguments can be modified to construct successively smaller stable manifolds in $H^s(\R_+)$, any $s\geq 1$,
but for constructing profiles one would like to intersect unstable/stable manifolds that are as large as possible, thus in the weakest possible
space.  Hence, it is interesting to know, for $H^1$ {\it profiles} of \eqref{prof} defined on the whole line, as opposed to decaying solutions defined
on a half line, is further regularity enforced?  For small-amplitude profiles, Kawashima-type energy estimates as in \cite{MZ,MZ2}
show that the answer is ``yes.''  A very interesting open question is whether one can find similar energy estimates in the large-amplitude
case yielding a similar conclusion. For related analysis in the finite-dimensional case, see \cite{MaZ4}.

A second question in somewhat opposite direction is ``what is the minimal regularity needed to enforce exponential decay?''
Specifically, we have shown that solutions of \eqref{nonlinear} that are sufficiently small in $H^1(\R_+)$ must decay pointwise at
exponential rate; moreover, they lie on our constructed local $H^1$ stable manifold.
What about solutions that are merely small in $L^\infty$? A very interesting observation due to Fedja Nazarov \cite{Na} based on
the indefinite Lyapunov functional relation $\langle u,Au\rangle'= \langle u,Q'(u^\pm)u\rangle -o(\|u\|_{\bH}^2)$ yields
the $L^2$-exponential decay result $e^{\beta |\cdot|}\|u(\cdot)\|\in L^2(\R_+)$ for some $\beta>0$, hence (by interpolation) in any
$L^p$, $2\leq p<\infty$. However, it is not clear what happens in the critical norm $p=\infty$; it would be very interesting to
exhibit a counterexample or prove decay.

	\medskip

	{\bf A glossary of notation:}\,  For $p\geq 1$, $J\subseteq\RR$ and $\bX$ a Banach space,
	$L^p(J,\bX)$ are the usual Lebesgue spaces on $J$ with values in
	$\bX$, associated with Lebesgue measure $d\tau$ on $J$. Similarly,
	$L^p(J,\bX;w(\tau)\rmd\tau)$ are the weighted spaces with a weight $w\geq
	0$. The respective spaces of bounded continuous functions on $J$
	are denoted by $C_{\rm b}(J,\bX)$ and $C_{\rm b}(J,\bX;w(\tau))$.
	$H^s(\RR, \bX)$, $s> 0$, is the usual Sobolev space of $\bX$
	valued functions. The identity operator on a Banach space $\bX$
	is denoted by $Id$ (or by $Id_{\bX}$ if its dependence on $\bX$
	needs to be stressed). The set of bounded linear operators from a
	Banach space $\bX$ to itself is denoted by $\cB(\bX)$.
	The set of bounded, Fredholm linear operators from a
	Banach space $\bX$ to itself is denoted by $\mathcal{F}$\textit{redh}$(\bX)$.
	For an
	operator $T$ on a Hilbert space we use $T^\ast$, $\dom(T)$, $\ker T$, $\im T$, $\sigma(T)$, $\rho(T)$, $R(\lambda,T)=(\lambda-T)^{-1}$ and $T_{|\bY}$ to denote the adjoint, domain,
	kernel, range, spectrum, resolvent set, resolvent operator and the
	restriction of $T$ to a subspace $\bY$ of $\bX$.  If $B:J\to\cB(\bX)$ then
	$M_B$ denotes the operator of multiplication by $B(\cdot)$ in
	$L^p(J,\bX)$ or $C_{\rm b}(J,\bX)$. If $\bX_1$ and $\bX_2$
	are two subspaces of $\bX$, then
	$\bX_1\oplus \bX_2$ denotes their direct (but not
	necessarily orthogonal) sum. The Fourier transform of a Borel measure $\mu$ is defined by $(\mathcal{F}\mu)(\omega)=\int_{\mathbb{R}}{e^{-2\pi ix\omega}d\mu(x)}$.

	\medskip

	{\bf Acknowledgement.}  Thanks to Hari Bercovici, Ciprian Demeter, Narcisse Randrianantoanina and Alberto Torchinsky for helpful discussions regarding Fourier multipliers, and
	to Fedja Nazarov and Benjamin Jaye for stimulating conversations regarding the $L^\infty$ decay problem.

	\section{Invertibility of the linearization at equilibria}\label{s2}

	In this section we study the invertibility property of the linearization of \eqref{nonlinear} at the equilibria $u^\pm$, given by the
	\begin{equation}\label{def-Lpm}
	\cL^\pm=A\pa_\tau-Q'(u^\pm).
	\end{equation}
	We can view the differential expression $\cL^\pm$ as a densely defined, closed operator on the weighted $L^2_\eta(\RR,\bH)$ space with
	domain $\dom(\cL^\pm)=\{u\in L^2_\eta(\RR,\bH):Au'\in L^2_\eta(\RR,\bH))\}$. Throughout this section we assume Hypotheses (H1) and (H3).
	First, we note that the operator $\cL^\pm$ is not invertible on $L^2_\eta(\RR,\bH)$ in the case when $\eta=0$.
	\begin{remark}\label{r2.1}
	Under assumptions (H1) and (H3), the linear operator $\cL^\pm$ is not invertible on $L^2(\RR,\bH)$. Indeed, one can readily check that the linear 
	operator $\cL^\pm$ is
	invertible on $L^2(\RR,\bH)$ with bounded inverse if and only if the operator of multiplication by the continuous, operator valued function $\widehat{\cL^\pm}(\omega)=2\pi\rmi\omega A-Q'(u^\pm)$
	is invertible on $L^2(\RR,\bH)$ with bounded inverse. From the later we can infer that $Q'(u^\pm)$ are invertible on $\bH$, which contradicts Hypothesis (H3).
	\end{remark}
	Next, we study the invertibility of $\cL^\pm$ on the weighted space $L^2_\eta(\RR,\bH)$ for some $\eta$ small, $\eta\ne0$. To achieve this goal it is enough to prove that
	the linear operators
	\begin{equation}\label{def-Lpm-eta}
	\cL^\pm_\eta=A\pa_\tau-Q'(u^\pm)\mp\eta A
	\end{equation}
	are invertible on $L^2(\RR,\bH)$.
	Again,
	we consider the differential expressions $\cL^\pm_\eta$ as closed linear operators on $L^2(\RR,\bH)$ with their maximal domain equal to the domain of $\cL^\pm$. To prove the desired invertibility properties, we start by proving some preliminary results satisfied by the linear operator $A$ and the bilinear map $Q$.
	First, we recall the following result which follows from the Kawashima condition.
	\begin{lemma}\label{l2.2}
	Assume Hypothesis (H1) and that the linear operator $T\in\cB(\bH)$ satisfies Hypothesis (H2). Then, the following assertion holds true
	$$\ker(A-\lambda I)\cap\bV^\perp=\{0\},\quad\mbox{for any}\quad\lambda\in\CC.$$
	\end{lemma}
	\begin{proof} Since the linear operator $A$ is self-adjoint we have that $\sigma(A)\subset\RR$, which implies that
	$$\ker(A-\lambda I)=\{0\},\quad\mbox{for any}\quad\lambda\in\CC\setminus\RR.$$
	Therefore, to prove the lemma it is enough to prove it for the case when $\lambda\in\RR$. From Hypothesis (H2)(iii) it follows that \begin{equation}\label{esp-2.1}
	\gamma\|h\|^2\leq \langle\big(\mathrm{Re}(KA-T)\big)h,h\rangle=\mathrm{Re}\langle(KA-T)h,h\rangle=\mathrm{Re}\langle \lambda Kh,h\rangle=\lambda \mathrm{Re}\langle Kh,h\rangle.
	\end{equation}
	for any $h\in\ker(A-\lambda I)\cap\bV^\perp$ and $\lambda\in\RR$. Moreover, since the linear operator $K$ is skew-symmetric, we obtain that $\mathrm{Re}\langle Kh,h\rangle=0$. From \eqref{esp-2.1} we infer that $h=0$, proving the lemma.
	\end{proof}
	The next lemma is a crucial preliminary result of this section
	used
	to prove the invertibility of the linear operators $\cL^\pm_\eta$.
	\begin{lemma}\label{l2.3}
	Assume Hypothesis (H1) and that the linear operator $T\in\cB(\bH)$ satisfies Hypothesis (H2). Then, there exists $\eta_1>0$ such that
	\begin{equation}\label{ker-TplusA}
	\ker(T+sA)=\{0\}\quad\mbox{for any}\quad s\in(-\eta_1,\eta_1)\setminus\{0\}.
	\end{equation}
	\end{lemma}
	\begin{proof}
	To begin the proof, we introduce $P_{\bV}$ and $P_{\bV^\perp}$, the orthogonal projections onto
	$\bV$ and $\bV^\perp$, respectively. Since the linear operators $A$ and $T$ satisfy Hypotheses (H1) and (H2), we obtain the following decompositions
	\begin{equation}\label{decomp-AT}
	A=\begin{bmatrix}
	A_{11}&A_{12}\\
	A_{21}&A_{22}\end{bmatrix}:\bV^\perp\oplus\bV\to\bV^\perp\oplus\bV,\quad T=\begin{bmatrix}
	0&0\\
	0&T_{22}\end{bmatrix}:\bV^\perp\oplus\bV\to\bV^\perp\oplus\bV
	\end{equation}
	Next, we consider $s\in\RR\setminus\{0\}$, $u\in\ker(T+sA)$. If we denote by
	$v=P_\bV u$ and $h=P_{\bV^\perp}u$, the pair $(h,v)\in\bV^\perp\oplus\bV$ satisfies the system
	\begin{equation}\label{2.3-1}
	A_{11}h+A_{12}v=0,\quad
	sA_{21}h+sA_{22}v+T_{22}v=0.
	\end{equation}
	From Hypothesis (H1) we have that the linear operator $A$ is self-adjoint, therefore, since the decomposition $\bH=\bV^\perp\oplus\bV$ is orthogonal, we obtain that
	the linear operator $A_{11}$ is a self-adjoint linear operator on the finite dimensional vector space $\bV^\perp$. It follows that the following orthogonal decomposition holds true:
	$\bV^\perp=\ker A_{11}\oplus\im A_{11}$. We conclude that there exist $y\in\ker A_{11}$ and $z\in\im A_{11}$ such that
	\begin{equation}\label{2.3-2}
	h=y+z\quad\mbox{and}\quad y\perp z.
	\end{equation}
	From \eqref{2.3-1} and \eqref{2.3-2} we infer that
	\begin{equation}\label{2.3-3}
	A_{11}z+A_{12}v=0.
	\end{equation}
	Since $\ker A_{11}\perp\im A_{11}$, $\im A_{11}\subseteq\bV^\perp$ and $\dim\bV^\perp<\infty$ we obtain that the linear operator $\tilde{A}_{11}=A_{11\,\big|\im A_{11}}$ is invertible with bounded inverse on $\im A_{11}$. From \eqref{2.3-3} it follows that $z=-(\tilde{A}_{11})^{-1}A_{12}v$. Moreover, from \eqref{2.3-1} and \eqref{2.3-2} we have that
	\begin{equation}\label{2.3-4}
	sA_{21}y+sA_{12}z+sA_{22}v+T_{22}v=0.
	\end{equation}
	Furthermore, we claim that $A_{21}y\perp v$. Since the linear operator $A$ is self-adjoint, from the decomposition \eqref{decomp-AT} we have that $A_{21}^*=A_{12}$. Since, in addition, $A_{11}^*=A_{11}$ and $y\in\ker A_{11}$, from \eqref{2.3-3} we infer that
	\begin{equation}\label{2.3-5}
	\langle A_{21}y,v\rangle=\langle y,A_{21}^*v\rangle=\langle y,A_{12}v\rangle=-\langle y,A_{11}z\rangle=-\langle A_{11}y,z\rangle=0,
	\end{equation}
	proving the claim. Taking scalar product with $v$ in \eqref{2.3-4} we obtain that
	\begin{equation}\label{2.3-6}
	s\underbrace{\langle A_{21}y,v\rangle}_{=0}+s\langle A_{12}z,v\rangle+s\langle A_{22}v,v\rangle+\langle T_{22}v,v\rangle=0.
	\end{equation}
	Since $z=-(\tilde{A}_{11})^{-1}A_{12}v$ from \eqref{2.3-6} it follows that
	\begin{equation}\label{2.3-7}
	-\langle T_{22}v,v\rangle=s\big\langle \big(A_{22}-A_{21}(\tilde{A}_{11})^{-1}A_{12}\big)v,v\big\rangle.
	\end{equation}
	Next, we define $\eta_1=\frac{\delta}{2}\Big(\big\|A_{22}-A_{21}(\tilde{A}_{11})^{-1}A_{12}\big\|+1\Big)^{-1}>0$ and we consider $s\in (-\eta_1,\eta_1)\setminus\{0\}$.
	From \eqref{2.3-7} and Hypothesis (H1) we infer that
	\begin{equation}\label{2.3-8}
	\delta\|v\|^2\leq -\langle T_{22}v,v\rangle=\Big|s\big\langle \big(A_{22}-A_{21}(\tilde{A}_{11})^{-1}A_{12}\big)v,v\big\rangle\Big|\leq |s|\big\|A_{22}-A_{21}(\tilde{A}_{11})^{-1}A_{12}\big\|\|v\|^2\leq \frac{\delta}{2}\|v\|^2,
	\end{equation}
	which implies that $v=0$. From \eqref{2.3-1} we conclude that $A_{11}h=0$ and $A_{21}h=0$. Since $h\in\bV^\perp$ and $v=0$ we obtain that
	$Ah=AP_{\bV^\perp}h=A_{11}h+A_{21}h=0$, proving that $h\in\ker A\cap\bV^\perp$. From Lemma~\ref{l2.2} we conclude that $h=0$, proving the lemma.
	\end{proof}
	\begin{remark}\label{r2.4}
	In many interesting applications the linear operator $A$ is a multiplication operator by a scalar, non-zero function on a some function space. Therefore, it makes sense to assume that the linear operator $A$ is one-to-one. This might simplify the proof of Lemma~\ref{l2.3} a little bit. However, the main idea would be the same, so we choose to formulate the lemma under minimal conditions.
	\end{remark}
	In the next lemma we prove the invertibility of the perturbation by a small multiple of $A$ of $Q'(u^\pm)$.
	\begin{lemma}\label{l2.5}
	Assume Hypotheses (H1) and (H3). Then, there exists $\eta^*>0$ such that
	$Q'(u^\pm)\pm\eta A$ is invertible with bounded inverse on $\bH$ for all $\eta\in(0,\eta^*)$.
	\end{lemma}
	\begin{proof}
	 From Lemma~\ref{l2.3} we have that there exists $\eta^\pm_1>0$ such that
	\begin{equation}\label{2.5-1}
	\ker\big(Q'(u^\pm)+sA\big)=\{0\}\quad\mbox{for any}\quad s\in(-\eta^\pm_1,\eta^\pm_1)\setminus\{0\}.
	\end{equation}
	Since the equilibria $u^\pm$ satisfy Hypothesis (H3), we have that $Q'(u^\pm)_{|\bV}\leq-\delta_\pm I_\bV$, for some $\delta_\pm>0$. Moreover, $\dim\bV^\perp<\infty$, which implies that $Q'(u^\pm)$ are bounded, Fredholm linear operators on $\bH$. Since $\mathcal{F}$\textit{redh}$(\bH)$, the set of all bounded, Fredholm linear operators on $\bH$, is an open set of $\cB(\bH)$, we infer that there exists $\eta^\pm_2>0$ such that
	\begin{equation}\label{2.5-2}
	\{S\in\cB(\bH):\|S-Q'(u^\pm)\|<\eta^\pm_2\}\subset\cF\mbox{\it redh}(\bH).
	\end{equation}
	Let $\eta^*=\min\{\eta^+_1,\eta^-_1,\eta^+_2(\|A\|+1)^{-1},\eta^-_2(\|A\|+1)^{-1}\}>0$ and $\eta\in(0,\eta^*)$. From Hypothesis (H3), \eqref{2.5-1} and \eqref{2.5-2} we infer that $Q'(u^\pm)\pm\eta A$ is a bounded, self-adjoint, one-to-one, Fredholm linear operator on $\bH$, proving the lemma.
	\end{proof}
	Let $\eta^*>0$ be the positive constant defined in Lemma~\ref{l2.5} and fix $\eta\in(0,\eta^*)$. Taking Fourier Transform, one can readily check that the linear operators $\cL^\pm_\eta$ are similar on $L^2(\RR,\bH)$ to the operators of multiplication by the operator valued function $\widehat{\cL^\pm_\eta}:\RR\to\cB(\bH)$ defined by $\widehat{\cL^\pm_\eta}(\omega)=2\pi\rmi\omega A-Q'(u^\pm)\mp\eta A$. Therefore, to prove the theorem it is enough to show that $\widehat{\cL^\pm_\eta}(\omega)$ is invertible on $\bH$ for all $\omega\in\RR$ and that $\sup_{\omega\in\RR}\|\widehat{\cL^\pm_\eta}(\omega)^{-1}\|<\infty$. Since the linear operators $A$ and $Q'(u^\pm)$ are self-adjoint on $\bH$ we obtain that
	\begin{equation}\label{2.6-2}
	\big(\widehat{\cL^\pm_\eta}(\omega)\big)^*=-2\pi\rmi\omega A^*-\big(Q'(u^\pm)\pm\eta A\big)^*=-2\pi\rmi\omega A-\big(Q'(u^\pm)\pm\eta A\big)=\widehat{\cL^\pm_\eta}(-\omega)
	\end{equation}
	for any $\omega\in\RR$. In the next couple of lemmas we prove several partial results needed to establish that the linear operator $\widehat{\cL^\pm_\eta}(\omega)$ is invertible on $\bH$ for any $\omega\in\RR$ and to prove that its inverse is uniformly bounded in $\omega$.
	\begin{lemma}\label{l2.6}
	Assume Hypotheses (H1) and (H3). Then, there exists $\eta^*>0$ such that
	$\widehat{\cL^\pm_\eta}(\omega)$ is a Fredholm operator on $\bH$ for any $\omega\in\RR$ and $\eta\in(0,\eta^*)$.
	\end{lemma}
	\begin{proof} To simplify the notation, we denote by $E^\eta_\pm=Q'(u^\pm)\pm\eta A\in\cB(\bH)$. From Hypotheses (H1) and (H3) we have that the linear operators $A$ and $Q'(u^\pm)$ are self-adjoint on $\bH$, which implies that the linear operators $E^\eta_\pm$ are self-adjoint on $\bH$. Moreover, $\eta^*>0$ can be chosen small enough such that
	\begin{equation}\label{2.6-1}
	-\langle E^\eta_\pm v,v\rangle\geq -\langle Q'(u^\pm)v,v\rangle-|\eta|\,\|A\|\|v\|^2\geq \frac{\delta}{2}\|v\|^2
	\end{equation}
	for all $v\in\bV$ and $\eta\in (0,\eta^*)$. Here $\delta=\min\{\delta_+,\delta_-\}>0$.

	The first step towards proving the lemma is to prove that
	\begin{equation}\label{2.6-3}
	P_\bV\widehat{\cL^\pm_\eta}(\omega)P_\bV\;\mbox{is a Fredholm operator on}\,\,\bH\,\,\mbox{for any}\;\omega\in\RR,\eta\in(0,\eta^*).
	\end{equation}

	Since $P_\bV$ is the orthogonal projection onto $\bV$ we have that $P_\bV$ is a self-adjoint linear operator on $\bH$. Thus, from \eqref{2.6-2} we compute that
	\begin{align}\label{2.6-4}
	\mathrm{Re}\Big(P_\bV\widehat{\cL^\pm_\eta}(\omega)P_\bV\Big)&=\frac{1}{2}\Big[P_\bV\widehat{\cL^\pm_\eta}(\omega)P_\bV+\big(P_\bV\widehat{\cL^\pm_\eta}(\omega)P_\bV\big)^*\Big]=
	\frac{1}{2}\Big(P_\bV\widehat{\cL^\pm_\eta}(\omega)P_\bV+P_\bV\widehat{\cL^\pm_\eta}(-\omega)P_\bV\Big)\nonumber
	\\&=-P_\bV E^\eta_\pm P_\bV\;\mbox{for all}\;\omega\in\RR.
	\end{align}
	From \eqref{2.6-1} it follows that
	\begin{align}\label{2.6-5}
	\mathrm{Re}\langle P_\bV\widehat{\cL^\pm_\eta}(\omega)P_\bV u,u\rangle&=\langle \mathrm{Re}\big(P_\bV\widehat{\cL^\pm_\eta}(\omega)P_\bV\big) u,u\rangle\nonumber=-\langle P_\bV E^\eta_\pm P_\bV u,u\rangle\\
	&=-\langle E^\eta_\pm P_\bV u,P_\bV u\rangle\geq \frac{\delta}{2}\|P_\bV u\|^2
	\end{align}
	for any $\omega\in\RR$ and $u\in\bH$. Using again the fact that $P_\bV$ is a self-adjoint projector on $\bH$, we conclude that
	\begin{equation}\label{2.6-6}
	\mathrm{Re}\langle P_\bV\widehat{\cL^\pm_\eta}(\omega)P_\bV u,u\rangle=\mathrm{Re}\langle P_\bV\widehat{\cL^\pm_\eta}(\omega)P_\bV u,P_\bV u\rangle\leq \|P_\bV\widehat{\cL^\pm_\eta}(\omega)P_\bV u\| \|P_\bV u\|
	\end{equation}
	for any $\omega\in\RR$ and $u\in\bH$. From \eqref{2.6-5} and \eqref{2.6-6} we conclude that
	\begin{equation}\label{2.6-7}
	\|P_\bV\widehat{\cL^\pm_\eta}(\omega)P_\bV u\|\geq\frac{\delta}{2} \|P_\bV u\|\;\mbox{for all}\; u\in\bH, \omega\in\RR.
	\end{equation}
	Next, we use the estimate \eqref{2.6-7} to prove that $\im\big(P_\bV\widehat{\cL^\pm_\eta}(\omega)P_\bV\big)$ is a closed subspace of $\bH$ for any $\omega\in\RR$. Indeed, let us fix $\omega\in\RR$, consider $\{u_n\}_{n\geq 1}$ a sequence of vectors in $\bH$ and let $y\in\bH$ be chosen such that $y_n:=P_\bV\widehat{\cL^\pm_\eta}(\omega)P_\bV u_n\to y$ in $\bH$ as $n\to\infty$. Next, we define the sequence of vectors $\{v_n\}_{n\geq 1}$ by $v_n=P_\bV u_n$ and we note that $P_\bV\widehat{\cL^\pm_\eta}(\omega)P_\bV v_n=P_\bV\widehat{\cL^\pm_\eta}(\omega)P_\bV u_n$ for any $n\geq 1$. From \eqref{2.6-7} we obtain that
	\begin{equation}\label{2.6-8}
	\|v_n-v_m\|=\|P_\bV(v_n-v_m)\|\leq \frac{2}{\delta}\|P_\bV\widehat{\cL^\pm_\eta}(\omega)P_\bV(v_n-v_m)\|=\frac{2}{\delta}\|y_n-y_m\|
	\end{equation}
	for all $n,m\geq 1$, which proves that the sequence $\{v_n\}_{n\geq 1}$ is a fundamental sequence in the closed subspace $\bV$. It follows that there exists $v_*\in\bV$ such that $v_n\to v_*$ in $\bV$ as $n\to\infty$. Passing to the limit, we conclude that $y=P_\bV\widehat{\cL^\pm_\eta}P_\bV v_*\in\im\big(P_\bV\widehat{\cL^\pm_\eta}(\omega)P_\bV\big)$, proving that
	\begin{equation}\label{2.6-9}
	\im\big(P_\bV\widehat{\cL^\pm_\eta}(\omega)P_\bV\big)\;\mbox{is a closed subspace of}\;\bH\;\mbox{for any}\; \omega\in\RR.
	\end{equation}
	Next, we prove that $\ker\big(P_\bV\widehat{\cL^\pm_\eta}(\omega)P_\bV\big)=\bV^\perp$ for any $\omega\in\RR$. Fix $\omega\in\RR$. First, we note that $\bV^\perp=\ker P_\bV\subseteq\ker\big(P_\bV\widehat{\cL^\pm_\eta}(\omega)P_\bV\big)$. Second, if $u\in\ker\big(P_\bV\widehat{\cL^\pm_\eta}(\omega)P_\bV\big)$ from \eqref{2.6-7} we obtain that $P_\bV u=0$, which implies that $u\in\bV^\perp$. Thus, we can infer that
	\begin{equation}\label{2.6-10}
	\ker\big(P_\bV\widehat{\cL^\pm_\eta}(\omega)P_\bV\big)=\bV^\perp\;\mbox{for any}\; \omega\in\RR.
	\end{equation}
	Moreover, from \eqref{2.6-2} and \eqref{2.6-10} we conclude that
	\begin{equation}\label{2.6-11}
	\Big[\im\big(P_\bV\widehat{\cL^\pm_\eta}(\omega)P_\bV\big)\Big]^\perp=\ker\big(P_\bV\widehat{\cL^\pm_\eta}(\omega)P_\bV\big)^*=\ker\big(P_\bV\widehat{\cL^\pm_\eta}(-\omega)P_\bV\big)=\bV^\perp
	\end{equation}
	for any $\omega\in\RR$. Since $\bV^\perp$ is a finite dimensional subspace of $\bH$, from \eqref{2.6-9}--\eqref{2.6-11} we conclude that $P_\bV\widehat{\cL^\pm_\eta}(\omega)P_\bV$ is a bounded linear operator on $\bH$ with closed range of finite codimension and has a finite dimensional kernel, which implies that $P_\bV\widehat{\cL^\pm_\eta}(\omega)P_\bV$ is a Fredholm operator on $\bH$ for all $\omega\in\RR$, proving claim \eqref{2.6-3}. In addition, we note that
	\begin{equation}\label{2.6-12}
	\widehat{\cL^\pm_\eta}(\omega)=(P_\bV+P_{\bV^\perp})\widehat{\cL^\pm_\eta}(\omega)(P_\bV+P_{\bV^\perp})=P_\bV\widehat{\cL^\pm_\eta}(\omega)P_\bV+P_\bV\widehat{\cL^\pm_\eta}(\omega)P_{\bV^\perp}
	+P_{\bV^\perp}\widehat{\cL^\pm_\eta}(\omega)
	\end{equation}
	for any $\omega\in\RR$. Using again that $\bV^\perp$ is a finite dimensional subspace of $\bH$ we have that the projection $P_{\bV^\perp}$ is a finite rank operator, which implies that
	\begin{equation}\label{2.6-13}
	P_\bV\widehat{\cL^\pm_\eta}(\omega)P_{\bV^\perp}+P_{\bV^\perp}\widehat{\cL^\pm_\eta}(\omega)\;\mbox{is a finite rank operator for any}\;\omega\in\RR.
	\end{equation}
	Finally, from \eqref{2.6-3}, \eqref{2.6-12} and \eqref{2.6-13} we infer that
	$\widehat{\cL^\pm_\eta}(\omega)$ is a Fredholm operator on $\bH$ for any $\omega\in\RR$, proving the lemma.
	\end{proof}
	\begin{lemma}\label{l2.7}
	Assume Hypotheses (H1) and (H3). Then, there exists $\eta^*>0$ such that
	\begin{enumerate}
	\item[(i)] $\ker\big(\widehat{\cL^\pm_\eta}(\omega)\big)=\{0\}$ for any $\omega\in\RR$ and $\eta\in(0,\eta^*)$;
	\item[(ii)] $\widehat{\cL^\pm_\eta}(\omega)$ is invertible with bounded inverse on $\bH$ for any $\omega\in\RR$ and $\eta\in(0,\eta^*)$.
	\end{enumerate}
	\end{lemma}
	\begin{proof}
	As
	in the previous lemma, we use the notation $E^\eta_\pm=Q'(u^\pm)\pm\eta A$. To start the proof of (i), we fix $\eta\in (0,\eta^*)$ with $\eta^*>0$ defined such that
	\eqref{2.6-1} holds true. Let $\omega\in\RR$ and $u\in\ker\big(\widehat{\cL^\pm_\eta}(\omega)\big)$. We note that if $\omega=0$ we have that $E^\eta_\pm u=0$. By Lemma~\ref{l2.5} we know that
	$E^\eta_\pm=Q'(u^\pm)\pm\eta A$ is invertible on $\bH$, which implies that $u=0$.

	In the case when $\omega\ne0$, we denote by $v=P_\bV u$ and $h=P_{\bV^\perp}u$. Moreover, we have that
	\begin{equation}\label{2.7-1}
	0=\langle \widehat{\cL^\pm_\eta}(\omega) u,u\rangle=2\pi\rmi\omega\langle Au,u\rangle-\langle E^\eta_\pm u,u\rangle.
	\end{equation}
	From Hypotheses (H1) and (H3) we have that the linear operators $A$ and $E^\eta_\pm$ are self-adjoint on $\bH$, which implies that $\langle Au,u\rangle,\langle E^\eta_\pm u,u\rangle\in\RR$. Since $\omega\ne 0$, from \eqref{2.7-1} it follows that $\langle Au,u\rangle=\langle E^\eta_\pm u,u\rangle=0$, which is equivalent to
	\begin{equation}\label{2.7-2}
	\langle Au,u\rangle=\langle Q'(u^\pm) u,u\rangle=0.
	\end{equation}
	From Hypothesis (H3) we have that $\ker\big(Q'(u^\pm)\big)=\bV^\perp$ and $\im\big(Q'(u^\pm)\big)=\bV$, thus, we conclude that $Q'(u^\pm)=P_\bV Q'(u^\pm)P_\bV$. Since the projection $P_\bV$ is self-adjoint, it follows that
	\begin{equation}\label{2.7-3}
	0=-\langle Q'(u^\pm) u,u\rangle=\langle P_\bV Q'(u^\pm)P_\bV u,u\rangle=\langle Q'(u^\pm)P_\bV u,P_\bV u\rangle\geq \delta\|P_\bV u\|^2=\delta\|v\|^2.
	\end{equation}
	Hence, $v=0$. From Hypothesis (H3) we obtain that
	\begin{equation}\label{2.7-4}
	0=\widehat{\cL^\pm_\eta}(\omega)u=\widehat{\cL^\pm_\eta}(\omega)h=(2\pi\rmi\omega\mp\eta)Ah,
	\end{equation}
	which implies that $h\in\ker A\cap\bV^\perp$. From Lemma~\ref{l2.2} we conclude that $h=0$, which implies that $u=v+h=0$, proving (i).

	\noindent\textit{Proof of (ii)} First, we note that from \eqref{2.6-2} and (i) we have that
	\begin{equation}\label{2.7-5}
	\ker\big(\widehat{\cL^\pm_\eta}(\omega)\big)^*=\ker\big(\widehat{\cL^\pm_\eta}(-\omega)\big)=\{0\}\;\mbox{for any}\;\omega\in\RR,\eta\in(0,\eta^*).
	\end{equation}
	From Lemma~\ref{l2.6}, (i) and \eqref{2.7-5} we have that the linear operator $\widehat{\cL^\pm_\eta}(\omega)$ is Fredholm on $\bH$, with trivial kernel and its adjoint has trivial kernel, for any $\omega\in\RR$ and $\eta\in(0,\eta^*)$. We infer that $\widehat{\cL^\pm_\eta}(\omega)$ is invertible with bounded inverse on $\bH$ for any $\omega\in\RR$ and $\eta\in(0,\eta^*)$, proving the lemma.
	\end{proof}
	\begin{lemma}\label{l2.8}
	Assume Hypotheses (H1) and (H3). Then, there exists $\eta^*>0$ such that
	\begin{equation}\label{inv-est}
	\sup_{\omega\in\RR}\big\|\big(\widehat{\cL^\pm_\eta}(\omega)\big)^{-1}\big\|<\infty\;\mbox{for any}\; \eta\in(0,\eta^*).
	\end{equation}
	\end{lemma}
	\begin{proof}
	First, we recall the notation $E^\eta_\pm=Q'(u^\pm)\pm\eta A$. Moreover, we fix $\eta\in (0,\eta^*)$ with $\eta^*>0$ defined in Lemma~\ref{l2.5} chosen small enough such that
	\eqref{2.6-1} holds true. Let $\omega\in\RR\setminus\{0\}$, $w\in\bH$ and $u=\big(\widehat{\cL^\pm_\eta}(\omega)\big)^{-1}w$. It follows that
	\begin{equation}\label{2.8-1}
	2\pi\rmi\omega\langle Au,u\rangle-\langle E^\eta_\pm u,u\rangle=\langle w,u\rangle.
	\end{equation}
	We recall that, from Hypotheses (H1) and (H3) we have that the linear operators $A$ and $E^\eta_\pm$ are self-adjoint on $\bH$, which implies that $\langle Au,u\rangle,\langle E^\eta_\pm u,u\rangle\in\RR$. Since $\omega\ne 0$, from \eqref{2.8-1} it follows that
	\begin{equation}\label{2.8-2}
	\langle Au,u\rangle=\frac{1}{2\pi\omega}\,\mathrm{Im}\langle w,u\rangle\;\mbox{and}\;\langle E^\eta_\pm u,u\rangle=-\mathrm{Re}\langle w,u\rangle.
	\end{equation}
	Let $\gamma=\min\{\gamma_+,\gamma_-\}>0$, where $\gamma_\pm$ are the constants from Hypothesis (H3). From \eqref{2.8-2} and Hypothesis (H3) it follows that
	\begin{align}\label{2.8-3}
	\gamma\|u\|^2&\leq \big\langle\mathrm{Re}\big(K_\pm A-Q'(u^\pm)\big)u,u\big\rangle=\mathrm{Re}\big\langle\big(K_\pm A-Q'(u^\pm)\big)u,u\big\rangle\nonumber\\
	&=\mathrm{Re}\langle K_\pm A u,u\rangle-\langle E^\eta_\pm u,u\rangle\mp\eta\langle Au,u\rangle=-\mathrm{Re}\langle A u,K_\pm  u\rangle+\mathrm{Re}\langle w,u\rangle\mp \frac{\eta}{2\pi\omega}\,\mathrm{Im}\langle w,u\rangle\nonumber\\
	&\leq |\mathrm{Re}\langle w,u\rangle|+\frac{\eta^*}{2\pi|\omega|} |\mathrm{Im}\langle w,u\rangle|+|\langle A u,K_\pm  u\rangle|\nonumber\\&\leq \Big(1+\frac{\eta^*}{2\pi|\omega|}\Big) |\langle w,u\rangle|+\frac{1}{2\pi|\omega|}|\langle E^\eta_\pm u+w,K_\pm  u\rangle|\nonumber\\
	&\leq \Big(1+\frac{\eta^*}{2\pi|\omega|}\Big) \|w\|\,\|u\|+\frac{1}{2\pi|\omega|}\Big(\|E^\eta_\pm\|\,\|K_\pm\|\,\|u\|^2+ \|K_\pm\|\,\|w\|\,\|u\|\Big)\nonumber\\
	&\leq \Big(1+\frac{c}{|\omega|}\Big) \|w\|\,\|u\|+\frac{c}{|\omega|}\|u\|^2.
	\end{align}
	for some $c>0$, that depends only on $\eta>0$. We infer that
	\begin{equation}\label{2.8-4}
	 \gamma\|u\|^2\leq \big(1+\frac{\gamma}{2}\big)\|w\|\,\|u\|+\frac{\gamma}{2}\|u\|^2\quad\mbox{whenever}\quad |\omega|\geq\frac{2c}{\gamma},
	\end{equation}
	which implies that
	\begin{equation}\label{2.8-5}
	\big\|\big(\widehat{\cL^\pm_\eta}(\omega)\big)^{-1}w\big\|\leq \big(1+\frac{\gamma}{2}\big)\|w\|\quad\mbox{for any}\quad w\in\bH,\; \omega\in\RR\quad\mbox{with}\quad|\omega|\geq\frac{2c}{\gamma}.
	\end{equation}
	One can readily check that the operator valued function $\widehat{\cL^\pm_\eta}:\RR\to\cB(\bH)$ is continuous on $\RR$. In addition, from Lemma~\ref{l2.7}(ii) we have that $\widehat{\cL^\pm_\eta}(\omega)$ is invertible for any $\omega\in\RR$. It follows that the operator valued function $\big(\widehat{\cL^\pm_\eta}(\cdot)\big)^{-1}$ is continuous on $\RR$ which implies that
	\begin{equation}\label{2.8-6}
	\sup\{\big\|\big(\widehat{\cL^\pm_\eta}(\omega)\big)^{-1}\big\|:\omega\in\RR,|\omega|\leq\frac{2c}{\gamma}\}<\infty.
	\end{equation}
	The lemma follows
	readily
	from \eqref{2.8-5} and \eqref{2.8-6}.
	\end{proof}
	\noindent{\bf Proof of Theorem 1.1.}
	The theorem is a direct consequence of the last three lemmas. Indeed, since the linear operators $\cL^\pm_\eta$ are similar on $L^2(\RR,\bH)$ to the operators of multiplication by the operator valued function $\widehat{\cL^\pm_\eta}$, from Lemma~\ref{l2.7}(ii) and Lemma~\ref{l2.8} we have that the operator of multiplication by the operator valued function $\widehat{\cL^\pm_\eta}$ is invertible with bounded inverse.

	\section{Stable bi-semigroups and exponential dichotomies of the linearization}\label{s3}
	In this section we continue to study the properties of the linearization of equation \eqref{nonlinear} (with $s=0$) at the equilibria $u^\pm$. Throughout this section we assume Hypotheses (H1), (H3), (H4) and (H5) and we recall the notation used in several lemmas from the previous section, $E^\eta_\pm=Q'(u^\pm)\pm\eta A$, with $\eta\in(0,\eta^*)$.
	In particular, we prove that equation
	\begin{equation}\label{ED-etaLpm}
	Au'=(Q'(u^\pm)\pm\eta A)u
	\end{equation}
	is equivalent to an equation of the form $u'=Su$, where the linear operator $S$ generates a stable bi-semigroup. Here, $\eta\in(0,\eta^*)$ and $\eta^*>0$ is defined in Section~\ref{s2}. We recall that a linear operator generates a bi-semigroup on a Banach or Hilbert space $\bX$,
	if there exist two closed subspaces $\bX_{j}$, $j=1,2$, of $\bX$, \textit{invariant} under $S$, such that $\bX=\bX_1\oplus\bX_2$  and $S_{|\bX_1}$ and $-S_{|\bX_2}$ generate $C^0$-semigroups on $\bX_j$, $j=1,2$. We say that the bi-semigroup is stable if the two semigroups are stable.

	It is well-known, see e.g. \cite{LPS,LP2}, that the invertibility of $\cL^\pm$ on $L^2(\RR,\bH)$ is equivalent to the exponential dichotomy on $\bH$ of equations
	\begin{equation}\label{ED-Lpm}
	Au'=Q'(u^\pm)u.
	\end{equation}
	From Remark~\ref{r2.1} we have that the linear operators $\cL^\pm$ defined in \eqref{def-Lpm} are \textit{not invertible} on $L^2(\RR,\bH)$. Moreover, we note that for any $u_0\in\bV^\perp$ the constant function $u(\tau)=u_0$ is a solution of equation \eqref{ED-Lpm}. In this section we also prove that equations \eqref{ED-Lpm} exhibit an exponential dichotomy on a direct complement of the finite dimensional space $\bV^\perp$. Using the decomposition
	\begin{equation}\label{decomp-AQ}
	A=\begin{bmatrix}
	A_{11}&A_{12}\\
	A_{21}&A_{22}\end{bmatrix}:\bV^\perp\oplus\bV\to\bV^\perp\oplus\bV,\quad Q'(u^\pm)=\begin{bmatrix}
	0&0\\
	0&Q_{22}'(u^\pm)\end{bmatrix}:\bV^\perp\oplus\bV\to\bV^\perp\oplus\bV
	\end{equation}
	and denoting by $v=P_{\bV}u$ and $h=P_{\bV^\perp}$, one can readily check that equation \eqref{ED-Lpm} is equivalent to the system
	\begin{equation}\label{ED-Lpm-decomposed}
	 \left\{\begin{array}{ll} A_{11}h'+A_{12}v'=0,\\
	A_{21}h'+A_{22}v'=Q'_{22}(u^\pm)v. \end{array}\right.
	\end{equation}
	We note that Hypothesis (H5) holds if and only if the linear operator $A_{11}$ is invertible on $\bV^\perp$.
	Integrating the first equation, we obtain that solutions $u=(h,v)$ that decay to $0$ at $\pm\infty$, satisfy the conditions
	\begin{equation}\label{ED-Lpm-decomposed2}
	 \left\{\begin{array}{ll} h=-A_{11}^{-1}A_{12}v,\\
	A_{21}h'+A_{22}v'=Q'_{22}(u^\pm)v. \end{array}\right.
	\end{equation}
	To prove that equations \eqref{ED-Lpm} have an exponential dichotomy on a complement of $\bV^\perp$ it is enough to show that equation
	\begin{equation}\label{ED-reduced}
	(A_{22}-A_{21}A_{11}^{-1}A_{12})v'=Q'_{22}(u^\pm)v
	\end{equation}
	is equivalent to an equation of the form $u'=Su$, where the linear operator $S$ generates a stable bi-semigroup on $\bV$. In the following lemma we show that the reduced equation
	\eqref{ED-reduced} can be treated
	%
	similarly to
	equation \eqref{ED-etaLpm} since their operator valued, constant coefficients share many properties. This lemma will allow us
	to treat these two equations in
	a unified way.
	\begin{lemma}\label{reduced-similar} Assume Hypotheses (H1), (H3), (H4) and (H5). Then, the following assertions hold true:
	\begin{enumerate}
	\item[(i)] The linear operator $\tilde{A}:=A_{22}-A_{21}A_{11}^{-1}A_{12}$ is self-adjoint and one-to-one;
	\item[(ii)] The linear operator $Q'_{22}(u^\pm)$ is invertible with bounded inverse on $\bV$;
	\item[(iii)] The linear operator $2\pi\rmi\omega\tilde{A}-Q'_{22}(u^\pm)$ is invertible on $\bV$ for any $\omega\in\RR$;
	\item[(iv)] $\sup_{\omega\in\RR}\|(2\pi\rmi\omega\tilde{A}-Q'_{22}(u^\pm))^{-1}\|<\infty$.
	\end{enumerate}
	\end{lemma}
	\begin{proof}
	(i) Since the linear operator $A$ is self-adjoint, from decomposition \eqref{decomp-AQ}, we obtain that $A_{11}^*=A_{11}$, $A_{22}^*=A_{22}$ and $A_{12}^*=A_{21}$, which implies that
	$\tilde{A}$ is self-adjoint. To show that $\tilde{A}$ is one-to-one, we consider $v\in\ker\tilde{A}$ and denote by $h=-A_{11}^{-1}A_{12}v\in\bV^\perp$. Using again the decomposition \eqref{decomp-AQ}, one can readily check that $A(h+v)=0$. From Hypothesis (H4) we infer that $h=-v$. Since $v\in\bV$ and $h\in\bV^\perp$ we conclude that $v=0$, proving (i).

	Assertion (ii) follows immediately from Hypothesis (H3) since $Q'_{22}(u^\pm)\leq -\delta_\pm I_{\bV}$. To prove (iii) and (iv), we note that since $\tilde{A}$ and $Q'_{22}(u^\pm)$ are self-adjoint we have that
	$\mathrm{Re}\tcL(\omega)=-Q'_{22}(u^\pm)$, where $\tcL(\omega):=2\pi\rmi\omega\tilde{A}-Q'_{22}(u^\pm)$. We obtain that
	\begin{equation}\label{r-s-1}
	\mathrm{Re}\langle\tcL(\omega)v,v\rangle=-\langle Q'_{22}(u^\pm)v,v\rangle\geq\delta_\pm\|v\|^2\quad\mbox{for any}\quad\omega\in\RR, v\in\bV,
	\end{equation}
	which implies that
	\begin{equation}\label{r-s-2}
	\|\tcL(\omega)v\|\geq\delta_\pm\|v\|\quad\mbox{for any}\quad\omega\in\RR, v\in\bV.
	\end{equation}
	It follows that $\tcL(\omega)$ is one-to-one and $\im\tcL(\omega)$ is closed in $\bV$ for any $\omega\in\RR$. Moreover, from \eqref{r-s-1} we infer that $\ker\tcL(\omega)^*=\{0\}$ for any $\omega\in\RR$, proving (iii). Assertion (iv) is a consequence of \eqref{r-s-2}.
	\end{proof}
	In what follows we focus our attention on equations of the form
	\begin{equation}\label{general-2}
	\Gamma u'=Eu,
	\end{equation}
	on some Hilbert space $\bX$, where the linear operators $\Gamma$ and $E$ satisfy the following Hypothesis (S) below. Our goal is to prove that equation \eqref{general-2} is equivalent to an equation of the form $u'=S_{\Gamma,E}u$, where the linear operator $S_{\Gamma,E}$ generates a \textit{stable bi-semigroup}.

	\vspace{0.2cm}
	{\bf Hypothesis (S)} We assume that bounded linear operators $\Gamma,E\in\cB(\bX)$ satisfy the following conditions:
	\begin{enumerate}
	\item[(i)] $\Gamma$ is self-adjoint and one-to-one;
	\item[(ii)] The linear operator $E$ is self-adjoint and invertible with bounded inverse on $\bX$;
	\item[(iii)] The linear operator $2\pi\rmi\omega\Gamma-E$ is invertible on $\bX$ for any $\omega\in\RR$;
	\item[(iv)] $\sup_{\omega\in\RR}\|(2\pi\rmi\omega\Gamma-E)^{-1}\|<\infty$.
	\end{enumerate}
	\vspace{0.2cm}
	\begin{remark}\label{r3.1} Assume Hypotheses (S). Then, the linear operator
	\begin{equation}\label{def-Spm}
	S_{\Gamma,E}=\Gamma^{-1}E:\dom(S_{\Gamma,E})=\{u\in\bX:Eu\in\im \Gamma\}\to\bX,
	\end{equation}
	is closed and densely defined. Indeed, since the linear operators $\Gamma$ and $E$ are bounded, one can readily check that the graph of $S_{\Gamma,E}$ is a closed subspace of $\bX\times\bX$. Next, we show that the domain of $S_{\Gamma,E}$ is dense. We note that from Hypothesis (S)(i) we have that the linear operator $\Gamma$ is self-adjoint and $\ker \Gamma=\{0\}$, which proves that $\im \Gamma$ is a dense subspace of $\bX$.
	Fix $u\in\bH$. Since $\im \Gamma$ is dense in $\bX$ it follows that there exists $\{w_n\}_{n\geq 1}$ a sequence of elements of $\bX$ such that $\Gamma w_n\to Eu$ as $n\to\infty$. We define the sequence $\{u_n\}_{n\geq 1}$ by $u_n=E^{-1}\Gamma w_n$. We note that $u_n\in\dom(S_{\Gamma,E})$ for all $n\geq 1$. Moreover, from Hypothesis (S) (ii) we have that the linear operator $E$ is invertible with bounded inverse on $\bX$, therefore, we infer that $u_n\to u$ as $n\to\infty$, proving that the domain of $S_{\Gamma,E}$ is dense in $\bX$.
	\end{remark}
	\begin{lemma}\label{l3.2}
	Assume Hypothesis (S). Then, the linear operator $S_{\Gamma,E}$ defined in \eqref{def-Spm} generates a bi-semigroup on $\bX$.
	\end{lemma}
	\begin{proof}
	Since $\Gamma\in\cB(\bX)$ is self-adjoint and one-to-one, we have that the linear operator $\Gamma$ is similar to a multiplication operator on an $L^2$ space. Therefore, there exists $\{\bX_n\}_{n\geq 1}$ a sequence of closed subspaces of $\bX$, invariant under $\Gamma$, such that
	\begin{equation}\label{3.2-1}
	\bX=\bigoplus_{n=1}^{\infty}\bX_n,\; \Gamma_{|\bX_n}\;\mbox{is invertible with bounded inverse on}\; \bX_n
	\end{equation}
	for any $n\geq 1$. Since $E$ is invertible with bounded inverse, we conclude that for any $n\geq1$ the subspace $\bG_n=E^{-1}\bX_n$ is closed in $\bX$. Moreover, from \eqref{3.2-1} we have that $\bX=\bigoplus_{n=1}^{\infty}\bG_n$. Since $E\bG_n=\bX_n$ and $\Gamma_{|\bX_n}^{-1}$ is bounded for any $n\geq 1$ we have that
	\begin{equation}\label{3.2-2}
	(S_{\Gamma,E})_{|\bG_n}=\Gamma^{-1}E_{|\bG_n}=\Gamma_{|\bX_n}^{-1}E_{|\bG_n}
	\end{equation}
	is bounded. Since $\bX=\bigoplus_{n=1}^{\infty}\bG_n$ we infer that $S_{\Gamma,E}$ is generating a bi-semigroup on $\bX$, proving the lemma.
	\end{proof}
	To prove that the bi-semigroup generated by $S_{\Gamma,E}$ is stable, we use the methods from \cite{LP2}.
	In the next lemma we prove that $S_{\Gamma,E}$ is hyperbolic and the basic estimates satisfied by the norm of the resolvent operators. To formulate the lemma, we introduce the operator valued function $\cL_{\Gamma,E}:\RR\to\cB(\bX)$ defined by $\cL_{\Gamma,E}(\omega)=2\pi\rmi\omega\Gamma-E$.
	\begin{lemma}\label{l3.3}
	Assume Hypothesis (S). Then, the following assertions hold true:
	\begin{enumerate}
	\item[(i)] $\rmi\RR\subseteq\rho(S_{\Gamma,E})$ and $R(2\pi\rmi\omega,S_{\Gamma,E})=(2\pi\rmi\omega-S_{\Gamma,E})^{-1}=\big(\widehat{\cL_{\Gamma,E}}(\omega)\big)^{-1}\Gamma$ for all $\omega\in\RR$;
	\item[(ii)] There exists $c>0$ such that $\|R(2\pi\rmi\omega,S_{\Gamma,E})\|\leq \frac{c}{1+|\omega|}$ for all $\omega\in\RR$.
	\end{enumerate}
	\end{lemma}
	\begin{proof} Assertion (i) follows from Hypothesis (S)(iii) and the definition of $S_{\Gamma,E}$ in \eqref{def-Spm}. Indeed, since $\Gamma S_{\Gamma,E} u=E u$ for any $u\in\dom(S_{\Gamma,E})$ one readily checks that
	$\big(\cL_{\Gamma,E}(\omega)\big)^{-1}\Gamma(2\pi\rmi\omega-S_{\Gamma,E})u=u$ for any $u\in\dom(S_{\Gamma,E})$. Moreover, since the linear operators $\Gamma$ and $E$ are bounded, we have that
	$E(2\pi\rmi\omega \Gamma-E)^{-1}=2\pi\rmi\omega \Gamma(2\pi\rmi\omega \Gamma-E)^{-1}-I$, which implies that
	\begin{equation}\label{3.3-1}
	E\big(\cL_{\Gamma,E}(\omega)\big)^{-1}\Gamma u=E(2\pi\rmi\omega \Gamma-E)^{-1}\Gamma u=2\pi\rmi\omega \Gamma(2\pi\rmi\omega \Gamma-E)^{-1}\Gamma u-\Gamma u\in\im \Gamma
	\end{equation}
	for any $u\in\bX$. It follows that $\big(\cL_{\Gamma,E}(\omega)\big)^{-1}\Gamma u\in\dom(S_{\Gamma,E})$ and $$S_{\Gamma,E}(\cL_{\Gamma,E}(\omega)\big)^{-1}\Gamma u=2\pi\rmi\omega (2\pi\rmi\omega \Gamma-E)^{-1}\Gamma u-u$$ for any $u\in\bX$, proving (i).

	\noindent\textit{Proof of (ii).} Using the same argument used to prove the resolvent equation, one can show that
	\begin{equation}\label{3.3-2}
	(2\pi\rmi\omega_1\Gamma-E)^{-1}-(2\pi\rmi\omega_2\Gamma-E)^{-1}=2\pi\rmi(\omega_2-\omega_1)(2\pi\rmi\omega_1\Gamma-E)^{-1}\Gamma(2\pi\rmi\omega_2\Gamma-E)^{-1}
	\end{equation}
	for any $\omega_1,\omega_2\in\RR$. Setting $\omega_2=0$ and multiplying the equation by $E$ from the right we obtain that
	\begin{equation}\label{3.3-3}
	(2\pi\rmi\omega \Gamma-E)^{-1}E+I=2\pi\rmi\omega(2\pi\rmi\omega\Gamma-E)^{-1}\Gamma=2\pi\rmi\omega R(2\pi\rmi\omega,S_{\Gamma,E})
	\end{equation}
	for any $\omega\in\RR$. Assertion (ii) follows
	readily
	from Hypothesis(S)(iv) and \eqref{3.3-3}.
	\end{proof}
	Next, we define $\cS_{\Gamma,E}:\dom(\cS_{\Gamma,E})\subset L^2(\RR,\bX)\to L^2(\RR,\bX)$ by $(\cS_{\Gamma,E} u)(\tau)=u'(\tau)-S_{\Gamma,E} u(\tau)$ with its natural, maximal domain given by
	\begin{equation}\label{def-cSpm}
	\dom(\cS_{\Gamma,E})=\big\{u\in H^1(\RR,\bX):u(\tau)\in\dom(S_{\Gamma,E})\;\mbox{for a.e.}\; \tau\in\RR, u'-S_{\Gamma,E} u\in L^2(\RR,\bX)\big\}.
	\end{equation}
	To prove that the linear operator $S_{\Gamma,E}$ generates a stable bi-semigroup, we need to prove that the linear operator $\cS_{\Gamma,E}$ defined in \eqref{def-cSpm} is invertible on $L^2(\RR,\bX)$. To be able to study the properties of the operator $\cS_{\Gamma,E}$, we start by describing its domain using a frequency domain formulation.
	To prove the next lemma, we recall a classical result on the commutative properties of the Fourier transform with closed linear operators.
	\begin{remark}\label{r3.4}
	Assume $T:\dom(T)\subset\bX\to\bX$ is a closed linear operator. Then, the following assertions hold true:
	\begin{enumerate}
	\item[(i)] If $J\subseteq\RR$ is a measurable set, $f\in L^1(J,\bX)$ such that $f(\tau)\in\dom(T)$ for almost all $\tau\in J$ and $Tf(\cdot)\in L^1(J,\bX)$, then
	$\int_J f(\tau)\rmd\tau\in\dom(T)$ and $T\int_J f(\tau)\rmd\tau=\int_J Tf(\tau)\rmd\tau$;
	\item[(ii)] If $f\in L^2(\RR,\bX)$ such that $f(\tau)\in\dom(T)$ for almost all $\tau\in \RR$ and $Tf(\cdot)\in L^2(\RR,\bX)$, then
	$\widehat{v}(\omega)\in\dom(T)$ and $T\widehat{v}(\omega)=\widehat{Tv}(\omega)$ for almost all $\tau\in J$.
	\end{enumerate}
	\end{remark}
	\begin{lemma}\label{l3.5}
	Assume Hypothesis (S). Then, the domain of $\cS_{\Gamma,E}$ is equal to the set of all $u\in L^2(\RR,\bX)$ for which there exists $f\in L^2(\RR,\bX)$ such that
	\begin{equation}\label{dom-cSpm}
	\widehat{u}(\omega)=R(2\pi\rmi\omega,S_{\Gamma,E})\widehat{f}(\omega)\quad\mbox{for almost all}\quad\omega\in\RR.
	\end{equation}
	In the case the above equation holds then $\cS_{\Gamma,E}u=f$.
	\end{lemma}
	\begin{proof} To start the proof, we denote by $\cD_{\Gamma,E}$ the set of all $u\in L^2(\RR,\bX)$ for which there exists $f\in L^2(\RR,\bX)$ such that \eqref{dom-cSpm} holds. First, we show that
	$\dom(\cS_{\Gamma,E})\subseteq\cD_{\Gamma,E}$. Let $u\in\dom(\cS_{\Gamma,E})$ and $f=\cS_{\Gamma,E} u=u'-S_{\Gamma,E} u\in L^2(\RR,\bX)$. From \eqref{def-cSpm} we have that $u\in H^1(\RR,\bX)$, which implies that $I_{\RR}\widehat{u}\in L^2(\RR,\bX)$. Here, we recall the notation $I_{\RR}$ as the identity function on $\RR$. In addition, we have that $\Gamma f=\Gamma u'-Eu$. It follows that
	\begin{equation}\label{3.5-1}
	\Gamma\widehat{f}(\omega)=(2\pi\rmi\omega \Gamma-E)\widehat{u}(\omega)=\cL_{\Gamma,E}(\omega)\widehat{u}(\omega)\quad\mbox{for almost all}\quad\omega\in\RR.
	\end{equation}
	From Hypothesis(S)(iii) we have that we can solve \eqref{3.5-1} for $\widehat{u}$. Moreover, from Lemma~\ref{l3.3}(i) we obtain that
	\begin{equation}\label{3.5-2}
	\widehat{u}(\omega)=\big(\cL_{\Gamma,E}(\omega)\big)^{-1}\Gamma\widehat{f}(\omega)=R(2\pi\rmi\omega,S_{\Gamma,E})\widehat{f}(\omega)\quad\mbox{for almost all}\quad\omega\in\RR,
	\end{equation}
	proving that $\dom(\cS_{\Gamma,E})\subseteq\cD_{\Gamma,E}$. To prove the converse inclusion, we consider $u\in\dom(\cS_{\Gamma,E})$ and $f\in L^2(\RR,\bX)$ such that \eqref{dom-cSpm} holds true. From Lemma~\ref{l3.3}(ii) we have that
	\begin{equation}\label{3.5-3}
	\|2\pi\rmi\omega\widehat{u}(\omega)\|\leq 2\pi|\omega| \|R(2\pi\rmi\omega,S_{\Gamma,E})\|\,\|\widehat{f}(\omega)\|\leq\frac{c|\omega|}{1+|\omega|}\|\widehat{f}(\omega)\|\leq c\|\widehat{f}(\omega)\|
	\end{equation}
	for almost all $\omega\in\RR$. Since $f,\widehat{f}\in L^2(\RR,\bX)$ we conclude that $I_{\RR}\widehat{u}\in L^2(\RR,\bX)$, which implies that $u\in H^1(\RR,\bX)$. Given that $u$ and $f$ satisfy equation \eqref{dom-cSpm}, we obtain that $\widehat{u}(\omega)\in\dom(S_{\Gamma,E})$ for almost all $\omega\in\RR$ and
	\begin{equation}\label{3.5-4}
	S_{\Gamma,E}\widehat{u}(\omega)=2\pi\rmi\omega\widehat{u}(\omega)-(2\pi\rmi\omega-S_{\Gamma,E})\widehat{u}(\omega)=\widehat{u'-f}(\omega)\quad\mbox{for almost all}\quad\omega\in\RR.
	\end{equation}
	Since $u,\widehat{u'-f}\in L^2(\RR,\bX)$ from Remark~\ref{r3.4}(ii) and \eqref{3.5-4} we infer that $u(\tau)\in\dom(S_{\Gamma,E})$ for almost all $\tau\in\RR$ and $u'-S_{\Gamma,E} u=f\in L^2(\RR,\bX)$. It follows that $u\in\dom(\cS_{\Gamma,E})$ and $\cS_{\Gamma,E} u=f$, proving that $\cD_{\Gamma,E}\subseteq\dom(\cS_{\Gamma,E})$ and the lemma.
	\end{proof}
	Now we can prove the one of the main results of this section, the invertibility of $\cS_{\Gamma,E}$.
	\begin{theorem}\label{t3.6}
	Assume Hypothesis (S). Then, the following assertions hold true:
	\begin{enumerate}
	\item[(i)] The linear operator $\cS_{\Gamma,E}$ is a closed, densely defined linear operator on $L^2(\RR,\bX)$;
	\item[(ii)] The linear operator $\cS_{\Gamma,E}$ is invertible with bounded inverse and $\cS_{\Gamma,E}^{-1}=\cF^{-1}M_{R(2\pi\rmi\cdot,S_{\Gamma,E})}\cF$. Here $\cF$ denotes the Fourier Transform and $M_{R(2\pi\rmi\cdot,S_{\Gamma,E})}$ denotes the operator of multiplication by the operator valued function $R(2\pi\rmi\cdot,S_{\Gamma,E})$ on $L^2(\RR,\bX)$;
	\item[(iii)] The bi-semigroup generated by $S_{\Gamma,E}$ is exponentially stable.
	\end{enumerate}
	\end{theorem}
	\begin{proof} (i) First, we note that the $R(2\pi\rmi\omega,S_{\Gamma,E})$ is a bounded linear operator on $\bX$ for any $\omega\in\RR$. Moreover, from Lemma~\ref{l3.3}(ii) we have that
	the operator valued function $R(2\pi\rmi\cdot,S_{\Gamma,E})$ is bounded on $\RR$. From \eqref{dom-cSpm} we infer that $\cS_{\Gamma,E}$ is a closed linear operator. To prove that the domain of the
	linear operator $\cS_{\Gamma,E}$ is dense in $L^2(\RR,\bX)$, we fix $h\in\big(\dom(\cS_{\Gamma,E})\big)^\perp$ and $f\in L^2(\RR,\bX)$.

	Next, we define $u_f=\cF^{-1}M_{R(2\pi\rmi\cdot,S_{\Gamma,E})}f$. One can readily check that $u_f\in\dom(\cS_{\Gamma,E})$. From \eqref{dom-cSpm} we obtain that
	\begin{align}\label{3.6-1}
	\langle M_{R(2\pi\rmi\cdot,S_{\Gamma,E})^*}\cF h,f\rangle_{L^2(\RR,\bX)}&=\langle \cF h,\cF\cF^{-1}M_{R(2\pi\rmi\cdot,S_{\Gamma,E})}f\rangle_{L^2(\RR,\bX)}\nonumber\\
	&=\langle \cF h,\cF u_f\rangle_{L^2(\RR,\bX)}=\langle h,u_f\rangle_{L^2(\RR,\bX)}=0.
	\end{align}
	Given that \eqref{3.6-1} holds true for any $f\in L^2(\RR,\bX)$ we conclude that $M_{R(2\pi\rmi\cdot,S_{\Gamma,E})^*}\cF h=0$. Since the $R(2\pi\rmi\omega,S_{\Gamma,E})^*$ is an one-to-one linear operator on $\bX$ and the Fourier Transform is invertible on $L^2(\RR,\bX)$, we infer that $h=0$, proving (i).

	\noindent\textit{Proof of (ii).} From equation \eqref{dom-cSpm} we have that $\cF u=M_{R(2\pi\rmi\cdot,S_{\Gamma,E})}\cF\cS_{\Gamma,E} u$ for any $u\in\dom(\cS_{\Gamma,E})$. Since the Fourier Transform is invertible on $L^2(\RR,\bX)$, and the operator valued function $R(2\pi\rmi\cdot,S_{\Gamma,E})$ is bounded on $\RR$, one can readily check that $\cS_{\Gamma,E}$ is invertible with bounded inverse and $\cS_{\Gamma,E}^{-1}=\cF^{-1}M_{R(2\pi\rmi\cdot,S_{\Gamma,E})}\cF$.

	Assertion (iii) is a direct consequence of (ii) and the main result in \cite{LP2}.
	\end{proof}
	To conclude this section, we use Theorem~\ref{t3.6} to prove that equations \eqref{ED-etaLpm} and \eqref{ED-reduced} exhibit exponential dichotomy on $\bH$ and $\bV$, respectively. We recall the definition of the linear operators
	\begin{equation}\label{def-generator-eta-reduced}
	S^\eta_\pm:=A^{-1}\big(Q'(u^\pm)\pm\eta A\big),\quad  S^{\mathrm{r}}_\pm=\tilde{A}^{-1}Q_{22}'(u^\pm).
	\end{equation}

	\noindent{\bf Proof of Theorem 1.2.}
	From Hypothesis (H1) and the results proved in Section~\ref{s2} we infer that the linear operators $A$ and $Q'(u^\pm)\pm\eta A$ satisfy Hypothesis (S). Indeed, the conditions listed in Hypothesis (S)(i)-(iv) follow, respectively, from
	Hypothesis (H1), Lemma~\ref{l2.5}, Lemma~\ref{l2.7}(ii) and Lemma~\ref{l2.8}. Similarly, from Lemma~\ref{reduced-similar} we have that the linear operators $\tilde{A}=A_{22}-A_{21}A_{11}^{-1}A_{12}$ and $Q_{22}'(u^\pm)$ satisfy Hypothesis (S). Assertions (i) and (ii) follow directly from Theorem~\ref{t3.6}(iii). Assertion (iii) is a direct consequence of (i). Since equation \eqref{ED-Lpm} is equivalent to the system \eqref{ED-Lpm-decomposed2}, we infer that assertion (iv) follows
	readily
	from (ii). Moreover, if we denote the stable/unstable spaces of equation \eqref{ED-reduced} by $\bV^{\mathrm{s/u}}_\pm$, then the stable/unstable subspaces of equation \eqref{ED-Lpm} are given by the formula
	\begin{equation}\label{subspaces-ED-Lpm}
	\bH^{\mathrm{s/u}}_\pm=\big\{(h,v)\in\bV^\perp\oplus\bV:h=-A_{11}^{-1}A_{12}v,\;v\in\bV^{\mathrm{s/u}}_\pm\big\}.
	\end{equation}
	One can readily check that $\bH=\bV^\perp\oplus \bH^{\mathrm{s}}_\pm\oplus \bH^{\mathrm{u}}_\pm$, proving the corollary.

	\section{Solutions of general steady relaxation systems}\label{s4}
	In this section we analyze the qualitative properties of solutions of the steady equation
	\begin{equation}\label{Relax-Sys}
	Au_\tau=Q(u)
	\end{equation}
	in $\bH$ satisfying $\lim_{\tau\to\pm\infty}u(\tau)=u^\pm$ and its linearization along $u^\pm$. In particular, we are
	interested
	in describing
	the smoothness properties of these solutions. Also, it is interesting to consider all of these equations on $\RR_\pm$, respectively. Making the change of variable
	$w^\pm(\tau)=u(\tau)-u^\pm$ in \eqref{Relax-Sys} we obtain the equations
	\begin{equation}\label{Relax-Sys-pm}
	Aw^\pm_\tau(\tau)=2B(u^\pm,w^\pm(\tau))+Q(w^\pm(\tau)).
	\end{equation}
	Here, we recall that $Q(u)=B(u,u)$ is bilinear, symmetric, continuous on $\bH$. Moreover, since the range of the bilinear map $B$ is contained in $\bV$, denoting by $h^\pm=P_{\bV^\perp}w^\pm$ and $v^\pm=P_{\bV}w^\pm$, we obtain that equation \eqref{Relax-Sys-pm} is equivalent to the system
	\begin{equation}\label{Relax-Sys-pm2}
	 \left\{\begin{array}{ll} A_{11}h^\pm_\tau(\tau)+A_{12}v^\pm_\tau(\tau)=0,\\
	A_{21}h^\pm_\tau(\tau)+A_{22}v^\pm_\tau(\tau)=Q'_{22}(u^\pm)v^\pm(\tau)+Q(h^\pm(\tau)+v^\pm(\tau)). \end{array}\right.
	\end{equation}
	Integrating the first equation and using that $\lim_{\tau\to\pm\infty}w^\pm(\tau)=0$, we obtain that solutions $w^\pm=(h^\pm,v^\pm)$ of \eqref{Relax-Sys-pm2}
	satisfy the condition $h^\pm=-A_{11}^{-1}A_{12}v^\pm$. Plugging in the second of \eqref{Relax-Sys-pm2} we obtain that to prove the existence of a center-stable manifold around the equilibria $u^\pm$, it is enough to prove the existence of a center-stable manifold around equilibria $P_{\bV}u^\pm$ of equation
	\begin{equation}\label{Relax-Reduced}
	(A_{22}-A_{21}A_{11}^{-1}A_{12})v^\pm_\tau(\tau)=Q'_{22}(u^\pm)v^\pm(\tau)+Q\big(v^\pm(\tau)-A_{11}^{-1}A_{12}v^\pm(\tau)\big).
	\end{equation}
	We note that it is especially important to study the solutions of equations \eqref{Relax-Sys-pm2} and \eqref{Relax-Reduced} close to $\pm\infty$, therefore we focus our attention on their solutions on $\RR_\pm$, rather than the entire line. To study these equations we use the properties of stable bi-semigroups. We recall that if a linear operator $S$ generates a stable bi-semigroup, then the linear operator $-S$ generates a stable bi-semigroup as well. Making the change of variables $\tau\to-\tau$ in \eqref{Relax-Reduced}, we obtain an equation that can be handled in the same way as the original equation, as shown in \cite[Section 4]{LP2}. Therefore, to understand the limiting properties of solutions of equations \eqref{Relax-Reduced} at $\pm\infty$, we need to understand the limiting properties in equations of the form
	\begin{equation}\label{h-plus}
	\Gamma \bu_\tau(\tau)=E\bu(\tau)+D(\bu(\tau),\bu(\tau)),\quad\tau\in\RR_+.
	\end{equation}
	Here the pair of bounded linear operators $(\Gamma,E)$ on a Hilbert space $\bX$ satisfies Hypothesis (S) and $D:\bX\times\bX\to\bX$ is a bounded, bi-linear map. In addition, the linear operator $E$ is negative definite.

	In what follows we denote by $R_{\Gamma,E}:\RR\to\cB(\bX)$ the operator valued function $R_{\Gamma,E}(\omega)=(2\pi\rmi\omega \Gamma-E)^{-1}$. The bi-semigroup generated by $S_{\Gamma,E}=\Gamma^{-1}E$ on $\bX$ is denoted by
	$\{T_{\rms/\rmu}^{\Gamma,E}(\tau)\}_{\tau\geq 0}$ on $\bX_{\rms/\rmu}^{\Gamma,E}$, the stable/unstable subspaces of $\bX$ invariant under $\Gamma^{-1}E$. In addition, we denote by $P_{\rms/\rmu}^{\Gamma,E}$ the projections onto $\bX_{\rms/\rmu}^{\Gamma,E}$ parallel to $\bX_{\rmu/\rms}^{\Gamma,E}$, associated to the decomposition $\bX=\bX_\rms^{\Gamma,E}\oplus\bX_\rmu^{\Gamma,E}$. In the sequel we use that the function $R_{\Gamma,E}$ satisfies
	\begin{equation}\label{char-eq}
	R_{\Gamma,E}(\omega_1)-R_{\Gamma,E}(\omega_2)=2\pi\rmi(\omega_2-\omega_1)R_{\Gamma,E}(\omega_1)\Gamma R_{\Gamma,E}(\omega_2)\quad\mbox{for all}\quad\omega_1,\omega_2\in\RR.
	\end{equation}
	A first step towards understanding equation \eqref{h-plus}, is to study the perturbed equation
	\begin{equation}\label{pert-h-plus}
	\Gamma \bu_\tau(\tau)=E\bu(\tau)+f(\tau),\quad\tau\in\RR_+,
	\end{equation}
	for some function $f\in L^1_{\mathrm{loc}}(\RR_+,\bX)$ or $f\in L^2_{\mathrm{loc}}(\RR_+,\bX)$. For a function
	$g$ defined on a proper subset of $\RR$ we keep the same notation $g$ to denote its extension to $\RR$ by $0$.
	\begin{definition}\label{d3.1}
	We  say that
	\begin{enumerate}
	\item[(i)] The function $\bu$ is a \textit{smooth} solution of \eqref{pert-h-plus} on $[\tau_0,\tau_1]$ if $\bu\in H^1([\tau_0,\tau_1],\bX)$ satisfies \eqref{pert-h-plus};
	\item[(ii)] The function $\bu$  is a \textit{mild} solution of \eqref{pert-h-plus} on $[\tau_0,\tau_1]$ if $\bu\in L^2([\tau_0,\tau_1],\bX)$ satisfies
	\begin{equation}\label{mild}
	 \widehat{\bu}(\omega)=R(2\pi\rmi\omega,S_{\Gamma,E})\big[e^{-2\pi\rmi\omega\tau_0}\bu(\tau_0)-e^{-2\pi\rmi\omega\tau_1}\bu(\tau_1)\big]+R_{\Gamma,E}(\omega)\widehat{f_{|[\tau_0,\tau_1]}}(\omega)\;\mbox{for almost all}\;\omega\in\RR;
	\end{equation}
	\item[(iii)]  The function $\bu$  is a \textit{smooth/mild} solution of \eqref{pert-h-plus} on $\RR_+$ of \eqref{pert-h-plus} if $\bu$ is a \textit{smooth/mild} solution of \eqref{pert-h-plus} on $[0,\tau_1]$ for any $\tau_1>0$.
	\end{enumerate}
	\end{definition}
	Our definition of mild solutions follows \cite[Section 2]{LP2}, where it is shown that the frequency domain
	reformulation given in \eqref{mild} is much easier to handle than the classical approach where one defines the mild solution by simply integrating equation \eqref{pert-h-plus}. We note that by taking Fourier transform in \eqref{pert-h-plus} and integrating by parts, it is easy to verify that smooth solutions of equation
	are also mild solutions.
	\begin{remark}\label{r4.2} Denoting by $\cG_{\Gamma,E}:\RR\to\cB(\bX)$ the Green function defined by
	\begin{equation}\label{Green}
	\cG_{\Gamma,E}(\tau)=\left\{\begin{array}{l l}
	T^{\Gamma,E}_\rms(\tau)P^{\Gamma,E}_\rms & \; \mbox{if $\tau\geq0$ }\\
	-T^{\Gamma,E}_\rmu(-\tau)P^{\Gamma,E}_\rmu & \; \mbox{if $\tau<0$}\\
	\end{array} \right.,
	\end{equation}
	the following assertions hold true
	\begin{enumerate}
	\item[(i)] There exist two positive constants $c$ and $\nu$ such that $\|\cG_{\Gamma,E}(\tau)\|\leq ce^{-\nu|\tau|}$ for any $\tau\in\RR$;
	\item[(ii)] $\cF\cG_{\Gamma,E}(\cdot)\bu=R(2\pi\rmi\cdot,S_{\Gamma,E})\bu$ for any $\bu\in\bX$.
	\end{enumerate}
	\end{remark}
	Next, we define the linear operator $\cK_{\Gamma,E}:L^2(\RR,\bX)\to L^2(\RR,\bX)$ by $\cK_{\Gamma,E}f=\cF^{-1}M_{R_{\Gamma,E}}\cF f$. Here we recall that $M_{R_{\Gamma,E}}$ denotes the multiplication operator on $L^2(\RR,\bX)$ by the operator valued function $R_{\Gamma,E}$. From Hypothesis (S)(iv) we have that $\sup_{\omega\in\RR}\|R_{\Gamma,E}(\omega)\|<\infty$, which proves that $\cK_{\Gamma,E}$ is well defined and bounded on $L^2(\RR,\bX)$.

	To prove our results we need to understand the properties of the Fourier multiplier defined by $\cK_{\Gamma,E}$. Our first goal in this section is to show that the definition we use for mild solutions of equation \eqref{pert-h-plus} can be seen as an extension of the classical variation of constants formula. To prove such a result we need to understand some of the smoothing properties of $\cK_{\Gamma,E}$.
	\begin{lemma}\label{l4.3}
	Assume Hypothesis (S). Then, we have that $\Gamma(\cK_{\Gamma,E}f)(\cdot)\in C_0(\RR,\bX)$ for any $f\in L^2(\RR,\bX)$.
	\end{lemma}
	\begin{proof} Let $f\in L^2(\RR,\bX)$ and $g=\cK_{\Gamma,E}f$. To prove the lemma we note that it is enough to show that $\widehat{\Gamma g}\in L^1(\RR,\bX)$. Using the definition of $\cK_{\Gamma,E}$ we have that
	\begin{equation}\label{4.3-1}
	\widehat{\Gamma g}(\omega)=\Gamma\widehat{g}(\omega)=\Gamma\widehat{\cK_{\Gamma,E}f}(\omega)=\Gamma R_{\Gamma,E}(\omega)\widehat{f}(\omega)\quad\mbox{for all}\quad \omega\in\RR.
	\end{equation}
	From Lemma~\ref{l3.3} and the definition of $E$ and its associated bi-semigroup, we have that
	\begin{equation}\label{4.3-2}
	\|\Gamma R_{\Gamma,E}(\omega)\|=\|R_{\Gamma,E}(\omega)^*\Gamma^*\|=\|R_{\Gamma,E}(-\omega)\Gamma\|=\|R(-2\pi\rmi\omega,S_{\Gamma,E})\|\leq \frac{c}{1+|\omega|}\quad\mbox{for all}\quad \omega\in\RR.
	\end{equation}
	From \eqref{4.3-1} and \eqref{4.3-2} we conclude that $\widehat{\Gamma g}\in L^1(\RR,\bX)$, proving the lemma.
	\end{proof}
	Now, we are ready to prove that \eqref{mild} is a generalization of the variation of constants formula.
	\begin{lemma}\label{l4.4}
	Assume Hypothesis (S). Let $f\in L^2(\RR_+,\bX)$ and let $\bu$ be a mild solution of \eqref{pert-h-plus}. Then, $\bu\in L^2(\RR_+,\bX)$ and $\Gamma\bu\in C_0(\RR_+,\bX)$ if and only if
	\begin{equation}\label{var-const}
	\bu(\tau)=T_\rms^{\Gamma,E}(\tau)P_\rms^{\Gamma,E}\bu(0)+(\cK_{\Gamma,E}f)(\tau)\quad\mbox{for all}\quad\tau\geq0;
	\end{equation}
	\end{lemma}
	\begin{proof} First, we prove that any mild solution $\bu\in L^2(\RR,\bX)$ of \eqref{pert-h-plus} satisfies equation \eqref{var-const}, provided $\Gamma\bu\in C_0(\RR_+,\bX)$ . Since
	$\chi_{[0,\tau_1]}\to \chi_{[0,\infty)}$ simple as $\tau_1\to\infty$, from the Lebesgue dominated Convergence theorem we obtain that $\bu\chi_{[0,\tau_1]}\to \bu$ and $f\chi_{[0,\tau_1]}\to f$ in $L^2(\RR_+,\bX)\hookrightarrow L^2(\RR,\bX)$ as $\tau_1\to\infty$. Since the linear operators $\cF$ and $\cK_{\Gamma,E}$ are continuous on $L^2(\RR,\bX)$ we conclude that
	\begin{equation}\label{4.4-1}
	\cF\big(\bu\chi_{[0,\tau_1]}-\cK_{\Gamma,E}(f\chi_{[0,\tau_1]})\big)\to \cF(\bu-\cK_{\Gamma,E}f)\;\mbox{in}\;L^2(\RR,\bX)\quad\mbox{as}\quad\tau_1\to\infty.
	\end{equation}
	Moreover, since $\bu$ is a solution of \eqref{pert-h-plus} on $[0,\tau_1]$ for all $\tau_1>0$ we have that
	\begin{equation}\label{4.4-2}
	\cF\big(\bu\chi_{[0,\tau_1]}-\cK_{\Gamma,E}(f\chi_{[0,\tau_1]})\big)(\omega)=R_{\Gamma,E}(\omega)\big( \Gamma\bu(0)-e^{-2\pi\rmi\omega\tau_1}\Gamma\bu(\tau_1)\big)\quad\mbox{for all}\quad \omega\in\RR.
	\end{equation}
	Since $\Gamma\bu\in C_0(\RR_+,\bX)$ from \eqref{4.4-2} it follows that
	\begin{equation}\label{4.4-3}
	\cF\big(\bu\chi_{[0,\tau_1]}-\cK_{\Gamma,E}(f\chi_{[0,\tau_1]})\big)(\omega)\to R_{\Gamma,E}(\omega)\Gamma\bu(0)\;\mbox{as}\;\quad\tau_1\to\infty,\quad\mbox{for all}\quad \omega\in\RR.
	\end{equation}
	From \eqref{4.4-1} and \eqref{4.4-3} we infer that $\cF(\bu-\cK_{\Gamma,E}f)(\omega)=R_{\Gamma,E}(\omega)\Gamma\bu(0)=R(2\pi\rmi\omega,S_{\Gamma,E})\bu(0)$ for almost all $\omega\in\RR$. Taking inverse Fourier transform, from Remark~\ref{r4.2}(ii) we obtain that
	\begin{equation}\label{4.4-4}
	\bu(\tau)=T_\rms^{\Gamma,E}(\tau)P_\rms^{\Gamma,E}\bu(0)+(\cK_{\Gamma,E}f)(\tau)\quad\mbox{for almost all}\quad\tau\geq0.
	\end{equation}
	Next, we prove that equality holds true in \eqref{4.4-4} for any $\tau\geq0$. Indeed, multiplying the equation by $\Gamma$ from the left, we obtain that
	$\Gamma\bu=\Gamma T_\rms^{\Gamma,E}(\cdot)P_\rms^{\Gamma,E}\bu(0)+\Gamma(\cK_{\Gamma,E}f)(\cdot)$ almost everywhere on $\RR_+$. Since $\Gamma\bu$ is continuous on $\RR_+$, $\{T_\rms^{\Gamma,E}(\tau)\}_{\tau\geq 0}$ is a strongly continuous semigroup, and from Lemma~\ref{l4.3} we have that $\Gamma(\cK_{\Gamma,E})(\cdot)$ is continuous, we infer that the equality $\Gamma\bu=\Gamma T_\rms^{\Gamma,E}(\cdot)P_\rms^{\Gamma,E}\bu(0)+\Gamma(\cK_{\Gamma,E}f)(\cdot)$ holds everywhere on $\RR_+$. Since $\Gamma$ is one-to-one on $\bX$, by Hypothesis (S)(i), it follows that equation \eqref{var-const} holds true.

	To finish the proof of lemma, we prove that under the assumption that $f\in L^2(\RR_+,\bX)$, any function $\bu$ satisfying equation \eqref{var-const} belongs to $L^2(\RR,\bX)$, $\Gamma\bu\in C_0(\RR_+,\bX)$ and is a mild solution of \eqref{pert-h-plus} on $[0,\tau_1]$ for any $\tau_1>0$. Indeed, since $\{T_\rms^{\Gamma,E}(\tau)\}_{\tau\geq 0}$ is a stable $C^0$-semigroup on $\bX$ and $\cK_{\Gamma,E}$ is well-defined and bounded on $L^2(\RR,\bX)$, one can readily check that $\bu$ belongs to $L^2(\RR_+,\bX)$. Moreover, from Lemma~\ref{l4.3} and \eqref{var-const} we conclude that $\Gamma\bu\in C_0(\RR_+,\bX)$.

	Let $\varphi\in C_0^\infty(\RR)$ a smooth, scalar function with compact support. Using the elementary properties of the Fourier transform and convolution, from  \eqref{char-eq} and \eqref{var-const} we obtain that
	\begin{align}\label{4.4-5}
	\widehat{\varphi \bu}(\omega)&-R_{\Gamma,E}(\omega)\widehat{\varphi' \Gamma\bu}(\omega)=(\widehat{\varphi}*\widehat{\bu})(\omega)-R_{\Gamma,E}(\omega)(\widehat{\varphi'}*\widehat{\Gamma\bu})(\omega)\nonumber\\
&=\int_{\RR}\widehat{\varphi}(\omega-\theta)\widehat{\bu}(\theta)\rmd\theta-R_{\Gamma,E}(\omega)\int_{\RR}2\pi\rmi(\omega-\theta)\widehat{\varphi}(\omega-\theta)\Gamma\widehat{\bu}(\theta)\rmd\theta\nonumber\\
	&=\int_{\RR}\widehat{\varphi}(\omega-\theta)\Big(I-2\pi\rmi(\omega-\theta)R_{\Gamma,E}(\omega)\Gamma\Big)\widehat{\bu}(\theta)\rmd\theta\nonumber\\
	&=\int_{\RR}\widehat{\varphi}(\omega-\theta)\Big(I-2\pi\rmi(\omega-\theta)R_{\Gamma,E}(\omega)\Gamma\Big)R_{\Gamma,E}(\theta)\Big(\Gamma\bu(0)+\widehat{f}(\theta)\Big)\rmd\theta\nonumber\\
	&=\int_{\RR}\widehat{\varphi}(\omega-\theta)\Big(R_{\Gamma,E}(\theta)-2\pi\rmi(\omega-\theta)R_{\Gamma,E}(\omega)\Gamma R_{\Gamma,E}(\theta)\Big)\Big(\Gamma\bu(0)+\widehat{f}(\theta)\Big)\rmd\theta\nonumber\\	 &=R_{\Gamma,E}(\omega)\Big(\int_{\RR}\widehat{\varphi}(\omega-\theta)\rmd\theta\Big)\Gamma\bu(0)+R_{\Gamma,E}(\omega)\int_{\RR}\widehat{\varphi}(\omega-\theta)\widehat{f}(\theta)\rmd\theta\nonumber\\
	&=\varphi(0)R_{\Gamma,E}(\omega)\Gamma\bu(0)+R_{\Gamma,E}(\omega)\widehat{\varphi f}(\omega)\quad\mbox{for any}\quad\omega\in\RR.
	\end{align}
	Fix $\tau_1>0$ and let $\{\varphi_n\}_{n\geq1}$ be a sequence of functions in  $C_0^\infty(\RR)$ with the following properties:
	$0\leq \varphi_n\leq 1$, $\|\varphi'_n\|_\infty\leq cn$, $\varphi_n(\tau) = 1$ for any $\tau\in [0, \tau_1-1/n]$ and $\varphi_n(\tau) = 0$
	for any $\tau\notin(-1/n, \tau_1)$. Since the function $\bu$ is defined on $\RR_+$ and is extended to $\RR$ by $0$, we conclude that
	\begin{align}\label{4.4-6}
	\widehat{\varphi_n' \Gamma\bu}(\omega)+e^{-2\pi\rmi\tau_1\omega}\Gamma\bu(\tau_1)&=\int_{\RR_+}e^{-2\pi\rmi\tau\omega}\varphi_n'(\tau)\Gamma\bu(\tau)\rmd\tau+e^{-2\pi\rmi\tau_1\omega}\Gamma\bu(\tau_1)\nonumber\\&=
	\int_{\tau_1-1/n}^{\tau_1}\varphi_n'(\tau)\Big(e^{-2\pi\rmi\tau\omega}\Gamma\bu(\tau)-e^{-2\pi\rmi\tau_1\omega}\Gamma\bu(\tau_1)\Big)\rmd\tau
	\end{align}
	for any $n \geq1$ and $\omega\in\RR$. Hence, the following estimate holds
	\begin{equation}\label{4.4-7}
	\|\widehat{\varphi_n' \Gamma\bu}(\omega)+e^{-2\pi\rmi\tau_1\omega}\Gamma\bu(\tau_1)\|\leq nc\int_{\tau_1-1/n}^{\tau_1}\big\|e^{-2\pi\rmi\tau\omega}\Gamma\bu(\tau)-e^{-2\pi\rmi\tau_1\omega}\Gamma\bu(\tau_1)\big\|\rmd\tau
	\end{equation}
	for any $n \geq1$ and $\omega\in\RR$. Since $\Gamma\bu$ is continuous on $\RR_+$, from \eqref{4.4-7} we infer that $\widehat{\varphi_n' \Gamma\bu}(\omega)\to-e^{-2\pi\rmi\tau_1\omega}\Gamma\bu(\tau_1)$ as $n\to\infty$ for any $\omega\in\RR$. Since $\varphi_n\to\chi_{[0,\tau_1)}$ pointwise as $n\to\infty$ and $0\leq \varphi_n\leq 1$, for any $n\geq 1$, from the Lebesgue Dominated Convergence Theorem we obtain that $\varphi_n\bu\to\chi_{[0,\tau_1]}\bu$ and $\varphi_nf\to\chi_{[0,\tau_1]}f$ in $L^2(\RR,\bX)$ as $n\to\infty$.
	Passing to the limit in \eqref{4.4-5} with $\varphi=\varphi_n$ we infer that
	\begin{equation}\label{4.4-8}
	 \widehat{\chi_{[0,\tau_1]}\bu}(\omega)+e^{-2\pi\rmi\tau_1\omega}R_{\Gamma,E}(\omega)\Gamma\bu(\tau_1)=R_{\Gamma,E}(\omega)\Gamma\bu(0)+R_{\Gamma,E}(\omega)\widehat{\chi_{[0,\tau_1]}f}(\omega)\quad\mbox{for any}\quad\omega\in\RR,
	\end{equation}
	which implies that \eqref{pert-h-plus} holds, proving the lemma.
	\end{proof}
	To better understand the solutions of equation \eqref{h-plus} we need to further study the Fourier multiplier $\cK_{\Gamma,E}$: in particular we are
	interested in finding
	suitable subspaces of $L^2(\RR,\bX)$ that are invariant  under $\cK_{\Gamma,E}$. To achieve this, we first use that the linear operators $\Gamma$ and $E$ are self-adjoint to describe the structure of the bi-semigroup generator $S_{\Gamma,E}$.
	\begin{lemma}\label{l4.5}
	Assume Hypothesis (S) and assume that the linear operator $E$ is negative-definite. Then, the bi-semigroup generator $S_{\Gamma,E}$ is similar to an operator of multiplication by some real-valued,  bounded from below, measurable function $H_{\Gamma,E}:\Lambda\to\RR$, on $L^2(\Lambda,\mu)$, where $(\Lambda,\mu)$ is some measure space.
	\end{lemma}
	\begin{proof} Since the linear operator $E$ is bounded, self-adjoint, invertible and negative-definite, we have that $\tE=(-E)^{\frac{1}{2}}$ is a bounded, self-adjoint, invertible linear operator on $\bX$. One can readily check that
	\begin{equation}\label{4.5-1}
	\tE S_{\Gamma,E}\tE^{-1}=\tE\Gamma^{-1}E\tE^{-1}=-\tE\Gamma^{-1}\tE.
	\end{equation}
	Since the linear operator $\Gamma$ and $\tE$ are self-adjoint, we obtain that the linear operator $\tE S_{\Gamma,E}\tE^{-1}$ is self-adjoint. It follows that the linear operator $\tE S_{\Gamma,E}\tE^{-1}$ is
	unitarily
	equivalent to an operator of multiplication on some $L^2$ space. Therefore, there exists a measure space $(\Lambda,\mu)$, a real-valued function $H_{\Gamma,E}:\Lambda\to\RR$ and a \textit{unitary}, bounded, linear operator $V_{\Gamma,E}:\bX\to L^2(\Lambda,\mu)$ such that $\tE S_{\Gamma,E}\tE^{-1}=V_{\Gamma,E}^{-1}M_{H_{\Gamma,E}}V_{\Gamma,E}$. It follows that
	\begin{equation}\label{generator-representation}
	S_{\Gamma,E}=U_{\Gamma,E}^{-1}M_{H_{\Gamma,E}}U_{\Gamma,E},\quad\mbox{where}\quad U_{\Gamma,E}=V_{\Gamma,E}\tE\in\cB(\bX,L^2(\Lambda,\mu)).
	\end{equation}
	To finish the proof of lemma, we need to prove that the function $H_{\Gamma,E}$ is bounded from below. From Theorem~\ref{t3.6} we have that the linear operator $S_{\Gamma,E}$ generates a stable bi-semigroup. From \eqref{generator-representation} we conclude that there exists $\nu=\nu(\Gamma,E)>0$ such that
	\begin{equation}\label{4.5-2}
	\sigma(S_{\Gamma,E})=\sigma(M_{H_{\Gamma,E}})=\{z\in\CC: |\mathrm{Re}z|\geq \nu\}.
	\end{equation}
	Since the spectrum of any multiplication operator on $L^2$ spaces is given by its essential range, we conclude that
	\begin{equation}\label{4.5-3}
	|H_{\Gamma,E}(\lambda)|\geq \nu\quad\mbox{for}\;\mu\;\mbox{almost all}\; \lambda\in \Lambda.
	\end{equation}
	The representation \eqref{generator-representation} holds true when we modify the function $H_{\Gamma,E}$ on a set of $\mu$-measure $0$, therefore we can assume from now on that the inequality \eqref{4.5-3} is true for any $\lambda\in \Lambda$.
	\end{proof}
	We note that the main idea used to obtain the representation \eqref{generator-representation} is based on the unitary equivalence of self-adjoint operators to multiplication operators, which is spectral in nature. Thus, it is natural to refer to functions in $L^2(\Lambda,\mu)$ as spectral components of the generator $S_{\Gamma,E}$.
	\begin{lemma}\label{r4.6}
	Assume Hypothesis (S) and assume that the linear operator $E$ is negative definite. Then, the following assertions hold true:
	\begin{enumerate}
	\item[(i)] The linear operators $U_{\Gamma,E}$ and $E$ satisfy the identity
	\begin{equation}\label{r4.6-1bis}
	 U_{\Gamma,E}E^{-1}U_{\Gamma,E}^*=-Id_{L^2(\Lambda,\mu)};
	\end{equation}
	\item[(ii)] The operator-valued function $R_{\Gamma,E}$ has the following representation
	\begin{equation}\label{r4.6-1}
	R_{\Gamma,E}(\omega)=E^{-1}S_{\Gamma,E}^*R(2\pi\rmi\omega, S_{\Gamma,E}^*)=E^{-1}U_{\Gamma,E}^*\tR_{\Gamma,E}(\omega)(U_{\Gamma,E}^*)^{-1}
	\end{equation}
	for any $\omega\in\RR$, where $\tR_{\Gamma,E}:\RR\to\cB(L^2(\Lambda,\mu))$ is given by
	\begin{equation}\label{r4.6-2}
	\Big(\tR_{\Gamma,E}(\omega)\tf\Big)(\lambda)=\frac{H_{\Gamma,E}(\lambda)}{2\pi\rmi\omega-H_{\Gamma,E}(\lambda)}\tf(\lambda),\;\tf\in L^2(\Lambda,\mu), \lambda\in \Lambda.
	\end{equation}
	\item[(iii)] The bi-semigroup generated by the linear operator $S_{\Gamma,E}$ has the representation
	\begin{equation}\label{representation-stable-unstable}
	\bX^{\Gamma,E}_\rms=U_{\Gamma,E}^{-1}L^2(\Lambda_-,\mu),\quad \bX^{\Gamma,E}_\rmu=U_{\Gamma,E}^{-1}L^2(\Lambda_+,\mu);
	\end{equation}
	\begin{equation}\label{representation-bi-semigroup}
	T^{\Gamma,E}_{\rms/\rmu}(\tau)=U_{\Gamma,E}^{-1}\tT^{\Gamma,E}_{\rms/\rmu}(\tau){U_{\Gamma,E}}_{\big|\bX^{\Gamma,E}_{\rms/\rmu}},\quad\mbox{for any}\quad\tau\geq 0.
	\end{equation}
	Here $\Lambda_\pm:=\{\lambda\in\Lambda:\pm H_{\Gamma,E}(\lambda)>0\}$ and the $C^0$-semigroups $\{\tT^{\Gamma,E}_{\rms/\rmu}(\tau)\}_{\tau\geq0}$ are defined by
	\begin{align}\label{def-tilde-T}
	\Big(\tT^{\Gamma,E}_{\rms}(\tau)\tf\Big)(\lambda)&=e^{\tau H_{\Gamma,E}(\lambda)}\tf(\lambda),\quad \lambda\in\Lambda_-, \tf\in L^2(\Lambda_-,\mu);\nonumber\\
	\Big(\tT^{\Gamma,E}_{\rmu}(\tau)\tf\Big)(\lambda)&=e^{-\tau H_{\Gamma,E}(\lambda)}\tf(\lambda),\quad \lambda\in\Lambda_+, \tf\in L^2(\Lambda_+,\mu).
	\end{align}
	\end{enumerate}
	\end{lemma}
	\begin{proof}
	(i) Since the linear operator $V_{\Gamma,E}\in\cB(\bX,L^2(\Lambda,\mu))$ are unitary, $\tE^2=-E$ and $U_{\Gamma,E}=V_{\Gamma,E}\tE$ one can readily check that
	\begin{equation}\label{r4.6-3}
	U_{\Gamma,E}E^{-1}U_{\Gamma,E}^*=V_{\Gamma,E}\tE E^{-1}(V_{\Gamma,E}\tE)^*=-(V_{\Gamma,E}\tE^{-1})\tE V_{\Gamma,E}^*=-V_{\Gamma,E}V_{\Gamma,E}^{-1}=-Id_{L^2(\Lambda,\mu)}.
	\end{equation}
	\noindent\textit{Proof of (ii)} Using the definition of the linear operator $S_{\Gamma,E}$ in \eqref{def-Spm}, from Lemma~\ref{l3.3}(i) and Hypothesis(S)(i)-(ii) we obtain that
	$R(2\pi\rmi\omega,S_{\Gamma,E}^*)=(2\pi\rmi\omega-S_{\Gamma,E}^*)^{-1}=\Gamma R_{\Gamma,E}(\omega)$ for any $\omega\in\RR$,
	which implies that
	\begin{equation}\label{r4.6-4}
	S_{\Gamma,E}^*(2\pi\rmi\omega-S_{\Gamma,E}^*)^{-1}=S_{\Gamma,E}^*\Gamma R_{\Gamma,E}(\omega)=ER_{\Gamma,E}(\omega)\quad\mbox{for any}\quad\omega\in\RR.
	\end{equation}
	Moreover, from \eqref{generator-representation} we infer that
	\begin{equation}\label{r4.6-5}
	S_{\Gamma,E}^*(2\pi\rmi\omega-S_{\Gamma,E}^*)^{-1}=U_{\Gamma,E}^* M_{H_{\Gamma,E}}(2\pi\rmi\omega-M_{H_{\Gamma,E}})^{-1}(U_{\Gamma,E}^*)^{-1}=U_{\Gamma,E}^* \tR_{\Gamma,E}(\omega)(U_{\Gamma,E}^*)^{-1}
	\end{equation}
	for any $\omega\in\RR$. Since $E$ is invertible by Hypothesis (S)(iv), assertion \eqref{r4.6-1} follows from \eqref{r4.6-3}, \eqref{r4.6-4} and \eqref{r4.6-5}.

	\noindent\textit{Proof of (iii)} Defining $\Lambda_\pm:=\{\lambda\in\Lambda:\pm H_{\Gamma,E}(\lambda)>0\}$, from \eqref{4.5-3} we immediately conclude that
	\begin{equation}\label{representation-Ypm}
	\Lambda_\pm:=\{\lambda\in\Lambda:\pm H_{\Gamma,E}(\lambda)\geq\nu\}\quad \Lambda=\Lambda_+\cup \Lambda_-,\quad \Lambda_+\cap \Lambda_-=\emptyset.
	\end{equation}
	It follows that $L^2(\Lambda,\mu)=L^2(\Lambda_+,\mu)\oplus L^2(\Lambda_-,\mu)$. One can readily check that $M_{H_{\Gamma,E}}$, the operator of multiplication by the functions $\pm\chi_{\Lambda_\mp}H_{\Gamma,E}$ generate two $C^0$-semigroups on $L^2(\Lambda_\pm,\mu)$ given by \eqref{def-tilde-T}.
	Assertions \eqref{representation-stable-unstable} and \eqref{representation-bi-semigroup} are direct consequence of representation \eqref{generator-representation}.
	\end{proof}
	To better understand the properties of the Fourier multiplier $\cK_{\Gamma,E}$, we analyze the properties of the Fourier multiplier $\tcK_{\Gamma,E}:L^2(\RR,L^2(\Lambda,\mu))\to L^2(\RR,L^2(\Lambda,\mu))$
	defined by $\tcK_{\Gamma,E}\tf=\cF^{-1}M_{\tR_{\Gamma,E}}\cF\tf$. From Hypothesis (S)(iv) and \eqref{r4.6-1} we conclude that
	\begin{equation}\label{tilde-R-GammaE}
	\sup_{\omega\in\RR}\|\tR_{\Gamma,E}(\omega)\|<\infty,
	\end{equation}
	which implies that the linear operator $\tcK_{\Gamma,E}$ is well-defined and bounded on $L^2(\RR,L^2(\Lambda,\mu))$. Moreover, we note that the operator-valued functions $R_{\Gamma,E}$ and $\tR_{\Gamma,E}$ are differentiable, and from Hypothesis (S)(iv) and \eqref{tilde-R-GammaE}, respectively, we have that
	\begin{equation}\label{Mikhlin-Hormander}
	\sup_{\omega\in\RR}\|R_{\Gamma,E}(\omega)\|<\infty,\;\sup_{\omega\in\RR}|\omega|\|R_{\Gamma,E}'(\omega)\|<\infty,\;
	\sup_{\omega\in\RR}\|\tR_{\Gamma,E}(\omega)\|<\infty,\;\sup_{\omega\in\RR}|\omega|\|\tR_{\Gamma,E}'(\omega)\|<\infty.
	\end{equation}
	From the Mikhlin-Hormander multiplier theorem we conclude that the Fourier multipliers $\cK_{\Gamma,E}$ and $\tcK_{\Gamma,E}$ are well-defined and bounded on
	$L^p(\RR,X)$ and $L^p(\RR,L^2(\Lambda,\mu))$, respectively, for any $p\in(1,\infty)$. In the case of differential equations on finite dimensional spaces one proves the
	existence of the center-stable manifold by using a fixed point argument on $L^\infty(\RR_+,\bX)$ or $C_{\mathrm{b}}(\RR_+,\bX)$. In the example below, we prove that
	the Fourier multiplier $\cK_{\Gamma,E}$ (and thus, $\tcK_{\Gamma,E}$) is not a bounded, linear operator on $L^\infty(\RR_+,\bX)$. Therefore, to prove the existence result
	of a center-stable manifold of solutions of equation \eqref{Relax-Sys}, we need to find a proper subspace of $L^\infty(\RR_+,\bX)$ invariant under $\cK_{\Gamma,E}$.
	\begin{example}\label{e4.7} Let $\bX=\ell^2$, $\Gamma:\ell^2\to\ell^2$ defined by $\Gamma \bx=(-\frac{1}{e^n}\bx_n)_{n\geq 1}$ for any $\bx=(\bx_n)_{n\geq1}\in\ell^2$, $E=-Id_{\ell^2}$. Then, one can readily check that the pair $(\Gamma,E)$ satisfies Hypothesis (S) and $E$ is negative definite.
	Moreover, the Fourier multiplier $\cK_{\Gamma,E}=\cF^{-1}M_{R_{\Gamma,E}}\cF$ does not map $L^2(\RR_+,\ell^2)\cap L^\infty(\RR_+,\ell^2)$ into $L^\infty(\RR_+,\ell^2)$.

	Indeed, one can readily check that
	\begin{equation}\label{e4.7-1}
	R_{\Gamma,E}(\omega)\bx=(\frac{e^n}{e^n-2\pi\rmi\omega}\bx_n)_{n\geq1}=(\widehat{F}_n(\omega)\bx_n)_{n\geq1}\quad\mbox{for any}\quad \omega\in\RR, \bx=(\bx_n)_{n\geq1}\in\ell^2.
	\end{equation}
	Here, the sequence of functions $F_n:\RR\to\RR$ is defined by
	\begin{equation}\label{l2-kernel}
	F_n(\tau)=\left\{\begin{array}{l l}
	0 & \; \mbox{if $\tau\geq0$ }\\
	e^n e^{e^n\tau} & \; \mbox{if $\tau<0$}\\
	\end{array} \right..
	\end{equation}
	It follows that the following representation holds true:
	\begin{equation}\label{e4.7-2}
	(\cK_{\Gamma,E}f)(\tau)=\Big((F_n*f_n)(\tau)\Big)_{n\geq1}\quad\mbox{for any}\quad f=(f_n)_{n\geq1}\in L^2(\RR_+,\ell^2)\cap L^\infty(\RR_+,\ell^2),\;\tau\in\RR.
	\end{equation}
	Let $g:\RR_+\to\ell^2$ be defined by $g(\tau)=(g_n(\tau))_{n\geq1}$, where $g_n=\chi_{[e^{-(n+1)}, e^{-n})}$, $n\geq 1$. Here we recall that $\chi_J$ denotes the characteristic function of the set $J\subseteq\RR$. We compute
	\begin{equation}\label{4.7-3}
	\|g(\tau)\|^2_{\ell^2}=\sum_{n=1}^{\infty}\chi^2_{[e^{-(n+1)}, e^{-n})}(\tau)=\sum_{n=1}^{\infty}\chi_{[e^{-(n+1)}, e^{-n})}(\tau)=\chi_{(0,e^{-1})}(\tau)\quad\mbox{for any}\quad \tau\geq 0.
	\end{equation}
	We conclude that $g\in L^2(\RR_+,\ell^2)\cap L^\infty(\RR_+,\ell^2)$. Moreover, from \eqref{l2-kernel}, we obtain that
	\begin{equation}\label{e4.7-4}
	(F_n*g_n)(\tau)=\int_{\RR} F_n(\tau-s)g_n(s)\rmd s=\int_{e^{-(n+1)}}^{e^{-n}} e^n e^{e^n(\tau-s)} \rmd s=e^{e^n\tau}(e^{-1}-e^{-e})
	\end{equation}
	for any $\tau\in[0,e^{-(n+1)})$ and $n\geq 1$. Therefore, for any $m\in\NN$ and any $\tau\in [0,e^{-(m+1)}]$ we have that
	\begin{equation}\label{e4.7-5}
	\|(\cK_{\Gamma,E}g)(\tau)\|^2_{\ell^2}=\sum_{n=1}^{\infty}|(F_n*g_n)(\tau)|^2\geq\sum_{n=1}^{m}e^{2e^n\tau}(e^{-1}-e^{-e})^2\geq m (e^{-1}-e^{-e})^2.
	\end{equation}
	Assume for a contradiction that $\cK_{\Gamma,E}g\in L^\infty(\RR_+,\ell^2)$. From \eqref{e4.7-5} we infer that $\|\cK_{\Gamma,E}g\|_\infty\geq \sqrt{m}(e^{-1}-e^{-e})$ for any $m\in\NN$, which is a contradiction.
	\end{example}
	Next, we look for a pointwise convolution-like formula for the Fourier multiplier $\tcK_{\Gamma,E}$. In what follows we use the following notation: for any function
	$\tf\in L^\infty(\RR_+, L^2(\Lambda,\mu))$ we use the simplified notation $\tf(\tau,y)$ for $\big(\tf(\tau)\big)(y)$ for any $\tau\geq 0$ and $\lambda\in\Lambda$.

	We introduce the function $\Phi_{\Gamma,E}:\RR\times Y\to\RR$ defined by
	\begin{equation}\label{def-Phi}
	\Phi_{\Gamma,E}(\tau,y)=\left\{\begin{array}{l l}
	H_{\Gamma,E}(\lambda)e^{\tau H_{\Gamma,E}(\lambda)} & \; \mbox{if $\tau>0$ and $H_{\Gamma,E}(\lambda)<0$ }\\
	-H_{\Gamma,E}(\lambda)e^{\tau H_{\Gamma,E}(\lambda)} & \; \mbox{if $\tau<0$ and $H_{\Gamma,E}(\lambda)>0$ }\\
	0& \; \mbox{otherwise}
	\end{array} \right..
	\end{equation}
	Since $\widehat{\Phi_{\Gamma,E}(\cdot,y)}(\omega)=\frac{H_{\Gamma,E(y)}}{2\pi\rmi\omega-H_{\Gamma,E}(\lambda)}$ for any $\omega\in\RR$ and $\lambda\in\Lambda$, from the definition of the Fourier multiplier
	$\tcK_{\Gamma,E}$ and the operator-valued function $\tR_{\Gamma,E}$ in \eqref{r4.6-2}, we conclude that
	\begin{equation}\label{convolution-tilde-K}
	(\tcK_{\Gamma,E}\tf)(\tau,y)=\int_{\RR}\Phi_{\Gamma,E}(\tau-s,y)\tf(s,y)\rmd s\;\mbox{for any}\;\tf\in L^2(\RR_+,L^2(\Lambda,\mu))\cap L^\infty(\RR_+,L^2(\Lambda,\mu)).
	\end{equation}
	for any $\tau\geq 0$ and $\lambda\in\Lambda$.

	Next, we study if the Sobolev space $H^1(\RR_+,\bX)$ is invariant under $\cK_{\Gamma,E}$. First, we note that $g\in H^1(\RR_+,\bX)$ if and only if $g\in L^2(\RR_+,\bX)$ and
	the function $\omega\to 2\pi\rmi\omega \widehat{g}(\omega)-g(0)$ belongs to $L^2(\RR,\bX)$. Here, we recall that if a function $g$ is defined on a proper subset of $\RR$, we use the same notation to denote its extension by $0$ to the hole line. Using Hypothesis (S)(iv) we can show that the space $H^1(\RR,\bX)$ is invariant under $\cK_{\Gamma,E}$. However, by using the same argument, we can check that $H^1(\RR_+,\bX)$ is not invariant under $\cK_{\Gamma,E}$ since $R_{\Gamma,E}(\cdot)\bx\notin L^2(\RR,\bX)$ for any $\bx\in\bX\setminus\dom(|S_{\Gamma,E}|^{1/2})$. Our goal is to prove the existence of an $H^1$ center-stable manifold by using a fixed point argument on equation \eqref{var-const} for $f=D(\bu,\bu)$. Since $H^1(\RR_+,\bX)$ is not invariant under $\cK_{\Gamma,E}$, we need to rearrange the equation first by adding a correction term to $\cK_{\Gamma,E}$. We parameterize equation \eqref{var-const} as follows: we look for solutions $\bu$ satisfying $\bu(0)=\bv_0-E^{-1}f(0)$ for some $\bv_0$ to be chosen later. Therefore, equation \eqref{var-const} is equivalent to
	\begin{equation}\label{mod-var-const}
	\bu=T_\rms^{\Gamma,E}(\cdot)P_\rms^{\Gamma,E}\bv_0+(\cK_{\Gamma,E}f)_{|\RR_+}-T_\rms^{\Gamma,E}(\cdot)P_\rms^{\Gamma,E}E^{-1}f(0).
	\end{equation}
	For any $f\in H^1(\RR_+,\bX)$ we define the function $\cKm f:=(\cK_{\Gamma,E}f)_{|\RR_+}-T_\rms^{\Gamma,E}(\cdot)P_\rms^{\Gamma,E}E^{-1}f(0)$. Clearly, $\cKm$ is a linear operator from
	$H^1(\RR_+,\bX)$ to $L^2(\RR_+,\bX)$. In what follows we prove that $\cKm f\in H^1(\RR_+,\bX)$ for any $f\in H^1(\RR_+,\bX)$ and compute its derivative.
	\begin{lemma}\label{l4.8}
	Assume Hypothesis (S). Then, $\cKm f\in H^1(\RR_+,\bX)$ for any $f\in H^1(\RR_+,\bX)$ and $(\cKm f)'=(\cK_{\Gamma,E}f')_{|\RR_+}$. Moreover, there exists $c(\Gamma,E)>0$ such that
	\begin{equation}\label{cKm-H1-bound}
	\|\cKm f\|_{H^1(\RR_+,\bX)}\leq c(\Gamma,E)\|f\|_{H^1(\RR_+,\bX)}.
	\end{equation}
	\end{lemma}
	\begin{proof}
	To prove our general result, we prove it for functions in a dense subset of $H^1(\RR_+,\bX)$. We introduce the subspace $\cH^1_{\Gamma}=\{f:\RR_+\to\bX:\mbox{there exists}\; g\in H^1(\RR_+,\bX)\;\mbox{such that}\;f(\tau)=\Gamma g(\tau)\;\mbox{for any}\;\tau\geq0 \}$. One can readily check that $\cH^1_{\Gamma}$ is a dense subspace of $H^1(\RR_+,\bX)$. To prove
	the lemma we need to compute $\widehat{\cKm f}$.

	Let $g\in H^1(\RR_+,\bX)$ and $f=\Gamma g$. From the definition of the Fourier multiplier $\cK_{\Gamma,E}$, from Remark~\ref{r4.2}(ii) we obtain that
	\begin{equation}\label{4.8-1}
	\widehat{\cK_{\Gamma,E}f}(\omega)=R_{\Gamma,E}(\omega)\Gamma\widehat{g}(\omega)=(2\pi\rmi\omega\Gamma-E)^{-1}\Gamma\widehat{g}(\omega)
	=R(2\pi\rmi\omega,S_{\Gamma,E})\widehat{g}(\omega)=\widehat{\cG_{\Gamma,E}*g}(\omega).
	\end{equation}
	for any $\omega\in\RR$, which implies that $\cK_{\Gamma,E}f=\cG_{\Gamma,E}*g$. It follows that
	\begin{align}\label{4.8-2}
	(\cK_{\Gamma,E}f)(\tau)&=\int_{-\infty}^{\tau} T_\rms^{\Gamma,E}(\tau-s)P_\rms^{\Gamma,E}g(s)\rmd s-\int_{\tau}^{\infty}T_\rmu^{\Gamma,E}(s-\tau)P_\rmu^{\Gamma,E}g(s)\rmd s\nonumber\\
	&=-T_\rmu^{\Gamma,E}(-\tau)\int_{0}^{\infty}T_\rmu^{\Gamma,E}(s)P_\rmu^{\Gamma,E}g(s)\rmd s\quad\mbox{for any}\quad\tau<0.
	\end{align}
	We infer that $\chi_{\RR_+}(\cK_{\Gamma,E}f)=\cK_{\Gamma,E}f+F_1$, where $F_1:\RR\to\bX$ is the function defined by $F_1(\tau)=0$ for $\tau\geq 0$ and $F_1(\tau)=T_\rmu^{\Gamma,E}(-\tau)\int_{0}^{\infty}T_\rmu^{\Gamma,E}(s)P_\rmu^{\Gamma,E}g(s)\rmd s$  for $\tau<0$. We recall that the generator of the $C^0$-semigroup $\{T_\rmu(\tau)\}_{\tau\geq0}$ is
	$S_\rmu^{\Gamma,E}=-(S_{\Gamma,E})_{|\bX_\rmu^{\Gamma,E}}$. Therefore, we obtain that
	\begin{equation}\label{4.8-3}
	\widehat{F_1}(\omega)=\int_{-\infty}^{0}e^{-2\pi\rmi\omega\tau} T_\rmu(-\tau)\bx_\rmu\rmd\tau=-\big(2\pi\rmi\omega+S_\rmu^{\Gamma,E}\big)^{-1}\bx_\rmu=-R(2\pi\rmi\omega,S_{\Gamma,E})\bx_\rmu=-R_{\Gamma,E}(\omega)\Gamma\bx_\rmu
	\end{equation}
	for any $\omega\in\RR$, where $\bx_\rmu=\int_{0}^{\infty}T_\rmu^{\Gamma,E}(s)P_\rmu^{\Gamma,E}g(s)\rmd s$. Next, we define the function $F_2:\RR\to\bX$ by $F_2(\tau)=T_\rms(\tau)\bx_\rms$ for $\tau\geq 0$ and $F_2(\tau)=0$ for $\tau<0$, where $\bx_\rms=P_\rms^{\Gamma,E}E^{-1}f(0)$. Similarly, since the generator of the $C^0$-semigroup $\{T_\rms(\tau)\}_{\tau\geq0}$ is
	$S_\rms^{\Gamma,E}=(S_{\Gamma,E})_{|\bX_\rms^{\Gamma,E}}$, we have that
	\begin{align}\label{4.8-4}
	\widehat{F_2}(\omega)&=\int_{0}^{\infty}e^{-2\pi\rmi\omega\tau} T_\rms(\tau)\bx_\rms\rmd\tau=\big(2\pi\rmi\omega-S_\rms^{\Gamma,E}\big)^{-1}P_\rms^{\Gamma,E}E^{-1}f(0)=R_{\Gamma,E}(\omega)\Gamma P_\rms^{\Gamma,E}S_{\Gamma,E}^{-1}g(0)\nonumber\\
	&=R_{\Gamma,E}(\omega)\Gamma S_{\Gamma,E}^{-1}P_\rms^{\Gamma,E}g(0)=R_{\Gamma,E}(\omega)\Gamma E^{-1}\Gamma P_\rms^{\Gamma,E}g(0)\quad\mbox{for any}\quad\omega\in\RR.
	\end{align}
	From \eqref{4.8-2}, \eqref{4.8-3} and \eqref{4.8-4} we conclude that
	\begin{align}\label{4.8-5}
	\widehat{\cKm f}(\omega)&=\widehat{\chi_{\RR_+}(\cK_{\Gamma,E}f)}(\omega)-\widehat{F_2}(\omega)=\widehat{\cK_{\Gamma,E}f}(\omega)+\widehat{F_1}(\omega)-\widehat{F_2}(\omega)\nonumber\\
	&=R_{\Gamma,E}(\omega)\widehat{f}(\omega)-R_{\Gamma,E}(\omega)\Gamma\int_{0}^{\infty}T_\rmu^{\Gamma,E}(s)P_\rmu^{\Gamma,E}g(s)\rmd s-R_{\Gamma,E}(\omega)\Gamma E^{-1}\Gamma P_\rms^{\Gamma,E}g(0)
	\end{align}
	for any $\omega\in\RR$. In addition, from \eqref{4.8-2} we have that
	\begin{equation}\label{4.8-6}
	(\cKm f)(0)=-\bx_\rmu-P_\rms^{\Gamma,E}S_{\Gamma,E}^{-1}g(0)=-\int_{0}^{\infty}T_\rmu^{\Gamma,E}(s)P_\rmu^{\Gamma,E}g(s)\rmd s-S_{\Gamma,E}^{-1}P_\rms^{\Gamma,E}g(0).
	\end{equation}
	Since $g\in H^1(\RR_+,\bX)$ we have that $2\pi\rmi\omega\widehat{g}(\omega)=g(0)+\widehat{g'}(\omega)$ for any $\omega\in\RR$. From \eqref{4.8-5} and \eqref{4.8-6} it follows that
	\begin{align}\label{4.8-7}
	&2\pi\rmi\omega\widehat{(\cKm f)}(\omega)-(\cKm f)(0)=2\pi\rmi\omega R_{\Gamma,E}(\omega)\widehat{f}(\omega)-2\pi\rmi\omega R_{\Gamma,E}(\omega)\Gamma\int_{0}^{\infty}T_\rmu^{\Gamma,E}(s)P_\rmu^{\Gamma,E}g(s)\rmd s\nonumber\\
	&\qquad-2\pi\rmi\omega R_{\Gamma,E}(\omega)\Gamma E^{-1}\Gamma P_\rms^{\Gamma,E}g(0)+\int_{0}^{\infty}T_\rmu^{\Gamma,E}(s)P_\rmu^{\Gamma,E}g(s)\rmd s+S_{\Gamma,E}^{-1}P_\rms^{\Gamma,E}g(0)\nonumber\\
	&=R_{\Gamma,E}(\omega)\Gamma\big(2\pi\rmi\omega \widehat{g}(\omega)\big)+\Big(I_\bX-2\pi\rmi\omega R_{\Gamma,E}(\omega)\Gamma\Big)\Big(\int_{0}^{\infty}T_\rmu^{\Gamma,E}(s)P_\rmu^{\Gamma,E}g(s)\rmd s+ E^{-1}\Gamma P_\rms^{\Gamma,E}g(0)\Big)\nonumber\\
	&=2\pi\rmi\omega R_{\Gamma,E}(\omega)\Gamma\widehat{g'}(\omega)+R_{\Gamma,E}(\omega)\Gamma g(0)-R_{\Gamma,E}(\omega)E\Big(\int_{0}^{\infty}T_\rmu^{\Gamma,E}(s)P_\rmu^{\Gamma,E}g(s)\rmd s+ E^{-1}\Gamma P_\rms^{\Gamma,E}g(0)\Big)\nonumber\\
	&=R_{\Gamma,E}(\omega)\widehat{f'}(\omega)+R_{\Gamma,E}(\omega)\Big(\Gamma P_\rmu^{\Gamma,E}g(0)-E\int_{0}^{\infty}T_\rmu^{\Gamma,E}(s)P_\rmu^{\Gamma,E}g(s)\rmd s\Big)\quad\mbox{for any}\quad\omega\in\RR.
	\end{align}
	Since $g\in H^1(\RR_+,\bX)$ we obtain that $\int_{0}^{\infty}T_\rmu^{\Gamma,E}(s)P_\rmu^{\Gamma,E}g(s)\rmd s\in\dom(S_\rmu^{\Gamma,E})\subseteq\dom(S_{\Gamma,E})$ and
	\begin{equation}\label{4.8-8}
	S_\rmu^{\Gamma,E}\int_{0}^{\infty}T_\rmu^{\Gamma,E}(s)P_\rmu^{\Gamma,E}g(s)\rmd s=-P_\rmu^{\Gamma,E}g(0)-\int_{0}^{\infty}T_\rmu^{\Gamma,E}(s)P_\rmu^{\Gamma,E}g'(s)\rmd s.
	\end{equation}
	Since $S_\rmu^{\Gamma,E}=-(S_{\Gamma,E})_{|\bX_\rmu^{\Gamma,E}}$, from \eqref{4.8-8} we obtain that
	\begin{equation}\label{4.8-9}
	E\int_{0}^{\infty}T_\rmu^{\Gamma,E}(s)P_\rmu^{\Gamma,E}g(s)\rmd s=\Gamma P_\rmu^{\Gamma,E}g(0)+\Gamma \int_{0}^{\infty}T_\rmu^{\Gamma,E}(s)P_\rmu^{\Gamma,E}g'(s)\rmd s.
	\end{equation}
	From Using Remark~\ref{r4.2}(ii), \eqref{4.8-7} and \eqref{4.8-9} we infer that
	\begin{align}\label{4.8-10}
	2\pi\rmi\omega\widehat{(\cKm f)}(\omega)-(\cKm f)(0)&=R_{\Gamma,E}(\omega)\widehat{f'}(\omega)-R_{\Gamma,E}(\omega)\Gamma\int_{0}^{\infty}T_\rmu^{\Gamma,E}(s)P_\rmu^{\Gamma,E}g'(s)\rmd s\nonumber\\
	&=\cF\Big(\cK_{\Gamma,E}f'-\cG_{\Gamma,E}(\cdot)\int_{0}^{\infty}T_\rmu^{\Gamma,E}(s)P_\rmu^{\Gamma,E}g'(s)\rmd s\Big)
	\end{align}
	for any $\omega\in\RR$. Arguing
	similarly as in
	\eqref{4.8-1}, we have that $\cK_{\Gamma,E}f'=\cG_{\Gamma,E}*g'$, which implies that
	\begin{align}\label{4.8-11}
	(\cK_{\Gamma,E}f')(\tau)&=\int_{-\infty}^{\tau} T_\rms^{\Gamma,E}(\tau-s)P_\rms^{\Gamma,E}g'(s)\rmd s-\int_{\tau}^{\infty}T_\rmu^{\Gamma,E}(s-\tau)P_\rmu^{\Gamma,E}g'(s)\rmd s\nonumber\\
	&=-T_\rmu^{\Gamma,E}(-\tau)\int_{0}^{\infty}T_\rmu^{\Gamma,E}(s)P_\rmu^{\Gamma,E}g'(s)\rmd s=\cG_{\Gamma,E}(\tau)\int_{0}^{\infty}T_\rmu^{\Gamma,E}(s)P_\rmu^{\Gamma,E}g'(s)\rmd s
	\end{align}
	for any $\tau<0$. Since $\cG_{\Gamma,E}(\tau)\bx=0$ for any $\tau\geq 0$ and any $\bx\in\bX_\rmu^{\Gamma,E}$, from \eqref{4.8-10} and \eqref{4.8-11} we conclude that
	\begin{equation}\label{4.8-12}
	2\pi\rmi\omega\widehat{(\cKm f)}(\omega)-(\cKm f)(0)=\cF(\chi_{\RR_+}\cK_{\Gamma,E}f')(\omega)\quad\mbox{for any}\quad\omega\in\RR.
	\end{equation}
	Since $f'\in L^2(\RR_+,\bX)\hookrightarrow L^2(\RR,\bX)$ and $\cK_{\Gamma,E}$ is bounded on $L^2(\RR,\bX)$ from \eqref{4.8-12} we infer that $\cKm f\in H^1(\RR_+,\bX)$ and
	$(\cKm f)'=(\cK_{\Gamma,E}f')_{|\RR_+}$ for any $f\in\cH^1_{\Gamma}$.

	Next, we fix $f\in H^1(\RR_+,\bX)$ and let $\{f_n\}_{n\geq 1}$ be a sequence of functions in $\cH^1_{\Gamma}$ such that $f_n\to f$ as $n\to\infty$ in $H^1(\RR_+,\bX)$.
	From Remark~\ref{r4.2}(i) we obtain that
	\begin{equation}\label{4.8-13}
	\|T_\rms^{\Gamma,E}(\cdot)P_\rms^{\Gamma,E}E^{-1}f_n(0)-T_\rms^{\Gamma,E}(\cdot)P_\rms^{\Gamma,E}E^{-1}f(0)\|_2\leq c\|f_n(0)-f(0)\|\leq c\|f_n-f\|_{H^1}
	\end{equation}
	for any $n\geq 1$. Since the Fourier multiplier $\cK_{\Gamma,E}$ is bounded on $L^2(\RR,\bX)$, from \eqref{4.8-13} we conclude that $\cKm f_n=(\cK_{\Gamma,E}f_n)_{|\RR_+}-T_\rms^{\Gamma,E}(\cdot)P_\rms^{\Gamma,E}E^{-1}f_n(0)\to(\cK_{\Gamma,E}f)_{|\RR_+}-T_\rms^{\Gamma,E}(\cdot)P_\rms^{\Gamma,E}E^{-1}f(0)=\cKm f$ as $n\to\infty$ in $L^2(\RR_+,\bX)$ and $(\cKm f_n)'=(\cK_{\Gamma,E}f_n')_{|\RR_+}\to (\cK_{\Gamma,E}f')_{|\RR_+}$ as  $n\to\infty$ in $L^2(\RR_+,\bX)$. It follows that $\cKm f\in H^1(\RR_+,\bX)$ and $(\cKm f)'=(\cK_{\Gamma,E}f')_{|\RR_+}$.

	To finish the proof of lemma, we need to prove \eqref{cKm-H1-bound}. Indeed, since $\cK_{\Gamma,E}$ is bounded on $L^2(\RR,\bX)$  we have that
	\begin{align}\label{4.8-14}
	\|\cKm f\|_{H^1}^2&=\|(\cK_{\Gamma,E}f)_{|\RR_+}-T_\rms^{\Gamma,E}(\cdot)P_\rms^{\Gamma,E}E^{-1}f(0)\|_2^2+\|(\cK_{\Gamma,E}f')_{|\RR_+}\|_2^2\nonumber\\
	&\leq 4 \|\cK_{\Gamma,E}f\|_2^2+4\|T_\rms^{\Gamma,E}(\cdot)P_\rms^{\Gamma,E}E^{-1}f(0)\|_2^2+c\|f'\|_2^2\nonumber\\
	&\leq c\|f\|_2^2+c\|f'\|_2^2+c\|f(0)\|^2\leq c \|f\|_{H^1}^2.
	\end{align}
	for any $f\in H^1(\RR_+,\bX)$, proving the lemma.
	\end{proof}
	\begin{remark}\label{r4.10} Assume $\cW:\RR\to\cB(\bX)$ is a piecewise strongly continuous operator valued function such that $\|\cW(\tau)\|\leq ce^{-\nu|\tau|}$ for all $\tau\in\RR$. Then, for any $\alpha\in [0,\nu)$ we have that
	$\cW*f\in L^2_\alpha(\RR,\bX)$ and
	\begin{equation}\label{L2alpha-estimate}
	\|\cW*f\|_{L^2_\alpha}\leq c\|f\|_{L^2_\alpha}\quad\mbox{for any}\quad f\in L^2_\alpha(\RR,\bX).
	\end{equation}
	\end{remark}
\begin{proof}
Fix $\alpha\in [0,\nu)$. Since $\cW$ decays exponentially, one can readily check that
\begin{equation}\label{4.9-1}
\|(\cW*f)(\tau)\|^2\leq\Big(\int_\RR e^{-\frac{\nu-\alpha}{2}|\tau-s|}e^{-\frac{\nu+\alpha}{2}|\tau-s|}\,\|f(s)\|\rmd s\Big)^2\leq \frac{1}{\nu-\alpha}\int_\RR e^{-(\nu+\alpha)|\tau-s|}\,\|f(s)\|^2\rmd s
\end{equation}
for any $\tau\in\RR$. Since $\int_\RR e^{-(\nu+\alpha)|\tau-s|}e^{2\alpha|\tau|}\rmd\tau\leq\frac{1}{\nu-\alpha}e^{2\alpha|s|}$ for any $s\in\RR$, from \eqref{4.9-1} we obtain that
\begin{align}\label{4.9-2}
\int_\RR e^{2\alpha|\tau|}\|(\cW*f)(\tau)\|^2\rmd\tau&\leq\frac{1}{\nu-\alpha}\int_\RR\Big(\int_\RR e^{-(\nu+\alpha)|\tau-s|}e^{2\alpha|\tau|}\rmd\tau\Big)\|f(s)\|^2\rmd s\nonumber\\&\leq\frac{1}{(\nu-\alpha)^2}\int_\RR e^{2\alpha|s|}\|f(s)\|^2\rmd s=\frac{1}{(\nu-\alpha)^2}\|f\|_{L^2_\alpha}^2,
\end{align}
proving the remark.
\end{proof}
Next, we analyze the invariance properties of weighted spaces under $\cKm$. In particular, we are
	interested in checking whether the weighted Sobolev space $H^1_{\alpha}(\RR_+,\bX)$
	is
	invariant under $\cKm$. To prove this result we need the following lemma:
\begin{lemma}\label{l4.6}
	Assume Hypothesis (S) and let $\psi\in H^2(\RR)$ be a smooth scalar function. Then, the following identity holds:
	\begin{equation}\label{cK-weight}
	\cK_{\Gamma,E}\Big(\psi f+\psi'(\cG_{\Gamma,E}^**f)\Big)=\psi \cK_{\Gamma,E} f\quad\mbox{for any}\quad f\in L^2(\RR,\bX).
	\end{equation}
	\end{lemma}
	\begin{proof}
	To start the proof of lemma, we first justify that the left hand side of equation \eqref{cK-weight} is well defined. Indeed, since $\psi\in H^2(\RR)$ we have that
	$\psi,\psi'\in L^\infty(\RR)$. Moreover, from Remark~\ref{r4.2}(i) it follows that $\cG_{\Gamma,E}$ and thus $\cG_{\Gamma,E}^*$ are exponentially decaying, operator valued functions, which implies that
	$\cG_{\Gamma,E}^**f\in L^2(\RR,\bX)$ for any $f\in L^2(\RR,\bX)$. We conclude that $\psi f+\psi'(\cG_{\Gamma,E}^**f)\in L^2(\RR,\bX)$ for any $f\in L^2(\RR,\bX)$.

	From Remark~\ref{r4.2}(ii) we have that $\cF\cG_{\Gamma,E}(\cdot)\bu=R(2\pi\rmi\cdot,S_{\Gamma,E})\bu$ for any $\bu\in L^2(\RR,\bX)$. From Lemma~\ref{l3.3} we obtain that
	\begin{equation}\label{Fourier-cV}
	\cF\cG_{\Gamma,E}^*(\cdot)\bu=\big(\cF\cG_{\Gamma,E}(-\cdot)\big)^*\bu=R(-2\pi\rmi\cdot,S_{\Gamma,E})^*\bu=\big(R_{\Gamma,E}(-\cdot)\Gamma\big)^*\bu=\Gamma R_{\Gamma,E}(\cdot)\bu
	\end{equation}
	for any $\bu\in L^2(\RR,\bX)$. Since $\psi\in H^2(\RR)$ we have that $\widehat{\psi},\widehat{\psi'}\in L^1(\RR)$. Taking Fourier transform, from \eqref{char-eq} and \eqref{Fourier-cV} we obtain that
	\begin{align}\label{4.6-1}
	\widehat{\psi \cK_{\Gamma,E} f}(\omega)&=(\widehat{\psi}*\widehat{\cK_{\Gamma,E}f})(\omega)=\int_\RR \widehat{\psi}(\omega-\theta)\widehat{\cK_{\Gamma,E}f}(\theta)\rmd\theta=\int_\RR \widehat{\psi}(\omega-\theta)R_{\Gamma,E}(\theta)\widehat{f}(\theta)\rmd\theta\nonumber\\
	&=\int_\RR \widehat{\psi}(\omega-\theta)\Big(R_{\Gamma,E}(\omega)+2\pi\rmi(\omega-\theta)R_{\Gamma,E}(\omega)\Gamma R_{\Gamma,E}(\theta)\Big)\widehat{f}(\theta)\rmd\theta\nonumber\\
	&=R_{\Gamma,E}(\omega)\int_\RR \widehat{\psi}(\omega-\theta)\widehat{f}(\theta)\rmd\theta+R_{\Gamma,E}(\omega)\int_\RR 2\pi\rmi(\omega-\theta)\widehat{\psi}(\omega-\theta)\Gamma R_{\Gamma,E}(\theta)\widehat{f}(\theta)\rmd\theta\nonumber\\
	&=R_{\Gamma,E}(\omega)(\widehat{\psi}*\widehat{f})(\omega)+R_{\Gamma,E}(\omega)\int_\RR \widehat{\psi'}(\omega-\theta)\Gamma R_{\Gamma,E}(\theta)\widehat{f}(\theta)\rmd\theta\nonumber\\
	&=R_{\Gamma,E}(\omega)\widehat{\psi f}(\omega)+R_{\Gamma,E}(\omega)\int_\RR \widehat{\psi'}(\omega-\theta)\widehat{\cG_{\Gamma,E}}^*(\theta)\widehat{f}(\theta)\rmd\theta\nonumber\\
	&=R_{\Gamma,E}(\omega)\widehat{\psi f}(\omega)+R_{\Gamma,E}(\omega)\int_\RR \widehat{\psi'}(\omega-\theta)\widehat{\cG_{\Gamma,E}^**f}(\theta)\rmd\theta\nonumber\\
	&=R_{\Gamma,E}(\omega)\Big(\widehat{\psi f}(\omega)+\widehat{\psi'(\cG_{\Gamma,E}^**f)}(\omega)\Big)\quad\mbox{for any}\quad\omega\in\RR
	\end{align}
	Assertion \eqref{cK-weight} follows
	readily
	from the definition of $\cK_{\Gamma,E}$.
	\end{proof}
	\begin{lemma}\label{l4.11}
	Assume Hypothesis (S). Then, $\cKm f\in H^1_\alpha(\RR_+,\bX)$ for any $f\in H^1_\alpha(\RR_+,\bX)$ for any $\alpha\in (0,\nu)$. Moreover, there exists $c(\Gamma,E,\alpha)>0$ such that
	\begin{equation}\label{cKm-H1-alpha-bound}
	\|\cKm f\|_{H^1_\alpha(\RR_+,\bX)}\leq c(\Gamma,E,\alpha)\|f\|_{H^1_\alpha(\RR_+,\bX)}.
	\end{equation}
	\end{lemma}
	\begin{proof}
	Let $\{\phi_n\}_{n\geq 1}$ be a sequence of scalar functions in $C^\infty(\RR)$ such that $0\leq\phi_n\leq 1$, $\phi_n(\tau)=1$ whenever $|\tau|\leq n$, $\phi_n(\tau)=0$ whenever
	$|\tau|\geq n+1$ and $\sup_{n\geq 1}\|\phi'_n\|_\infty<\infty$. We define the sequence of scalar functions $\{\psi_n\}$ by the formula $\psi_n(\tau)=e^{\alpha\langle\tau\rangle}\phi_n(\tau)$, where $\langle\tau\rangle=\sqrt{1+|\tau|^2}$. Since $\alpha>0$ one can readily check that $\psi_n\in H^2(\RR)$ for any $n\geq 1$.
	Moreover, there exists a constant $c_0>0$, independent of $\Gamma$, $E$ and $\alpha$ such that
	\begin{equation}\label{4.11-1}
	|\psi_n(\tau)|\leq e^{\alpha\tau},\quad|\psi_n'(\tau)|\leq c_0e^{\alpha\tau},\quad\mbox{for any}\quad \tau\in\RR.
	\end{equation}
Let $f\in H^1_\alpha(\RR_+,\bX)\hookrightarrow H^1(\RR_+,\bX)$.
	First, we prove that $\cKm f\in L^2_\alpha(\RR_+,\bX)$. We note that $(\psi_n)_{|\RR_+}\cKm f=(\psi_n\cK_{\Gamma,E}f)_{|\RR_+}-(\psi_n)_{|\RR_+}T_\rms^{\Gamma,E}(\cdot)P_\rms^{\Gamma,E}E^{-1}f(0)$ for any $n\geq 1$. Using Lemma~\ref{l4.6}, Remark~\ref{r4.10} and \eqref{4.11-1} we estimate
	\begin{align}\label{4.11-2}
	\|(\psi_n\cK_{\Gamma,E}f)_{|\RR_+}\|_2&\leq\|\psi_n\cK_{\Gamma,E}f\|_2=\|\cK_{\Gamma,E}(\psi_n f+\psi_n'(\cG_{\Gamma,E}^**f))\|_2\leq c\|\psi_n f+\psi_n'(\cG_{\Gamma,E}^**f)\|_2\nonumber\\
	&\leq c(\Gamma,E)\|f\|_{L^2_\alpha}+c(\Gamma,E)\|\cG_{\Gamma,E}^**f\|_{L^2_\alpha}\leq c(\Gamma,E,\alpha)\|f\|_{L^2_\alpha}
	\end{align}
	for any $n\geq 1$. Moreover, since $\alpha\in (0,\nu)$ one can readily check that
	\begin{equation}\label{4-11-3}
	\|(\psi_n)_{|\RR_+}T_\rms^{\Gamma,E}(\cdot)P_\rms^{\Gamma,E}E^{-1}f(0)\|_2\leq c(\Gamma,E,\alpha)\|f(0)\|\leq c(\Gamma,E,\alpha)\|f\|_{H^1_\alpha}
	\end{equation}
	for any $n\geq 1$. From \eqref{4.11-2} and \eqref{4-11-3} we conclude that $(\psi_n)_{|\RR_+}\cKm f\in L^2(\RR_+,\bX)$ and $\|(\psi_n)_{|\RR_+}\cKm f\|_2\leq c(\Gamma,E,\alpha)\|f\|_{H^1_\alpha}$ for any $n\geq 1$. Passing to the limit as $n\to\infty$ we conclude that $\|\cKm f\|_{L^2_\alpha}\leq c(\Gamma,E,\alpha)\|f\|_{H^1_\alpha}$.

	From Lemma~\ref{l4.8} we have that $\cKm f\in H^1(\RR_+,\bX)$, therefore $\cKm f$ is absolutely continuous on $\RR_+$ and $(\cKm f)'=(\cK_{\Gamma,E}f')_{|\RR_+}$.
	Using again Lemma~\ref{l4.6}, Remark~\ref{r4.10} and \eqref{4.11-1} we have that
	\begin{align}\label{4.11-4}
	\|(\psi_n\cK_{\Gamma,E}f')_{|\RR_+}\|_2&\leq\|\psi_n\cK_{\Gamma,E}f'\|_2=\|\cK_{\Gamma,E}(\psi_n f'+\psi_n'(\cG_{\Gamma,E}^**f'))\|_2\leq c\|\psi_n f'+\psi_n'(\cG_{\Gamma,E}^**f')\|_2\nonumber\\
	&\leq c(\Gamma,E)\|f'\|_{L^2_\alpha}+c(\Gamma,E)\|\cG_{\Gamma,E}^**f'\|_{L^2_\alpha}\leq c(\Gamma,E,\alpha)\|f'\|_{L^2_\alpha}\leq c(\Gamma,E,\alpha)\|f\|_{H^1_\alpha}
	\end{align}
	for any $n\geq 1$. It follows that $\|(\psi_n)_{|\RR_+}(\cKm f)'\|_2\leq c(\Gamma,E,\alpha)\|f\|_{H^1_\alpha}$ for any $n\geq 1$. Passing to the limit as $n\to\infty$ we infer that $(\cKm f)'\in L^2_\alpha(\RR_+,\bX)$ and $\|(\cKm f)'\|_{L^2_\alpha}\leq c(\Gamma,E,\alpha)\|f\|_{H^1_\alpha}$, proving the lemma.
	\end{proof}

	\section{Center-Stable manifolds of general relaxation systems}\label{s5}

	In this section we prove the existence of a local center-stable manifold of equation \eqref{Relax-Sys}, by reducing the equation to \eqref{Relax-Reduced}, as explained in Section~\ref{s4}.
	To prove this result, we prove it for the general case of equation \eqref{h-plus}. First, we recall that
	 from Lemma~\ref{l4.4} we have that mild solutions of equation \eqref{h-plus} satisfy the
	equation
	\begin{equation}\label{fix-point-eq}
	\bu(\tau)=T_\rms^{\Gamma,E}(\tau)P_\rms^{\Gamma,E}\bu(0)+\big(\cK_{\Gamma,E}D(\bu(\cdot),\bu(\cdot))\big)(\tau),\;\tau\geq0;
	\end{equation}
	Using the parametrization $\bu(0)=\bv_0-E^{-1}D(\bu(0),\bu(0))$, as in equation \eqref{mod-var-const}, to prove our result it is enough to prove the existence of a center stable manifold
	of equation
	\begin{equation}\label{mod-fix-point-eq}
	\bu=T_\rms^{\Gamma,E}(\cdot)P_\rms^{\Gamma,E}\bv_0+\cKm D(\bu(\cdot),\bu(\cdot)).
	\end{equation}
	In this section we prove that for $\bv_0$ in a certain subspace of $\bX^{\Gamma,E}_{\rms}$ with $\|\bv_0\|$ small enough there exists a unique solution of equation of \eqref{mod-fix-point-eq}, in the space $H^1_\alpha(\RR_+,\bX)$, with $\alpha\in [0,\nu)$. Throughout this section we assume Hypothesis (S) and we assume that the linear operator $E$ is negative-definite. Next, we study the trajectories of the semigroup $\{T^{\Gamma,E}_\rms(\tau)\}_{\tau\geq0}$ that belong to the space $H^1_\alpha(\RR_+,\bX)$. First, we introduce the subspace
	\begin{equation}\label{def-X1/2}
	\bX_{\frac{1}{2}}^{\Gamma,E}=\Big\{\bx\in\bX:\int_\Lambda |H_{\Gamma,E}(\lambda)|\,|(U_{\Gamma,E}\bx)(\lambda)|^2\rmd\mu(\lambda)<\infty\Big\}.
	\end{equation}
	The space $\bX_{\frac{1}{2}}^{\Gamma,E}$ is a Banach space when endowed with the norm
	\begin{equation}\label{def-norm-X1/2}
	\|\bx\|_{\bX_{\frac{1}{2}}^{\Gamma,E}}=\big\||M_{H_{\Gamma,E}}|^{1/2} U_{\Gamma,E}\bx\big\|_{L^2(\Lambda,\mu)},
	\end{equation}
where $U_{\Gamma,E}$ is defined in \eqref{generator-representation}. From \eqref{generator-representation} it follows that $\bX_{\frac{1}{2}}^{\Gamma,E}=\dom(|S_{\Gamma,E}|^{\frac{1}{2}})$ and that the norm $\|\cdot\|_{\bX_{\frac{1}{2}}^{\Gamma,E}}$ is equivalent to the graph norm on $\dom(|S_{\Gamma,E}|^{\frac{1}{2}})$.
	\begin{lemma}\label{l5.1}
	Assume Hypothesis (S) and assume that the linear operator $E$ is negative-definite. If $\nu$ is a positive constant such that the growth rates of the $C^0$-semigroups $\{T^{\Gamma,E}_{\rms/\rmu}(\tau)\}_{\tau\geq 0}$ are less than or equal to $-\nu$, where $\nu=\nu(\Gamma,E)>0$,
	then $T^{\Gamma,E}_\rms(\cdot)\bv_0\in H^1_\alpha(\RR_+,\bX)$ and there exists $c=c(\Gamma,E)>0$ such that
	\begin{equation}\label{trajectory-Z}
	\|T^{\Gamma,E}_\rms(\cdot)\bv_0\|_{H^1_\alpha(\RR_+,\bX)}\leq c(\Gamma,E,\alpha)\|\bv_0\|_{\bX_{\frac{1}{2}}^{\Gamma,E}}
	\end{equation}
	for any $\bv_0\in\bX^{\Gamma,E}_\rms\cap\bX_{\frac{1}{2}}^{\Gamma,E}$ and $\alpha\in [0,\nu)$.
	\end{lemma}
	\begin{proof}
	Fix $\bv_0\in\bX^{\Gamma,E}_\rms\cap\bX_{\frac{1}{2}}^{\Gamma,E}$, $\alpha\in [0,\nu)$ and let $\tg_0=U_{\Gamma,E}\bv_0\in L^2(\Lambda,\mu)$. From \eqref{representation-stable-unstable} we have that $\tg_0\in L^2(\Lambda_-,\mu)$, that is
	$\tg_0(\lambda)=0$ for any $\lambda\in\Lambda_+$. Since $\bv_0\in\bX_{\frac{1}{2}}^{\Gamma,E}$ we infer that
	\begin{equation}\label{5.1-1}
	\int_{\Lambda}|H_{\Gamma,E}(\lambda)| |\tg_0(\lambda)|^2\rmd\mu(\lambda)=\big\||M_{H_{\Gamma,E}}|^{1/2}\tg_0\big\|_{L^2(\Lambda,\mu)}=\|\bv_0\|_{\bX_{\frac{1}{2}}^{\Gamma,E}}<\infty.
	\end{equation}
	From \eqref{representation-bi-semigroup} it follows that
	$T^{\Gamma,E}_\rms(\tau)\bv_0=U_{\Gamma,E}^{-1}\tT^{\Gamma,E}_\rms(\tau)\tg_0\quad\mbox{for any}\quad\tau\geq0$.
	We introduce the function $\tth_0:\RR_+\to L^2(\Lambda,\mu)$ defined by $\big[\tth_0(\tau)\big](\lambda):=e^{\alpha\tau}\big(\tT^{\Gamma,E}_\rms(\tau)\tg_0\big)(\lambda)$. From \eqref{def-tilde-T} we have that $\big[\tth_0(\tau)\big](\lambda)=e^{\tau(\alpha+H_{\Gamma,E}(\lambda))}\chi_{\Lambda_-}(\lambda)\tg_0(\lambda)$ for any $\tau\geq0$ and $\lambda\in\Lambda$. Since $H_{\Gamma,E}(\lambda)\leq -\nu$ for any $\lambda\in\Lambda_-$, from \eqref{5.1-1} we obtain that
	$\tth_0\in H^1(\RR_+,L^2(\Lambda,\mu))$ and
	\begin{align}\label{5.1-2}
	\|\tth_0\|_{H^1}^2&\leq \int_{0}^{\infty}e^{2\tau\alpha}\|\tT^{\Gamma,E}_\rms(\tau)\tg_0\|_{L^2(\Lambda,\mu)}^2\rmd\tau+\int_{0}^{\infty}\int_{\Lambda_-} |(\alpha+H_{\Gamma,E}(\lambda))e^{\tau(\alpha+H_{\Gamma,E}(\lambda))}\tg_0(\lambda)|^2\rmd\mu(\lambda)\rmd\tau\nonumber\\
	&\leq \int_{0}^{\infty}e^{2\tau(\alpha-\nu)}\rmd\tau\,\|\tg_0\|_{L^2(\Lambda,\mu)}^2+\int_{\Lambda_-} |\alpha+H_{\Gamma,E}(\lambda)|^2|\tg_0(\lambda)|^2\Big(\int_{0}^{\infty} e^{2\tau(\alpha+H_{\Gamma,E}(\lambda))}\rmd\tau\Big)\rmd\mu(\lambda)\nonumber\\
	&\leq\frac{1}{2(\nu-\alpha)}\|U_{\Gamma,E}\|^2\|\bv_0\|^2+\frac{1}{2}\int_{\Lambda_-} |\alpha+H_{\Gamma,E}(\lambda)||\tg_0(\lambda)|^2\rmd\mu(\lambda)\nonumber\\
	&\leq\frac{1}{2(\nu-\alpha)}\|U_{\Gamma,E}\|^2\|\bv_0\|^2+\frac{1}{2}|\alpha|\|\tg_0\|_{L^2(\Lambda,\mu)}^2+\frac{1}{2}\int_{\Lambda}|H_{\Gamma,E}(\lambda)| |\tg_0(\lambda)|^2\rmd\mu(\lambda)\nonumber\\
	&\leq\Big[|\alpha|+\frac{1}{2(\nu-\alpha)}\Big]\|U_{\Gamma,E}\|^2\|\bv_0\|^2+\frac{1}{2}\|\bv_0\|_{\bX_{\frac{1}{2}}^{\Gamma,E}}^2\leq c(\Gamma,E,\alpha)\|\bv_0\|_{\bX_{\frac{1}{2}}^{\Gamma,E}}^2.
	\end{align}
	We conclude that $T^{\Gamma,E}_\rms(\cdot)\bv_0\in H^1_\alpha(\RR_+,\bX)$ and \eqref{trajectory-Z} holds true, proving the lemma.
	\end{proof}
	We introduce the function $\Psi_{\Gamma,E}:\bX^{\Gamma,E}_\rms\cap\bX_{\frac{1}{2}}^{\Gamma,E}\times H^1_\alpha(\RR_+,\bX)\to H^1_\alpha(\RR_+,\bX)$ defined by
	\begin{equation}\label{def-Phi-GammaE}
	\Psi_{\Gamma,E}(\bv_0,f)=T^{\Gamma,E}_\rms P^{\Gamma,E}_\rms\bv_0+\cKm D(f(\cdot),f(\cdot)).
	\end{equation}
	To prove that equation \eqref{fix-point-eq} has a unique solution on $H^1_\alpha(\RR_+,\bX)$, we show that the function $\Psi_{\Gamma,E}$ satisfies the
	conditions of the Contraction Mapping Theorem. In what follows we will use the following notation to denote the closed balls of $H^1_\alpha(\RR_+,\bX)$ and $\bX^{\Gamma,E}_\rms\cap\bX_{\frac{1}{2}}^{\Gamma,E}$ centered at the origin
	\begin{align}\label{def-ball-H1}
	\Omega_\alpha(\eps)&=\big\{f\in H^1_\alpha(\RR_+,\bX):\|f\|_{H^1_\alpha(\RR_+,\bX)}\leq\eps\big\},\nonumber\\
	\ovB_\rms(0,\eps)&=\big\{\bv_0\in\bX^{\Gamma,E}_\rms\cap\bX_{\frac{1}{2}}^{\Gamma,E}:\|\bv_0\|_{\bX_{\frac{1}{2}}^{\Gamma,E}}\leq\eps\big\}.
	\end{align}
	\begin{lemma}\label{l5.2}
	Assume Hypothesis (S) and assume that the linear operator $E$ is negative-definite.  If $\nu$ is a positive constant such that the growth rates of the $C^0$-semigroups $\{T^{\Gamma,E}_{\rms/\rmu}(\tau)\}_{\tau\geq 0}$ are less than or equal to $-\nu$, where $\nu=\nu(\Gamma,E)>0$ and $\tnu\in(0,\nu)$, then there exist $\eps_1=\eps_1(\Gamma,E,\tnu)>0$ and
	$\eps_2=\eps_2(\Gamma,E,\tnu)>0$ such that $\Psi_{\Gamma,E}$ maps $\ovB_\rms(0,\eps_1)\times\Omega_\alpha(\eps_2)$ to $\Omega_\alpha(\eps_2)$ and
	\begin{equation}\label{Fix-point-estimate}
	\|\Psi_{\Gamma,E}(\bv_0,f)-\Psi_{\Gamma,E}(\bv_0,g)\|_{H^1_\alpha(\RR_+,\bX)}\leq\frac{1}{2}\|f-g\|_{H^1_\alpha(\RR_+,\bX)}
	\end{equation}
	for any $\bv_0\in\ovB_\rms(0,\eps_1)$, $f, g\in\Omega_\alpha(\eps_2)$ and $\alpha\in [0,\tnu]$.
	\end{lemma}
	\begin{proof} Since $D(\cdot,\cdot)$ is a bounded bilinear map on $\bX$, we infer that for any $f\in H^1_\alpha(\RR_+,\bX)$ we have that $D(f(\cdot),f(\cdot))\in H^1_\alpha(\RR_+,\bX)$ and $\big(D(f,f)\big)'=2D(f,f')$.  Moreover, since
	$H^1_\alpha(\RR_+,\bX)\hookrightarrow L^\infty_\alpha(\RR_+,\bX)$ we have that
	\begin{equation}\label{5.2-1}
	e^{\alpha\tau}\|f(\tau)\|\leq\|f\|_{ L^\infty_\alpha(\RR_+,\bX)}\leq\|f\|_{H^1_\alpha(\RR_+,\bX)}\;\mbox{for any}\;\tau\geq 0.
	\end{equation}
	Using again that $D(\cdot,\cdot)$ is a bounded bilinear map on $\bX$, from \eqref{5.2-1} it follows that
	\begin{align}\label{5.2-2}
	\|D(f,f)\|_{H^1_\alpha}^2&=\int_{0}^{\infty} e^{2\alpha\tau}\|D(f(\tau),f(\tau))\|^2\rmd\tau+4\int_{0}^{\infty} e^{2\alpha\tau}\|D(f(\tau),f'(\tau))\|^2\rmd\tau\nonumber\\
	&\leq\int_{0}^{\infty} e^{2\alpha\tau}\|D\|^2\,\|f(\tau)\|^4\rmd\tau+4\int_{0}^{\infty} e^{2\alpha\tau}\|D\|^2\,\|f(\tau)\|^2\,\|f'(\tau)\|^2\rmd\tau\nonumber\\
	&\leq \|D\|^2 \|f\|_{L^\infty_\alpha}^2\|f\|_{L^2_\alpha}^2+4\|D\|^2 \|f\|_{L^\infty_\alpha}^2\|f'\|_{L^2_\alpha}^2\leq 4\|D\|^2\|f\|_{H^1_\alpha}^4.
	\end{align}
	From Lemma~\ref{l4.11}, Lemma~\ref{l5.1} and \eqref{5.2-2} we infer that $\Psi_{\Gamma,E}(\bv_0,f)\in H^1_\alpha(\RR_+,\bX)$ and
	\begin{align}\label{5.2-3}
	\|\Psi_{\Gamma,E}(\bv_0,f)\|_{H^1_\alpha}&\leq \|T^{\Gamma,E}_\rms(\cdot)\bv_0\|_{H^1_\alpha}+\|\cKm D(f,f)\|_{H^1_\alpha}\leq c\,\|\bv_0\|_{\bX_{\frac{1}{2}}^{\Gamma,E}}+c\,\|D(f(\cdot),f(\cdot))\|_{H^1_\alpha}\nonumber\\
	&\leq c\,\|\bv_0\|_{\bX_{\frac{1}{2}}^{\Gamma,E}}+2c\|D\|\,\|f\|^2_{H^1_\alpha}\leq c\,(\eps_1+\eps_2^2).
	\end{align}
	for any $\bv_0\in\ovB_\rms(0,\eps_1)$ and $f\in\Omega_\alpha(\eps_2)$. Here the constant $c=c(\Gamma,E,\alpha)$ depends on the constants from \eqref{cKm-H1-alpha-bound} and \eqref{trajectory-Z}, therefore it can be it chosen such that
	\begin{equation}\label{5.2-4}
	\sup_{\alpha\in [0,\tnu]}c(\Gamma,E,\alpha)<\infty\;\mbox{for any}\; \tnu\in (0,\nu(\Gamma,E)).
	\end{equation}
	It follows that for any $\tnu\in (0,\nu(\Gamma,E))$ there exist $\eps_1=\eps_1(\Gamma,E,\tnu)>0$ and $\eps_2=\eps_2(\Gamma,E,\tnu)>0$ such that
	\begin{equation}\label{5.2-5}
	c(\Gamma,E,\alpha)\,\eps_2(\Gamma,E,\tnu)\leq\frac{1}{16}\;\mbox{and}\; c(\Gamma,E,\alpha)\,\eps_1(\Gamma,E,\tnu)\leq\frac{\eps_2(\Gamma,E,\tnu)}{2}\;\mbox{for any}\;\alpha\in [0,\tnu].
	\end{equation}
	From \eqref{5.2-3} and \eqref{5.2-5} we conclude that
	\begin{equation}\label{5.2-7}
	\Psi_{\Gamma,E}\;\mbox{maps}\;\ovB_\rms(0,\eps_1(\Gamma,E,\tnu))\times\Omega_\alpha(\eps_2(\Gamma,E,\tnu))\;\mbox{to}\;\Omega_\alpha(\eps_2(\Gamma,E,\tnu))
	\end{equation}
	for any $\alpha\in [0,\tnu]$ and $\tnu\in (0,\nu(\Gamma,E))$. To finish that proof of lemma we prove \eqref{Fix-point-estimate}.
	We note that
	\begin{align}\label{5.2-8}
	\Big\|\Psi_{\Gamma,E}(\bv_0,f)-\Psi_{\Gamma,E}(\bv_0,g)\Big\|_{H^1_\alpha}&=\Big\|\cKm\big(D(f,f)-D(g,g)\big)\Big\|_{H^1_\alpha}\nonumber\\&\leq c(\Gamma,E,\alpha)\|D(f,f)-D(g,g)\|_{H^1_\alpha}.
	\end{align}
	To estimate the $H^1_\alpha$-norm of the right hand side of \eqref{5.2-8}, we use that $D(\cdot,\cdot)$ is bilinear and bounded on $\bX$, which implies that
	$D(f,f)-D(g,g)=D(f-g,f-g)+2D(g,f-g)$. Since $D(\cdot,\cdot)$ is a bounded bilinear map on $\bX$, we obtain that for any $f,g\in H^1_\alpha(\RR_+,\bX)$ we have that $D(g,f-g)\in H^1_\alpha(\RR_+,\bX)$ and $\big(D(g,f-g)\big)'=D(g',f-g)+D(g,f'-g')$. It follows that
	\begin{align}\label{5.2-9}
	\|D(g,f&-g)\|_{H^1_\alpha}^2=\int_{0}^{\infty} e^{2\alpha\tau}\|D(g(\tau),f(\tau)-g(\tau))\|^2\rmd\tau+\int_{0}^{\infty} e^{2\alpha\tau}\|\big(D(g,f-g)\big)'(\tau)\|^2\rmd\tau\nonumber\\
	&\leq\|D\|^2\int_{0}^{\infty} e^{2\alpha\tau}\|g(\tau)\|^2\|f(\tau)-g(\tau))\|^2\rmd\tau+2\|D\|^2\int_{0}^{\infty} e^{2\alpha\tau}\|g'(\tau)\|^2\|f(\tau)-g(\tau))\|^2\rmd\tau\nonumber\\
	&\qquad\qquad\qquad\qquad\qquad\qquad\qquad\qquad+2\|D\|^2\int_{0}^{\infty} e^{2\alpha\tau}\|g(\tau)\|^2\|f'(\tau)-g'(\tau))\|^2\rmd\tau\nonumber\\
	&\leq 2\|D\|^2 \|g\|_{L^\infty_\alpha}^2\|f-g\|_{H^1_\alpha}^2+2\|D\|^2\|f-g\|_{L^\infty_\alpha}^2\|g'\|_{L^2_\alpha}\leq 2\|D\|^2 \|g\|_{H^1_\alpha}^2\|f-g\|_{H^1_\alpha}^2.
	\end{align}
	From \eqref{5.2-2} and \eqref{5.2-9} we obtain that
	\begin{align}\label{5.2-10}
	\|D(f,f)-D(g,g)\|_{H^1_\alpha}&\leq \|D(f-g,f-g)\|_{H^1_\alpha}+2\|D(g,f-g)\|_{H^1_\alpha} \leq 2\|D\|\|f-g\|_{H^1_\alpha}^2\nonumber\\
	&+2\|D\|\|g\|_{H^1_\alpha}\|f-g\|_{H^1_\alpha}\leq 4\|D\|\big(\|f\|_{H^1_\alpha}+\|g\|_{H^1_\alpha}\big)\|f-g\|_{H^1_\alpha}
	\end{align}
	Therefore, from \eqref{5.2-5}, \eqref{5.2-8} and \eqref{5.2-10} have that
	\begin{equation}\label{5.2-11}
	\|\Psi_{\Gamma,E}(\bv_0,f)-\Psi_{\Gamma,E}(\bv_0,g)\|_{H^1_\alpha}\leq 8c(\Gamma,E,\alpha)\eps_2(\Gamma,E,\tnu)\,\|f-g\|_{H^1_\alpha}\leq\frac{1}{2}\|f-g\|_{H^1_\alpha}
	\end{equation}
	for any $\bv_0\in\ovB_\rms(0,\eps_1(\Gamma,E,\tnu))$, $f,g\in\Omega_\alpha(\eps_2(\Gamma,E,\tnu))$, $\alpha\in [0,\tnu]$ and $\tnu\in (0,\nu(\Gamma,E))$, proving the lemma.
	\end{proof}
	Applying the Contraction Mapping Theorem and the results on smooth dependence of fixed point mappings on parameters, from Lemma~\ref{l5.2} we infer that equation $\bu=\Psi_{\Gamma,E}(\bv_0,\bu)$ has a \textit{unique}, local solution denoted $\obu(\cdot;\bv_0)$. Since $\Psi_{\Gamma,E}$ depends linearly on $\bv_0$, we have that $\obu(\cdot;\bv_0)\in\Omega_\alpha(\eps_2(\Gamma,E,\tnu))\subset H^1_\alpha(\RR_+,\bX)$ for any $\alpha\in [0,\tnu]$, $\tnu\in(0,\nu(\Gamma,E))$ and the function
	\begin{equation}\label{def-Solution}
	\Sigma_{\Gamma,E}:\ovB_\rms(0,\eps_1(\Gamma,E,\tnu))\to\Omega_\alpha(\eps_2(\Gamma,E,\tnu))\;\mbox{defined by}\;\Sigma_{\Gamma,E}(\bx_0)=\obu(\cdot,\bv_0)
	\end{equation}
	is of redclass $C^r$ for any $\alpha\in [0,\tnu]$.
	\begin{lemma}\label{l5.3}
	Assume Hypothesis (S) and assume that the linear operator $E$ is negative-definite. Then, the following assertions hold true:
	\begin{enumerate}
	\item[(i)]  $(\cKm f)(0)+P_\rms^{\Gamma,E}E^{-1}f(0)\in\bX_\rmu^{\Gamma,E}$ for any $f\in H^1(\RR_+,\bX)$;
	\item[(ii)] $\obu(\cdot;\bv_0)$ is a mild solution of equation \eqref{h-plus} satisfying the condition
	\begin{equation}\label{u-0}
	P^{\Gamma,E}_\rms\obu(0;\bv_0)=\bv_0-P_\rms^{\Gamma,E}E^{-1}D(\obu(0;\bv_0),\obu(0;\bv_0))\quad\mbox{for any}\quad \bv_0\in\ovB_\rms(0,\eps_1(\Gamma,E,\tnu)).
	\end{equation}
	\end{enumerate}
	\end{lemma}
	\begin{proof} (i) We recall the definition of the subspace $\cH^1_{\Gamma}=\{f:\RR_+\to\bX:\mbox{there exists}\; g\in H^1(\RR_+,\bX)\;\mbox{such that}\;f(\tau)=\Gamma g(\tau)\;\mbox{for any}\;\tau\geq0 \}$, which is dense in $H^1(\RR_+,\bX)$. Let $g\in H^1(\RR_+,\bX)$ and $f=\Gamma g$. From \eqref{4.8-6} we obtain that
	\begin{align}\label{5.3-1}
	(\cKm f)(0)+P_\rms^{\Gamma,E}E^{-1}f(0)&=-\int_{0}^{\infty}T_\rmu^{\Gamma,E}(s)P_\rmu^{\Gamma,E}g(s)\rmd s-S_{\Gamma,E}^{-1}P_\rms^{\Gamma,E}g(0)+P_\rms^{\Gamma,E}E^{-1}\Gamma g(0)\nonumber\\
	&=-\int_{0}^{\infty}T_\rmu^{\Gamma,E}(s)P_\rmu^{\Gamma,E}g(s)\rmd s\in \bX_\rmu^{\Gamma,E}.
	\end{align}
	From Lemma~\ref{l4.8} we infer that the operator $f\to (\cKm f)(0)+P_\rms^{\Gamma,E}E^{-1}f(0):H^1(\RR_+,\bX)\to\bX$ is linear and bounded. Since $\cH^1_{\Gamma}$ is dense in $H^1(\RR_+,\bX)$, from \eqref{5.3-1} we conclude that $(\cKm f)(0)+P_\rms^{\Gamma,E}E^{-1}f(0)\in\bX_\rmu^{\Gamma,E}$ for any $f\in H^1(\RR_+,\bX)$.

	\noindent\textit{Proof of (ii)} Fix $\bv_0\in\ovB_\rms(0,\eps_1(\Gamma,E,\tnu))$. Since $\obu(\cdot;\bv_0)$ is a solution of equation \eqref{mod-fix-point-eq}, from (i) we obtain that
	\begin{equation*}
	P_\rms^{\Gamma,E}\obu(0;\bv_0)=P_\rms^{\Gamma,E}\Big(\bv_0+(\cKm D(\obu(\cdot;\bv_0),\obu(\cdot;\bv_0))(0)\Big)=\bv_0-P_\rms^{\Gamma,E}E^{-1}D(\obu(0;\bv_0),\obu(0;\bv_0)),
	\end{equation*}
	proving \eqref{u-0}. Moreover,
	\begin{align*}
	\obu(\tau;\bv_0)&-T_\rms^{\Gamma,E}(\tau)P_\rms^{\Gamma,E}\obu(0;\bv_0)=T_\rms^{\Gamma,E}(\tau)P_\rms^{\Gamma,E}(\bv_0-\obu(0;\bv_0))+\big(\cKm D(\obu(\cdot;\bv_0),\obu(\cdot;\bv_0)\big)(\tau)\\
	&=T_\rms^{\Gamma,E}(\tau)P_\rms^{\Gamma,E}E^{-1}D(\obu(0;\bv_0),\obu(0;\bv_0))+\big(\cKm D(\obu(\cdot;\bv_0),\obu(\cdot;\bv_0))\big)(\tau)\\&=\big(\cK_{\Gamma,E} D(\obu(\cdot;\bv_0),\obu(\cdot;\bv_0))\big)(\tau)\quad\mbox{for any}\quad\tau\geq0.
	\end{align*}
	From Lemma~\ref{l4.4} we conclude that $\obu(\cdot;\bv_0)$ is a solution of equation \eqref{h-plus}, proving the lemma.
	\end{proof}
	Next, we define $\cJ^{\Gamma,E}_\rms:\ovB_\rms(0,\eps_1)\to \bX$ by
	\begin{equation}\label{def-J-GammaE}
	\cJ^{\Gamma,E}_\rms(\bv_0)=P^{\Gamma,E}_\rmu\obu(0;\bv_0)-P_\rms^{\Gamma,E}E^{-1}D(\obu(0;\bv_0),\obu(0;\bv_0)).
	\end{equation}
	and introduce the manifold
	\begin{equation}\label{def-manifold-GammaE}
	\cM^{\Gamma,E}_\rms=\big\{\bv_0+\cJ^{\Gamma,E}_\rms(\bv_0): \bv_0\in B_\rms(0,\eps_1)\big\}.
	\end{equation}
	Next, we are going to show that the manifold $\cM^{\Gamma,E}_\rms$ is invariant under the flow of equation \eqref{h-plus}. To prove this result we need to study
	the time translations of solutions $\obu(\cdot;\bv_0)$ of equation \eqref{mod-fix-point-eq}. To achieve this goal we need the results of the next lemma.
	\begin{lemma}\label{l5.4}
	Assume Hypothesis (S) and assume that the linear operator $E$ is negative-definite. Then, the following assertions hold true:
	\begin{enumerate}
	\item[(i)] The modified Fourier multiplier $\cKm$ satisfies the translation formula
	\begin{equation}\label{translation-cKm}
	(\cKm f)(\tau+\tau_0)=\big(\cKm f(\cdot+\tau_0)\big)(\tau)+T_\rms^{\Gamma,E}(\tau)\int_{0}^{\tau_0}T_\rms^{\Gamma,E}(\tau_0-s)P_\rms^{\Gamma,E} E^{-1}f'(s)\rmd s
	\end{equation}
	for any $f\in H^1(\RR_+,\bX)$, $\tau,\tau_0\geq 0$;
	\item[(ii)] $T_\rms^{\Gamma,E}(\tau)\bv_0\in \bX^{\Gamma,E}_\rms\cap\bX_{\frac{1}{2}}^{\Gamma,E}$ for any $\tau\geq0$ and $\bv_0\in \bX^{\Gamma,E}_\rms\cap\bX_{\frac{1}{2}}^{\Gamma,E}$. Moreover,
	\begin{equation}\label{X1/2-semigroup}
	\|T_\rms^{\Gamma,E}(\tau)\bv_0\|_{\bX_{\frac{1}{2}}^{\Gamma,E}}\leq e^{-\nu\tau}\|\bv_0\|_{\bX_{\frac{1}{2}}^{\Gamma,E}},\quad \lim_{\tau\to 0+}\|T_\rms^{\Gamma,E}(\tau)\bv_0-\bv_0\|_{\bX_{\frac{1}{2}}^{\Gamma,E}}=0;
	\end{equation}
	\item[(iii)] $\int_{0}^{\tau_0}T_\rms^{\Gamma,E}(\tau_0-s)P_\rms^{\Gamma,E} E^{-1}f(s)\rmd s\in \bX_{\frac{1}{2}}^{\Gamma,E}$ for any $f\in L^2([0,\tau_0],\bX)$ and
	\begin{equation}\label{X1/2-integral estimate}
	\big\|\int_{0}^{\tau_0}T_\rms^{\Gamma,E}(\tau_0-s)P_\rms^{\Gamma,E} E^{-1}f(s)ds\big\|_{\bX_{\frac{1}{2}}^{\Gamma,E}}\leq c(\Gamma,E)\Big(\int_{0}^{\tau_0}\|f(s)\|^2\rmd s\Big)^{\frac{1}{2}}\quad\mbox{for any}\quad f\in L^2([0,\tau_0],\bX).
	\end{equation}
	\end{enumerate}
	\end{lemma}
	\begin{proof} (i) Let $g\in H^1(\RR_+,\bX)$ and $f=\Gamma g$. From \eqref{4.8-1} we have that $\cK_{\Gamma,E}f=\cG_{\Gamma,E}*g$ and $\cK_{\Gamma,E}f(\cdot+\tau_0)=\cG_{\Gamma,E}*g(\cdot+\tau_0)$. It follows that
	\begin{align}\label{5.4-1}
	(\cK_{\Gamma,E}f)(\tau+\tau_0)&=\int_{0}^{\tau+\tau_0}T_\rms^{\Gamma,E}(\tau+\tau_0-s)P_\rms^{\Gamma,E}g(s)\rmd s-\int_{\tau+\tau_0}^{\infty} T_\rmu^{\Gamma,E}(s-\tau-\tau_0)P_\rmu^{\Gamma,E}g(s)\rmd s\nonumber\\
	&=\int_{-\tau_0}^{\tau}T_\rms^{\Gamma,E}(\tau-\theta)P_\rms^{\Gamma,E}g(\theta+\tau_0)\rmd\theta-\int_{\tau}^{\infty} T_\rmu^{\Gamma,E}(\theta-\tau)P_\rmu^{\Gamma,E}g(\theta+\tau_0)\rmd\theta\nonumber\\
	&=\big(\cG_{\Gamma,E}*g(\cdot+\tau_0)\big)(\tau)+\int_{-\tau_0}^{0}T_\rms^{\Gamma,E}(\tau-\theta)P_\rms^{\Gamma,E}g(\theta+\tau_0)\rmd\theta\nonumber\\
	&=\big(\cK_{\Gamma,E}f(\cdot+\tau_0)\big)(\tau)+\int_{0}^{\tau_0}T_\rms^{\Gamma,E}(\tau+\tau_0-s)P_\rms^{\Gamma,E}g(s)\rmd s\;\mbox{for any}\;\tau\geq0.
	\end{align}
	Since $S_\rms^{\Gamma,E}=(S_{\Gamma,E})_{|\bX_\rms^{\Gamma,E}}$ is the generator of the $C^0$-semigroup $\{T_\rms(\tau)\}_{\tau\geq0}$, we infer that the function $\tau\to T_\rms^{\Gamma,E}(\tau)(S_\rms^{\Gamma,E})^{-1}\bx:\RR_+\to\bX$ is of class $C^1$ for any $\bx\in\bX$. Integrating by parts we obtain that
	\begin{align}\label{5.4-2}
	&\int_{0}^{\tau_0}T_\rms^{\Gamma,E}(\tau+\tau_0-s)P_\rms^{\Gamma,E} E^{-1}f'(s)ds=\int_{0}^{\tau_0}T_\rms^{\Gamma,E}(\tau+\tau_0-s)P_\rms^{\Gamma,E} S_{\Gamma,E}^{-1} g'(s)ds\nonumber\\
	&=\int_{0}^{\tau_0}T_\rms^{\Gamma,E}(\tau+\tau_0-s)P_\rms^{\Gamma,E}g(s)\rmd s +T_\rms^{\Gamma,E}(\tau)P_\rms^{\Gamma,E}S_{\Gamma,E}^{-1}g(\tau_0)-T_\rms^{\Gamma,E}(\tau+\tau_0)P_\rms^{\Gamma,E}S_{\Gamma,E}^{-1}g(0)\nonumber\\
	&=\int_{0}^{\tau_0}T_\rms^{\Gamma,E}(\tau+\tau_0-s)P_\rms^{\Gamma,E}g(s)\rmd s+T_\rms^{\Gamma,E}(\tau)P_\rms^{\Gamma,E}E^{-1}f(\tau_0)-T_\rms^{\Gamma,E}(\tau+\tau_0)P_\rms^{\Gamma,E}E^{-1}f(0).
	\end{align}
	From \eqref{5.4-1} and \eqref{5.4-2} we conclude that
	\begin{align}\label{5.4-3}
	&(\cKm f)(\tau+\tau_0)=(\cK_{\Gamma,E}f)(\tau+\tau_0)-T_\rms^{\Gamma,E}(\tau+\tau_0)P_\rms^{\Gamma,E}E^{-1}f(0)\nonumber\\
	&=\big(\cK_{\Gamma,E}f(\cdot+\tau_0)\big)(\tau)+\int_{0}^{\tau_0}T_\rms^{\Gamma,E}(\tau+\tau_0-s)P_\rms^{\Gamma,E}g(s)\rmd s-T_\rms^{\Gamma,E}(\tau+\tau_0)P_\rms^{\Gamma,E}E^{-1}f(0)\nonumber\\
	&=\big(\cKm f(\cdot+\tau_0)\big)(\tau)+\int_{0}^{\tau_0}T_\rms^{\Gamma,E}(\tau+\tau_0-s)P_\rms^{\Gamma,E}g(s)\rmd s-T_\rms^{\Gamma,E}(\tau+\tau_0)P_\rms^{\Gamma,E}E^{-1}f(0)\nonumber\\
	&\qquad\qquad\qquad+T_\rms^{\Gamma,E}(\tau)P_\rms^{\Gamma,E}E^{-1}f(\tau_0)\nonumber\\
	&=\big(\cKm f(\cdot+\tau_0)\big)(\tau)+T_\rms^{\Gamma,E}(\tau)\int_{0}^{\tau_0}T_\rms^{\Gamma,E}(\tau_0-s)P_\rms^{\Gamma,E} E^{-1}f'(s)\rmd s
	\end{align}
	for any $f\in\cH^1_{\Gamma}$. Since $\cH^1_{\Gamma}$ is dense in $H^1(\RR_+,\bX)$, for any $f\in H^1(\RR_+,\bX)$ there exists $\{f_n\}_{n\geq 1}$ a sequence of functions in $\cH^1_{\Gamma}$ such that $f_n\to f$ as $n\to\infty$ in $H^1(\RR_+,\bX)$. It follows that $f_n(\cdot+\tau_0)\to f(\cdot+\tau_0)$ as $n\to\infty$ in $H^1(\RR_+,\bX)$. Since $\cKm$ is a bounded linear operator on $H^1(\RR_+,\bX)$ by Lemma~\ref{l4.8}, we infer that
	\begin{equation}\label{5.4-4}
	(\cKm f_n)\to (\cKm f)\;\mbox{and}\;\cKm f_n(\cdot+\tau_0)\to\cKm f(\cdot+\tau_0)\;\mbox{as}\; n\to\infty\;\mbox{in}\; H^1(\RR_+,\bX).
	\end{equation}
	Moreover, one can readily check that
	\begin{align}\label{5.4-5}
	&\Big\|  \int_{0}^{\tau_0}T_\rms^{\Gamma,E}(\tau_0-s)P_\rms^{\Gamma,E} E^{-1}f_n'(s)\rmd s -\int_{0}^{\tau_0}T_\rms^{\Gamma,E}(\tau_0-s)P_\rms^{\Gamma,E} E^{-1}f'(s)ds\Big\|\nonumber\\
	&\quad\leq c(\Gamma,E)\Big(\int_{0}^{\tau_0} e^{-2\nu(\tau_0-s)}\rmd s\Big)^{1/2}\, \|f_n'-f'\|_2\leq c(\Gamma,E)\|f_n-f\|_{H^1}\quad\mbox{for any}\quad n\geq 1.
	\end{align}
	Assertion (i) follows from \eqref{5.4-3}, \eqref{5.4-4} and \eqref{5.4-5}.

	\noindent\textit{Proof of (ii)}. Let $\bv_0\in\bX^{\Gamma,E}_\rms\cap\bX_{\frac{1}{2}}^{\Gamma,E}$. We recall that from \eqref{representation-stable-unstable} we obtain that $\tg_0=U_{\Gamma,E}\bv_0\in L^2(\Lambda,\mu)$ and $\tg_0(\lambda)=0$ for any $\lambda\in\Lambda_+$. From \eqref{representation-bi-semigroup} we infer that
	\begin{align}\label{5.4-6}
	&\int_\Lambda |H_{\Gamma,E}(\lambda)|\,|(U_{\Gamma,E}T_\rms^{\Gamma,E}(\tau)\bv_0)(\lambda)|^2\rmd\mu(\lambda)=\int_\Lambda |H_{\Gamma,E}(\lambda)|\,|(\tT_\rms^{\Gamma,E}(\tau)\tg_0)(\lambda)|^2\rmd\mu(\lambda)\nonumber\\
	&=\int_{\Lambda-} e^{2\tau H_{\Gamma,E}(\lambda)}|H_{\Gamma,E}(\lambda)|\,|\tg_0(\lambda)|^2\rmd\mu(\lambda)\leq e^{-2\nu\tau}\|\bv_0\|_{\bX_{\frac{1}{2}}^{\Gamma,E}}^2\;\mbox{for any}\;\tau\geq0.
	\end{align}
	From \eqref{5.4-6} we conclude that $T_\rms^{\Gamma,E}(\tau)\bv_0\in \bX^{\Gamma,E}_\rms\cap\bX_{\frac{1}{2}}^{\Gamma,E}$ and $\|T_\rms^{\Gamma,E}(\tau)\bv_0\|_{\bX_{\frac{1}{2}}^{\Gamma,E}}\leq e^{-\nu\tau}\|\bv_0\|_{\bX_{\frac{1}{2}}^{\Gamma,E}}$ for any $\tau\geq0$. Moreover, using \eqref{representation-bi-semigroup} again we obtain that
	\begin{align}\label{5.4-7}
	\big\|T_\rms^{\Gamma,E}(\tau)\bv_0-\bv_0\big\|_{\bX_{\frac{1}{2}}^{\Gamma,E}}^2&=\int_\Lambda |H_{\Gamma,E}(\lambda)|\,\Big|\Big(U_{\Gamma,E}\big(T_\rms^{\Gamma,E}(\tau)\bv_0-\bv_0)\Big)(\lambda)\Big|^2\rmd\mu(\lambda)\nonumber\\
	&=\int_{\Lambda-} \big(1-e^{2\tau H_{\Gamma,E}(\lambda)}\big)|H_{\Gamma,E}(\lambda)|\,|\tg_0(\lambda)|^2\rmd\mu(\lambda)
	\end{align}
	Passing to the limit as $\tau\to 0$ in \eqref{5.4-7}, from Lebesgue Dominated Convergence Theorem, it follows that $\lim_{\tau\to 0+}\|T_\rms^{\Gamma,E}(\tau)\bv_0-\bv_0\|_{\bX_{\frac{1}{2}}^{\Gamma,E}}=0$, proving (ii).

	\noindent\textit{Proof of (iii)}. First, we introduce the function $\tf:\RR_+\to L^2(\Lambda,\mu)$ by $\tf(\tau)=U_{\Gamma,E}P_\rms^{\Gamma,E}E^{-1}f(\tau)$ and let
	$\tth_0=U_{\Gamma,E}\big(\int_{0}^{\tau_0}T_\rms^{\Gamma,E}(\tau_0-s)P_\rms^{\Gamma,E} E^{-1}f(s)\rmd s\big)$. To simplify the notation, in what follows we  denote by
	$\tf(\tau,\lambda)=\big[\tf(\tau)\big](\lambda)$. Since $f\in L^2([0,\tau_0],\bX)$ we infer that $\tf\in L^2([0,\tau_0],L^2(\Lambda,\mu))$, thus
	\begin{equation}\label{5.4-8}
	\int_{0}^{\tau_0}\int_\Lambda |\tf(\tau,\lambda)|^2\rmd\mu(\Lambda)\rmd\tau=\|\tf\|_{L^2([0,\tau_0],L^2(\Lambda,\mu))}^2\leq c(\Gamma,E)\int_{0}^{\tau_0}\|f(s)\|^2\rmd s
	\end{equation}
	From \eqref{representation-bi-semigroup} one can readily check that
	\begin{equation}\label{5.4-9}
	\tth_0(\lambda)=\int_{0}^{\tau_0} \big[\tT_\rms^{\Gamma,E}(\tau_0-s)\tf(s)\big](\lambda)\rmd s=\int_{0}^{\tau_0} e^{(\tau_0-s)H_{\Gamma,E}(\lambda)}\tf(s,\lambda)\rmd s\quad\mbox{for any}\quad\lambda\in\Lambda.
	\end{equation}
	From \eqref{5.4-9} we obtain that
	\begin{align}\label{5.4-10}
	|\tth_0(\lambda)|^2&\leq \int_{0}^{\tau_0} e^{2(\tau_0-s)H_{\Gamma,E}(\lambda)}\rmd s\int_{0}^{\tau_0}|\tf(s,\lambda)|^2\rmd s=\frac{1-e^{2\tau H_{\Gamma,E}(\lambda)}}{2|H_{\Gamma,E}(\lambda)|}\int_{0}^{\tau_0}|\tf(s,\lambda)|^2\rmd s\nonumber\\
	&\leq\frac{1}{2|H_{\Gamma,E}(\lambda)|}\int_{0}^{\tau_0}|\tf(s,\lambda)|^2\rmd s\quad\mbox{for any}\quad\lambda\in\Lambda.
	\end{align}
	From \eqref{5.4-8} and \eqref{5.4-10} it follows that
	\begin{equation}\label{5.4-11}
	\int_{\Lambda}|H_{\Gamma,E}(\lambda)|\,|\tth_0(\lambda)|^2\rmd\mu(\lambda)\leq\frac{1}{2}\int_\Lambda\int_{0}^{\tau_0} |\tf(\tau,\lambda)|^2\rmd\tau\rmd\mu(\lambda)\leq c(\Gamma,E)\int_{0}^{\tau_0}\|f(s)\|^2\rmd s<\infty.
	\end{equation}
	We conclude that $\tth_0\in\dom(|H_{\Gamma,E}|^{1/2})$ and $\|\tth_0\|_{\dom(|M_{\Gamma,E}|^{1/2})}\leq c(\Gamma,E)\|f\|_{L^2([0,\tau_0],\bX)}$, proving the lemma.
	\end{proof}

	\begin{lemma}\label{l5.5}
	Assume Hypothesis (S) and assume that the linear operator $E$ is negative-definite. Then, the manifold $\cM^{\Gamma,E}_\rms$ is locally invariant under the flow of equation
	\eqref{h-plus}.
	\end{lemma}
	\begin{proof}
	Let $\tnu\in (0,\nu(\Gamma,E))$, $\alpha\in [0,\tnu]$ and assume $\bu_0$ is a $H^1_\alpha$ solution of equation \eqref{h-plus} such that $\bu_0(0)\in\cM^{\Gamma,E}_\rms$.
	Then, there exists $\bv_0\in\ovB(0,\eps_1)$ such that $\bu_0(0)=\bv_0+\cJ^{\Gamma,E}_\rms(\bv_0)$. From \eqref{u-0} and \eqref{def-J-GammaE} we obtain that $\bu_0(0)=\obu(0;\bv_0)$.
	Since $\bu_0$ is a $H^1_\alpha$ solution of equation \eqref{h-plus}, we have that $\bu_0$ satisfies equation \eqref{fix-point-eq}. It follows that
	\begin{align}\label{5.5-1}
	\bu_0(\tau)&=T_\rms^{\Gamma,E}(\tau)P_\rms^{\Gamma,E}\bu_0(0)+\big(\cK_{\Gamma,E}D(\bu_0,\bu_0)\big)(\tau)\nonumber\\
	&=T_\rms^{\Gamma,E}(\tau)\big(P_\rms^{\Gamma,E}\bu_0(0)+P_\rms^{\Gamma,E}E^{-1}D(\bu_0(0),\bu_0(0)) \big)+\big(\cKm D(\bu_0,\bu_0)\big)(\tau)\nonumber\\
	&=T_\rms^{\Gamma,E}(\tau)\big(P_\rms^{\Gamma,E}\obu(0;\bv_0)+P_\rms^{\Gamma,E}E^{-1}D(\obu(0;\bv_0),\obu(0;\bv_0)) \big)+\big(\cKm D(\bu_0,\bu_0)\big)(\tau)\nonumber\\
	&=T_\rms^{\Gamma,E}(\tau)P_\rms^{\Gamma,E}\bv_0+\big(\cKm D(\bu_0,\bu_0)\big)(\tau)\quad\mbox{for any}\quad\tau\geq 0.
	\end{align}
	From Lemma~\ref{l5.2} we infer that equation $\bu=\Psi_{\Gamma,E}(\bv_0,\bu)$ has a unique solution, which implies that
	$\bu_0=\obu(\cdot;\bv_0)$. Fix $\tau_0\geq0$ and let $\bu_{\tau_0}:\RR_+\to\bX$ be the function defined by $\bu_{\tau_0}(\tau)=\bu_0(\tau+\tau_0)=\obu(\tau+\tau_0;\bv_0)$. Since $\obu(\cdot;\bv_0)\in H^1_\alpha(\RR_+,\bX)$ for any $\alpha\in [0,\tnu]$, it follows that $u_{\tau_0}\in H^1_\alpha(\RR_+,\bX)$ for any $\alpha\in [0,\tnu]$.
	From Lemma~\ref{l5.4}(i), \eqref{u-0}, \eqref{5.5-1} and since $\bu_0=\obu(\cdot;\bv_0)$ we conclude that
	\begin{align}\label{5.5-2}
	\bu_{\tau_0}(\tau)&=\bu_0(\tau+\tau_0)=T^{\Gamma,E}_\rms(\tau+\tau_0)\bv_0+\Big(\cKm D(\bu_0,\bu_0)\Big)(\tau+\tau_0)\nonumber\\
	&=T^{\Gamma,E}_\rms(\tau)T_\rms^{\Gamma,E}(\tau_0)\bv_0+T_\rms^{\Gamma,E}(\tau)\int_{0}^{\tau_0}T_\rms^{\Gamma,E}(\tau_0-s)P_\rms^{\Gamma,E} E^{-1}\big(D(\bu_0(s),\bu_0(s))\big)'\rmd s\nonumber\\
	&\qquad+\big(\cKm D(\bu_{\tau_0},\bu_{\tau_0})\big)(\tau)\nonumber\\
	&=T^{\Gamma,E}_\rms(\tau)\bv_1+\big(\cKm D(\bu_{\tau_0},\bu_{\tau_0})\big)(\tau)\quad\mbox{for any}\quad\tau\geq 0.
	\end{align}
	where $\bv_1=T_\rms^{\Gamma,E}(\tau_0)\bv_0+\int_{0}^{\tau_0}T_\rms^{\Gamma,E}(\tau_0-s)P_\rms^{\Gamma,E} E^{-1}\big(D(\bu_0(s),\bu_0(s))\big)'\rmd s$. From Lemma~\ref{l5.4}(ii) and (iii) we infer that $\bv_1\in\bX_\rms^{\Gamma,E}\cap\bX_{\frac{1}{2}}^{\Gamma,E}$. Moreover,
	\begin{align}\label{5.5-3}
	\|\bv_1-\bv_0\|_{\bX_{\frac{1}{2}}^{\Gamma,E}}&\leq \|T_\rms^{\Gamma,E}(\tau_0)\bv_0-\bv_0\|_{\bX_{\frac{1}{2}}^{\Gamma,E}}+c(\Gamma,E)\Big(\int_{0}^{\tau_0}\|D(\bu_0(s),\bu_0'(s))\|^2\rmd s\Big)^{\frac{1}{2}}\nonumber\\
	&\leq\|T_\rms^{\Gamma,E}(\tau_0)\bv_0-\bv_0\|_{\bX_{\frac{1}{2}}^{\Gamma,E}}+c(\Gamma,E)\|\bu_0\|_{H^1}\Big(\int_{0}^{\tau_0}\|\bu_0'(s)\|^2\rmd s\Big)^{\frac{1}{2}}.
	\end{align}
	From Lemma~\ref{l5.4}(ii) and (iii) we infer that there exists $\tau_1>0$ such that $\bv_1=\bv_1(\tau_0)\in B_\rms(0,\eps_1)$ for any $\tau_0\in [0,\tau_1]$.
	Since equation $\bu=\Psi_{\Gamma,E}(\bv_1,\bu)$ has a unique solution in $\Omega_{\Gamma,E,\alpha}(\eps_2)$ for any $\alpha\in [0,\tnu]$, from \eqref{5.5-2} we conclude that $\bu_{\tau_0}=\obu(\cdot;\bv_1)$. From Lemma~\ref{l5.3}(ii) it follows that
	\begin{align}\label{5.5-4}
	\bu_0(\tau_0)&=\bu_{\tau_0}(0)=\obu(0;\bv_1)=P^{\Gamma,E}_\rms\obu(0;\bv_1)+P^{\Gamma,E}_\rmu\obu(0;\bv_1)\nonumber\\
	&=\bv_1-P_\rms^{\Gamma,E}E^{-1}D(\obu(0;\bv_1),\obu(0;\bv_1))+P^{\Gamma,E}_\rmu\obu(0;\bv_1)\nonumber\\
	&=\bv_1+\cJ^{\Gamma,E}_\rms(\bv_1)\in\cM^{\Gamma,E}_\rms\quad\mbox{for any}\quad \tau_0\in [0,\tau_1],
	\end{align}
	proving the lemma.
	\end{proof}
	In the next lemma we prove that the nonlinear manifold $\cM_\rms^{\Gamma,E}$ is tangent to the corresponding linear stable subspace $\bX_\rms^{\Gamma,E}$.
	\begin{lemma}\label{l5.6}
	Assume Hypothesis (S) and assume that the linear operator $E$ is negative-definite. Then, $\big(\cJ_\rms^{\Gamma,E}\big)'(0)=0$ which proves that the manifold $\cM^{\Gamma,E}_\rms$
	is tangent to $\bX_\rms^{\Gamma,E}$.
	\end{lemma}
	\begin{proof} First, we compute $\Sigma_{\Gamma,E}'(0)$, where $\Sigma_{\Gamma,E}$ is defined in \eqref{def-Solution}. Differentiating with respect to $\bv_0$ in the fixed point equation \eqref{mod-fix-point-eq}, we obtain that
	\begin{equation}\label{5.6-1}
	\Sigma_{\Gamma,E}'(\bv_0)=\pa_{\bv_0}\Psi_{\Gamma,E}(\bv_0,\Sigma_{\Gamma,E}(\bv_0))+\pa_{\bu}\Psi_{\Gamma,E}(\bv_0,\Sigma_{\Gamma,E}(\bv_0))\Sigma_{\Gamma,E}'(\bv_0)\;\mbox{for any}\;\bv_0\in\ovB_\rms(0,\eps_1).
	\end{equation}
	Since the equation $\bu=\Psi_{\Gamma,E}(0,\bu)$ has a unique solution, and $0$ trivially satisfies the equation, we infer that $\Sigma_{\Gamma,E}(0)=0$. Since $\Psi_{\Gamma,E}(\bv_0,0)=T_\rms^{\Gamma,E}(\cdot)P_\rms^{\Gamma,E}\bv_0$ for any $\bv_0\in\ovB_\rms(0,\eps_1)$ we have that $\big[\pa_{\bv_0}\Psi_{\Gamma,E}(0,0)\big]\bv_1=T_\rms^{\Gamma,E}(\cdot)P_\rms^{\Gamma,E}\bv_1$ for any $\bv_1\in\bX_\rms^{\Gamma,E}\cap\bX_{\frac{1}{2}}^{\Gamma,E}$. From \eqref{5.2-2} we have that
	\begin{equation}\label{5.6-2}
	\|\Psi_{\Gamma,E}(0,\bu)\|_{H^1_\alpha}=\|\cKm D(\bu,\bu)\|_{H^1_\alpha}\leq c(\Gamma,E)\| D(\bu,\bu)\|_{H^1_\alpha}\leq c(\Gamma,E)\|\bu\|_{H^1_\alpha}^2
	\end{equation}
	for any $\bu\in\Omega_{\alpha}(\eps_2)$, which implies that $\pa_{\bu}\Psi_{\Gamma,E}(0,0)=0$. From \eqref{5.6-1} it follows that
	$\big[\Sigma_{\Gamma,E}'(0)\big]\bv_1=T_\rms^{\Gamma,E}(\cdot)P_\rms^{\Gamma,E}\bv_1$ for any $\bv_1\in\bX_\rms^{\Gamma,E}\cap\bX_{\frac{1}{2}}^{\Gamma,E}$. Since the linear operator
	$\mathrm{Tr}_0:H^1_\alpha(\RR_+,\bX)\to\bX$ defined by $\mathrm{Tr}_0f=f(0)$ is bounded, we have that
	\begin{equation}\label{5.6-3}
	\big[\pa_{\bv_0}\obu(0;0)\big]\bv_1=\mathrm{Tr}_0\Sigma_{\Gamma,E}'(0)\bv_1=P_\rms^{\Gamma,E}\bv_1\quad\mbox{for any}\quad\bv_1\in\bX_\rms^{\Gamma,E}\cap\bX_{\frac{1}{2}}^{\Gamma,E}.
	\end{equation}
	Since the $D(\cdot,\cdot)$ is a bounded, bilinear map on $\bX$, from \eqref{def-J-GammaE} and \eqref{5.6-3} we obtain that
	\begin{align}\label{5.6-4}
	\Big[\big(\cJ_\rms^{\Gamma,E}\big)'(0)\Big]\bv_1&=P^{\Gamma,E}_\rmu\big[\pa_{\bv_0}\obu(0;0)\big]\bv_1-2P_\rms^{\Gamma,E}E^{-1}D\Big(\obu(0;0),\big[\pa_{\bv_0}\obu(0;0)\big]\bv_1\Big)&\nonumber\\
	&=P_\rmu^{\Gamma,E}P_\rms^{\Gamma,E}\bv_1-2P_\rms^{\Gamma,E}E^{-1}D(0,P_\rms^{\Gamma,E}\bv_1)=0\quad\mbox{for any}\quad\bv_1\in\bX_\rms^{\Gamma,E}\cap\bX_{\frac{1}{2}}^{\Gamma,E}.
	\end{align}
	From \eqref{def-manifold-GammaE} it follows immediately that the manifold $\cM_\rms^{\Gamma,E}$ is tangent to $\bX_\rms^{\Gamma,E}\cap\bX_{\frac{1}{2}}^{\Gamma,E}$ at $\bv_0=0$.
	\end{proof}
	\begin{theorem}\label{t5.7}
	Assume Hypothesis (S) and assume that the linear operator $E$ is negative-definite. Then, there exists a local stable manifold $\cM^{\Gamma,E}_\rms$ near the origin, tangent to $\bX_\rms^{\Gamma,E}\cap\bX_{\frac{1}{2}}^{\Gamma,E}$, locally invariant under the flow of equation $\Gamma\bu_\tau=E\bu+D(\bu,\bu)$ and uniquely  determined by the property that $\bu\in H^1(\RR_+,\bX)$.
	\end{theorem}
	\begin{proof}
	The theorem follows from Lemma~\ref{l5.3}(ii) and Lemma~\ref{l5.5} and the fact that the equation $\bu=\Psi_{\Gamma,E}(\bv_0,\bu)$ for any $\bv_0\in\ovB_\rms(0,\eps_1)$ has a unique solution on $H^1_\alpha(\RR_+,\bX)$ for any $\alpha\in [0,\tilde \nu]$ and $\tnu\in (0,\nu(\Gamma,E))$.
	\end{proof}
    \begin{remark}\label{r5.8} When proving results on existence of
	nonlinear center-stable/unstable manifolds in the case of differential equation on finite dimensional spaces, the manifolds can be expressed as graphs of $C^1$ functions from $\bH_{\rms/\rmu}$ to $\bH_{\rmu/\rms}\oplus\bH_{\rmc}$, where $\bH_\rms$, $\bH_\rmu$ and $\bH_\rmc$ are the linear stable, unstable and center subspaces of the linearization along the equilibria at $+\infty$. In our case we can prove a similar result assuming Hypothesis (S) and that the linear operator $E$ is negative-definite. Indeed,
	from the definitions of the function $\cJ_\rms^{\Gamma,E}$ in \eqref{def-J-GammaE} and of the manifold $\cM_\rms^{\Gamma,E}$ in \eqref{def-manifold-GammaE} we have that
	\begin{equation}\label{5.8-1}
	\cM^{\Gamma,E}_\rms=\big\{\bv_0-P_\rms^{\Gamma,E}E^{-1}D(\obu(0;\bv_0),\obu(0;\bv_0))+P^{\Gamma,E}_\rmu\obu(0;\bv_0): \bv_0\in B_\rms(0,\eps_1)\big\}.
	\end{equation}
	To prove the claim, we use the Inverse Function Theorem to solve for $\bv_0$ in the $\bX_\rms^{\Gamma,E}$ component of elements of the manifold. Since the $D(\cdot,\cdot)$ is a bounded, bilinear map on $\bX$, from \eqref{5.6-3} we obtain that the Fr\'echet derivative of the function $\cY_\rms^{\Gamma,E}:\{\bv_0\in\bX_\rms^{\Gamma,E}:\|\bv_0\|\leq\eps_1\}\to\bX_\rms^{\Gamma,E}$ defined by $\cY_\rms^{\Gamma,E}(\bv_0)=\bv_0-P_\rms^{\Gamma,E}E^{-1}D(\obu(0;\bv_0),\obu(0;\bv_0))$, is $\big(\cY_\rms^{\Gamma,E}\big)'(0)=I_{X_\rms^{\Gamma,E}}$. It follows that $\eps_1$ can be chosen small enough such that the function $\cY_\rms^{\Gamma,E}$ is invertible on $\{\bv_0\in\bX_\rms^{\Gamma,E}:\|\bv_0\|\leq\eps_1\}$. From \eqref{5.8-1} we obtain that
	\begin{equation}\label{5.8-2}
	\cM^{\Gamma,E}_\rms=\big\{\bv_1+P^{\Gamma,E}_\rmu\obu\big(0;(\cY_\rms^{\Gamma,E})^{-1}(\bv_1)\big)\in\bX_\rms^{\Gamma,E}\oplus\bX_\rmu^{\Gamma,E}: \bv_1\in \cY_\rms^{\Gamma,E}\big(B_\rms(0,\eps_1)\big)\big\}.
	\end{equation}
	Thus, $\cM_\rms^{\Gamma,E}=\mathrm{Graph}(\tilde{\cJ}_\rms^{\Gamma,E})$, where the function $\tilde{\cJ}_\rms^{\Gamma,E}:\cY_\rms^{\Gamma,E}\big(B_\rms(0,\eps_1)\big)\to\bX_\rmu^{\Gamma,E}$ is defined by $\tilde{\cJ}_\rms^{\Gamma,E}(\bv_1)=P^{\Gamma,E}_\rmu\obu\big(0;(\cY_\rms^{\Gamma,E})^{-1}(\bv_1)\big)$.
	\end{remark}

\begin{remark}\label{r5.9} We note that the domain of the function $\tilde{\cJ}_\rms^{\Gamma,E}$ is not equal to the image of the map $\cY_\rms^{\Gamma,E}$, since $B_\rms(0,\eps_1)$ denotes the open ball of radius $\eps_1$ in $\bX_\rms^{\Gamma,E}\cap\bX_{\frac{1}{2}}^{\Gamma,E}$ and not in $\bX_\rms^{\Gamma,E}$.
It would be interesting to know if $\cY_\rms^{\Gamma,E}\big(B_\rms(0,\eps_1)\big)$ contains an open ball of $\bX_{\frac{1}{2}}^{\Gamma,E}$ or some other dense subspace of $\bX$, but we do not pursue this here. \end{remark}
	
Using Theorem~\ref{t5.7} we can now prove the main result of this paper, the existence of center-stable and center-unstable manifolds of equation $Au_\tau=Q(u)$ near the equilibria $u^\pm$ at $\pm\infty$. We recall that the linear operator $S_\pm^\rmr=\big(A_{22}-A_{21}A_{11}^{-1}A_{12}\big)^{-1}Q'_{22}(u^\pm)$ generates a stable bi-semigroup on $\bV$ (Theorem~\ref{c3.8}(iv)) and that equation $Au'=Q'(u^\pm)u$ has an exponential trichotomy on $\bH$, with stable/unstable subspaces denoted $\bH^{\rms/\rmu}_\pm$ and center subspace $\bV^\perp$ (Theorem~\ref{c3.8}(vi)). Moreover, we recall that the pair $(\Gamma,E)=(A_{22}-A_{21}A_{11}^{-1}A_{12},Q'_{22}(u^\pm))$ satisfies Hypothesis (S) by Theorem~\ref{c3.8}(ii) and $Q'_{22}(u^\pm)$ is negative-definite by Hypothesis (H3). In this case we have that $\bX_{\frac{1}{2}}^{\Gamma,E}=\dom(|S_{\Gamma,E}|^{\frac{1}{2}})=\dom\big(|(A_{22}-A_{21}A_{11}^{-1}A_{12})^{-1}Q_{22}'(u^\pm)|^{\frac{1}{2}}\big)$.
	Finally, we introduce $-\nu_\pm:=-\nu\big(A_{22}-A_{21}A_{11}^{-1}A_{12},Q'_{22}(u^\pm)\big)<0$ the decay rate of the bi-semigroup generated by the pair $\big(A_{22}-A_{21}A_{11}^{-1}A_{12},Q'_{22}(u^\pm)\big)$.

\vspace{0.3cm}

\noindent{\bf Proof of Theorem 1.3}.
Making the change of variables $w^\pm=u-u^\pm$ in equation $Au_\tau=Q(u)$ and denoting by $h^\pm=P_{\bV^\perp}w^\pm$ and $v=P_{\bV}w^\pm$, we obtain the system $\tilde{A}v^\pm_\tau=Q'_{22}(u^\pm)v^\pm+D(v^\pm,v^\pm)$, $h^\pm=-A_{11}^{-1}A_{12}v^\pm$, as shown in Section~\ref{s4}. Here $\tilde{A}=A_{22}-A_{21}A_{11}^{-1}A_{12}$ and the bilinear map $D:\bV\times\bV\to\bV$ is defined by $D(v_1,v_2)=B(v_1-A_{11}^{-1}A_{12}v_1,v_2-A_{11}^{-1}A_{12}v_2)$, is bilinear and bounded on $\bV$. Since the linear operator $S_\pm^\rmr$ generates a stable bi-semigroup on $\bV$, the theorem follows from Theorem~\ref{c3.8}, Theorem~\ref{t5.7} and Remark~\ref{r5.8}.

	\vspace{0.3cm}

Next, we show how we can use Theorem~\ref{t5.7} to prove that solutions $u^*$ of equation $Au_\tau=Q(u)$ that converge at $\pm\infty$ to equilibria $u^\pm$ decay exponentially at $\pm\infty$.

	\vspace{0.3cm}

	\noindent{\bf Proof of Corollary 1.4}.
	First, we note that since $u^*$ is a solution of equation $Au_\tau=Q(u)$, we have that $\tilde{A}(v^*-v^\pm)'=Q'_{22}(u^\pm)(v^*-v^\pm)+D(v^*-v^\pm,v^*-v^\pm)$ and $h^*=-A_{11}^{-1}A_{12}v^*$.  Using the uniqueness property of solutions along the manifold $\cM_{\rms/\rmu}^\pm$ given by Theorem~\ref{t5.7}, we conclude that $v^*-v^\pm\in H^1_\alpha(\RR_\pm,\bV)$ for any $\alpha\in [0,\tnu]$, for some $\tnu\in (0,\min\{\nu_+,\nu_-\})$. Since $h^*=-A_{11}^{-1}A_{12}v^*$ we obtain that $u^*-u^\pm=(v^*-v^\pm)-A_{11}^{-1}A_{12}(v^*-v^\pm)$, which implies that $u^*-u^\pm\in H^1_\alpha(\RR_\pm,\bH)$, proving the corollary.


\vspace{0.3cm}

	\begin{thebibliography}{GMWZ7}

	\bibitem{AM1} A.\ Abbondandolo, P.\ Majer, \textit{Ordinary differential operators in Hilbert spaces and Fredholm pairs}, Math. Z. \textbf{243}
	(2003) 525--562.

	\bibitem{AM2} A.\ Abbondandolo, P.\ Majer, \textit{Morse homology on Hilbert spaces}, Comm. Pure Appl. Math. \textbf{54} (2001) 689--760.


\bibitem{BGK} H.\ Bart, I.\ Gohberg, M.\ A.\ Kaashoek,\textit{
Wiener-Hopf factorization, inverse Fourier transforms and exponentially dichotomous operators}.
J. Funct. Anal. \textbf{68} (1986), no. 1, 1-42.

\bibitem{BR} G.\ Boillat and T.\  Ruggeri,
	{\it On the shock structure problem for hyperbolic system of balance laws and convex entropy,}
Continuum Mechanics and Thermodynamics 10 (5), 285-292.

	\bibitem{CN} R. Caflisch and B. Nicolaenko,
	{\it Shock profile solutions of the Boltzmann equation,}
	Comm. Math. Phys.  \textbf{86}  (1982), no. 2, 161--194.

	\bibitem{CLL} Gui Qiang Chen, C. David Levermore, and Tai-Ping Liu,
		{\it Hyperbolic conservation laws with stiff relaxation terms and entropy,} Comm. Pure Appl. Math. \textbf{47} (1994) 787--830.


	\bibitem{DY} A. Dressler and W.-A. Yong, {\textit Existence of Traveling-Wave Solutions for Hyperbolic Systems of Balance Laws,} Arch. Rational Mech. Anal. \textbf{182} (2006) 49--75.

	\bibitem{GZ} R.A. Gardner and K. Zumbrun,
	{\it The gap lemma and geometric criteria for instability of viscous shock
	profiles,} Comm. Pure Appl. Math. \textbf{51} (1998), no. 7, 797--855.

	\bibitem{H} J.H\"arterich, {\it Viscous profiles for traveling waves of scalar balance laws: The canard case,}
	Methods and Applications of Analysis (2003), \textbf{10}, No.1, p.97-118.

	\bibitem{LMNPZ} C. Lattanzio, C. Mascia, T. Nguyen, R. Plaza, and K. Zumbrun, {\it Stability of scalar radiative shock profiles,} SIAM J. Math. Anal. \textbf{41} (2009/10),  no. 6, 2165--2206.

	\bibitem{LPS} Y.\ Latushkin, A.\ Pogan, R.\ Schnaubelt, {\it  Dichotomy and Fredholm properties of evolution equations}, J. Operator Theory \textbf{58} (2007), no. 2, 387-414.

	\bibitem{LP2}  Y.\ Latushkin, A.\ Pogan, \textit{The dichotomy theorem for evolution bi-families}. J. Diff. Eq. \textbf{245} (2008), no. 8, 2267--2306.

	\bibitem{LP3}  Y.\ Latushkin, A.\ Pogan, \textit{The Infinite Dimensional Evans Function}. J. Funct Anal., \textbf{268} (2015), no. 6, 1509--1586.

	\bibitem{LiuYu1} T. P.\ Liu, S. H.\ Yu,
	{\it Boltzmann equation: micro-macro decompositions
	and positivity of shock profiles},
	Comm. Math. Phys. \textbf{246} (2004), no. 1, 133--179.

	\bibitem{LiuYu} T. P.\ Liu, S. H.\ Yu, \textit{Invariant Manifolds for Steady Boltzmann
	Flows and Applications}, Arch. Rational Mech. Anal. \textbf{209} (2013) 869--997.

	\bibitem{Mallet-Paret} J.\ Mallet-Paret, \textit{The Fredholm alternative for functional-differential equations of mixed type}, J. Dyn. Diff.
	Eq. \textbf{11} (1999) 1--47.

	\bibitem{MasZ0} C.\ Mascia and K.\ Zumbrun, \textit{Pointwise Green's function bounds and stability of relaxation shocks,} Indiana Univ. Math. J. \textbf{51} (2002), no. 4, 773--904.

	\bibitem{MasZ1} C.\ Mascia, K.\ Zumbrun, \textit{Spectral stability of weak relaxation shock profiles}, Comm. Part. Diff. Eq. \textbf{34} (2009), no. 2, 119--136.

	\bibitem{MaZ2} C.\ Mascia, K.\ Zumbrun, \textit{Stability of large-amplitude shock profiles of
general relaxation systems}, SIAM J. Math. Anal. \textbf{37} (2005), no. 3, 889--913.

\bibitem{MaZ3} C. Mascia and K. Zumbrun,
{\it Pointwise Green's function bounds for shock profiles
with degenerate viscosity,}
Arch. Ration. Mech. Anal. \textbf{169} (2003) 177--263.

\bibitem{MaZ4} C. Mascia and K. Zumbrun, \textit{Stability of large-amplitude viscous shock profiles of
hyperbolic--parabolic systems,} Arch. Rat. Mech. Anal. \textbf{172} (2004),  no. 1, 93--131.

\bibitem{MTZ} G.\ M\'etivier, T.\ Texier, K.\ Zumbrun, {\em
Existence of quasilinear relaxation shock profiles in systems with characteristic velocities},
Ann. Fac. Sci. Toulouse Math. (6) \textbf{21} (2012), no. 1, 1-23.

\bibitem{MZ} G.\ M\'etivier, K.\ Zumbrun, {\it Existence of semilinear relaxation shocks},
J. Math. Pures Appl. (9) \textbf{92} (2009), no. 3, 209-231.

\bibitem{MZ2}  G.\ M\'etivier, K.\ Zumbrun, {\it Existence and sharp localization in velocity of small-amplitude Boltzmann shocks}, Kinet. Relat. Models \textbf{2} (2009), no. 4, 667-705.

\bibitem{Na} F. Nazarov, private communication.

\bibitem{NPZ} T. Nguyen, R. Plaza, and K. Zumbrun,
	{\it Stability of radiative shock profiles for hyperbolic-elliptic coupled systems,}
 Phys. D  \textbf{239}  (2010),  no. 8, 428--453.

\bibitem{PSS} D.\ Peterhof, B.\ Sandstede, A.\ Scheel, \textit{Exponential dichotomies for solitary-wave solutions of semilinear elliptic
equations on infinite cylinders}, J. Diff. Eq. \textbf{140} (1997) 266--308.


\bibitem{PS1} A.\ Pogan and A.\ Scheel, {\em Instability of Spikes in the Presence of Conservation Laws},  Z. Angew. Math. Phys. \textbf{61} (2010), 979--998.


\bibitem{PS4} A.\ Pogan and A.\ Scheel,  {\em Layers in the Presence of Conservation Laws}, J. Dyn. Diff. Eq., \textbf{24} (2012), 249--287.

\bibitem{PZ} A.\ Pogan and K.\ Zumbrun,  {\em Center manifolds of degenerate evolution equations and existence of small-amplitude kinetic shocks,} 
preprint.

\bibitem{RobSal} J.\ Robbin, D.\ Salamon, \textit{The spectral flow and the Maslov index}, Bull. London Math. Soc. \textbf{27} (1995) 1--33.

\bibitem{BjornSand} B.\ Sandstede, \textit{Stability of traveling waves}, in: Handbook of Dynamical Systems, vol. 2, North-Holland, Amsterdam,
2002, pp. 983--1055.

\bibitem{SS1} B.\ Sandstede, A.\ Scheel, \textit{On the structure of spectra of modulated traveling waves}, Math. Nachr. \textbf{232} (2001) 39--93.

\bibitem{SS2} B.\ Sandstede, A.\ Scheel, \textit{Relative Morse indices, Fredholm indices, and group velocities}, Discrete Contin. Dyn.
Syst. A \textbf{20} (2008) 139--158.

\bibitem{TZ5} B. Texier, K. Zumbrun,
{\it Nash-Moser iteration and singular perturbations},
Ann. Inst. H. Poincare Anal. Non Lineaire \textbf{28} (2011), no. 4, 499-527.

\bibitem{Z2} K. Zumbrun,
{\it Multidimensional stability of planar viscous shock waves,}
Advances in the theory of shock waves, 307--516,
Progr. Nonlinear Differential Equations Appl., 47, Birkh\"auser Boston,
Boston, MA, 2001.

\bibitem{Z3} K. Zumbrun, {\it Stability of large-amplitude shock
waves of compressible Navier--Stokes equations,}
with an appendix by Helge Kristian Jenssen and Gregory Lyng,
in Handbook of mathematical fluid dynamics. Vol. III,  311--533,
North-Holland, Amsterdam, (2004).

\bibitem{Z4} K. Zumbrun, {\it Planar stability criteria for viscous shock waves of systems with real viscosity,}
Hyperbolic systems of balance laws, 229--326, Lecture Notes in Math., 1911, Springer, Berlin, 2007.

\bibitem{Z5} K. Zumbrun,
{\it Stability and dynamics of viscous shock waves,} Nonlinear conservation laws
and applications, 123--167, IMA Vol. Math. Appl., 153, Springer, New York, 2011.


\bibitem{ZH} K. Zumbrun and P. Howard,
{\it Pointwise semigroup methods and stability of viscous shock waves}.
Indiana Mathematics Journal V47 (1998), 741--871;
Errata, Indiana Univ. Math. J.  \textbf{51}  (2002),  no. 4, 1017--1021.

\bibitem{ZS} K. Zumbrun and D. Serre,
{\it Viscous and inviscid stability of multidimensional
planar shock fronts,} Indiana Univ. Math. J. \textbf{48} (1999) 937--992.
\end{thebibliography}
\end{document}